\documentclass[a4paper]{amsart}

\usepackage[utf8]{inputenc}
\usepackage[T1]{fontenc}
\usepackage{lmodern, enumerate}
\usepackage{amssymb,amsxtra,amscd,amsmath}
\usepackage{amsthm}
\usepackage{stmaryrd}
\usepackage{amsfonts}
\usepackage{tikz-cd}
\usepackage{bbm}
\usepackage[all]{xy}
\usepackage{xcolor}
\usepackage{nicefrac,mathtools}
\usepackage[inline]{enumitem}
\usepackage{microtype}
\usepackage{comment}
\usepackage[
  pdftitle={A Baum-Connes assembly map for essential inverse semigroup crossed products},
  pdfauthor={Diego Mart\'{i}nez},
  pdfsubject={Mathematics},
  colorlinks=true, linkcolor=blue, linkbordercolor=blue, citecolor=red, citebordercolor=red, urlcolor=blue, linktocpage=true
]{hyperref}
\usepackage[capitalise, nameinlink, noabbrev, nosort]{cleveref} 
\usepackage[lite]{amsrefs}

\newtheorem{alphthm}{Theorem}     

\newtheorem{alphprop}[alphthm]{Proposition}

\setlist[enumerate]{font=\normalfont}
\crefname{enumi}{}{} 
\crefname{enumi}{}{} 
\setlist[enumerate]{label=(\roman*)}

\numberwithin{equation}{section}
\theoremstyle{plain}
\newtheorem{theorem}[equation]{Theorem}
\newtheorem{lemma}[equation]{Lemma}
\newtheorem{proposition}[equation]{Proposition}
\newtheorem{corollary}[equation]{Corollary}
\newtheorem{claim}[equation]{Claim}
\theoremstyle{definition}
\newtheorem{definition}[equation]{Definition}
\theoremstyle{remark}
\newtheorem{remark}[equation]{Remark}

\newtheorem{notation}[equation]{Notation}
\newtheorem{example}[equation]{Example}
\newtheorem{question}[equation]{Question}

\crefname{theorem}{Theorem}{Theorems}
\crefname{alphthm}{Theorem}{Theorems}
\crefname{proposition}{Proposition}{Propositions}
\crefname{lemma}{Lemma}{Lemmas}
\crefname{claim}{Claim}{Claims}
\crefname{remark}{Remark}{Remarks}
\crefname{conjecture}{Conjecture}{Conjectures}
\crefname{corollary}{Corollary}{Corollaries}
\crefname{question}{Question}{Questions}
\crefname{conjecture}{Conjecture}{Conjectures}
\crefname{fact}{Fact}{Facts}
\crefname{claim}{Claim}{Claims}
\crefname{case}{Case}{Cases}
\crefname{convention}{Convention}{Conventions}

\calclayout

\NewDocumentCommand{\CC}{}{\mathbb{C}}             
             \NewDocumentCommand{\FF}{}{\mathbb{F}}

             \NewDocumentCommand{\NN}{}{\mathbb{N}}
             
             \NewDocumentCommand{\RR}{}{\mathbb{R}}

             \NewDocumentCommand{\ZZ}{}{\mathbb{Z}}

\NewDocumentCommand{\CA}{}{\mathcal{A}}            \NewDocumentCommand{\CB}{}{\mathcal{B}}
\NewDocumentCommand{\CCC}{}{\mathcal{C}}           
\NewDocumentCommand{\CE}{}{\mathcal{E}}            \NewDocumentCommand{\CF}{}{\mathcal{F}}
            \NewDocumentCommand{\CH}{}{\mathcal{H}}
\NewDocumentCommand{\CI}{}{\mathcal{I}}            
\NewDocumentCommand{\CK}{}{\mathcal{K}}            \NewDocumentCommand{\CL}{}{\mathcal{L}}
\NewDocumentCommand{\CM}{}{\mathcal{M}}            \NewDocumentCommand{\CN}{}{\mathcal{N}}

\NewDocumentCommand{\CU}{}{\mathcal{U}}            \NewDocumentCommand{\CV}{}{\mathcal{V}}
            
            \NewDocumentCommand{\CZ}{}{\mathcal{Z}}

\NewDocumentCommand{\fc}{}{\mathfrak{c}}

\NewDocumentCommand{\fs}{}{\mathfrak{s}}

            \NewDocumentCommand{\FP}{}{\mathfrak{P}}

\NewDocumentCommand{\Eth}{}{E}
\NewDocumentCommand{\Egrp}{omm}{%
  \IfNoValueTF{#1}
    {E\left(#2,#3\right)}
    {E_{#1}\left(#2,#3\right)}
}
\NewDocumentCommand{\Etop}{omm}{%
  \IfNoValueTF{#1}
    {E^{\rm top}_{\star}\left(#2; #3\right)}
    {E^{\rm top}_{\star,#1}\left(#2; #3\right)}
}
\NewDocumentCommand{\Kth}{}{K}
\NewDocumentCommand{\KKth}{}{KK}

\NewDocumentCommand{\KKgrp}{omm}{%
  \IfNoValueTF{#1}
    {KK\left(#2,#3\right)}
    {KK_{#1}\left(#2,#3\right)}
}
\NewDocumentCommand{\Kgrp}{om}{%
  \IfNoValueTF{#1}
    {K_\star\left(#2\right)}
    {K_{#1}\left(#2\right)}
}
\NewDocumentCommand{\asymparrow}{}{\dashrightarrow}

\NewDocumentCommand{\actson}{}{\curvearrowright}
\NewDocumentCommand{\variable}{}{-}
\NewDocumentCommand{\sectionsof}{m}{\Gamma\left(#1\right)}

\NewDocumentCommand{\tzero}{}{T\(_0\) }

\NewDocumentCommand{\czeror}{}{\contz{\RR}}
\NewDocumentCommand{\hkindalg}{om}{%
  \IfNoValueTF{#1}
    {\CCC_0\left(#2\right)}
    {\CCC_{0,#1}\left(#2\right)}
}
\NewDocumentCommand{\hkalg}{om}{%
  \IfNoValueTF{#1}
    {\CA\left(#2\right)}
    {\CA_{ {#1} }\left(#2\right)}
}

\NewDocumentCommand{\hkhilb}{om}{%
  \IfNoValueTF{#1}
    {\CH(#2)}
    {\CH_{#1}(#2)}
}

\NewDocumentCommand{\fhalg}{O{h}m}{%
  \IfNoValueTF{#1}
    {\CF\left(#2\right)}
    {\CF_{#1}\left(#2\right)}
}
\NewDocumentCommand{\schwartzfuncs}{om}{%
  \IfNoValueTF{#1}
    {\fs\left(#2\right)}
    {\fs\left(#1,#2\right)}
}
\NewDocumentCommand{\delmaux}{}{\alpha}
\NewDocumentCommand{\ddelmaux}{}{\beta}
\NewDocumentCommand{\delm}{o}{%
  \IfNoValueTF{#1}
    {\delmaux}
    {\delmaux^{#1}}
}
\NewDocumentCommand{\reddelm}{o}{%
  \IfNoValueTF{#1}
    {\delmaux_{\red}}
    {\delmaux^{#1}_{\red}}
}
\NewDocumentCommand{\maxdelm}{o}{%
  \IfNoValueTF{#1}
    {\delmaux_{{\rm max}}}
    {\delmaux^{#1}_{{\rm max}}}
}
\NewDocumentCommand{\ddelm}{o}{%
  \IfNoValueTF{#1}
    {\ddelmaux}
    {\ddelmaux^{#1}}
}
\NewDocumentCommand{\redddelm}{o}{%
  \IfNoValueTF{#1}
    {\ddelmaux_{\red}}
    {\ddelmaux^{#1}_{\red}}
}
\NewDocumentCommand{\maxddelm}{o}{%
  \IfNoValueTF{#1}
    {\ddelmaux_{{\rm max}}}
    {\ddelmaux^{#1}_{{\rm max}}}
}

\NewDocumentCommand{\crossedprod}{omm}{%
  \IfNoValueTF{#1}
    {#2 \rtimes #3}
    {#2 \rtimes_{#1} #3}
}
\NewDocumentCommand{\algcrossedprod}{omm}{%
  \IfNoValueTF{#1}
    {#2 \rtimes_{{\rm alg}} #3}
    {#2 \rtimes^{#1}_{{\rm alg}} #3}
}
\NewDocumentCommand{\redcrossedprod}{omm}{%
  \IfNoValueTF{#1}
    {#2 \rtimes_{\red} #3}
    {#2 \rtimes^{#1}_{\red} #3}
}
\NewDocumentCommand{\maxcrossedprod}{omm}{%
  \IfNoValueTF{#1}
    {#2 \rtimes_{\full} #3}
    {#2 \rtimes^{#1}_{\full} #3}
}

\NewDocumentCommand{\cstar}{}{\texorpdfstring{\(C^*\)\nobreakdash-\hspace{0pt}}{*-}}
\NewDocumentCommand{\Star}{}{\texorpdfstring{\(^*\)\nobreakdash-\hspace{0pt}}{*-}}

\NewDocumentCommand{\range}{o}{%
  \IfNoValueTF{#1}
    { {\rm rng} }
    { {\rm rng} \left(#1\right) }
}
\NewDocumentCommand{\source}{o}{%
  \IfNoValueTF{#1}
    { {\rm source} }
    { {\rm source} \left(#1\right) }
}

\NewDocumentCommand{\supp}{m}{\text{supp}\left(#1\right)}
\NewDocumentCommand{\norm}{m}{\left\|#1\right\|}

\NewDocumentCommand{\abs}{m}{\left|#1\right|}
\NewDocumentCommand{\innerprod}{mm}{\langle #1, #2 \rangle}
\NewDocumentCommand{\linnerprod}{omm}{%
  \IfNoValueTF{#1}
    {_{\rm L}\innerprod{#2}{#3}}
    {_{#1}\innerprod{#2}{#3}}
}
\NewDocumentCommand{\rinnerprod}{omm}{%
  \IfNoValueTF{#1}
    {\innerprod{#2}{#3}_{\rm R}}
    {\innerprod{#2}{#3}_{#1}}
}
\NewDocumentCommand{\closure}{m}{\overline{#1}}
\NewDocumentCommand{\multialg}{m}{\CM\left(#1\right)}
\NewDocumentCommand{\unitof}{m}{\widetilde{#1}}

\NewDocumentCommand{\cpc}{}{c.c.p.\ }

\NewDocumentCommand{\spn}{m}{{\rm span}\left(#1\right)}
\NewDocumentCommand{\spnc}{m}{\overline{{\rm span}}\left\{#1\right\}}

\NewDocumentCommand{\borelalg}{m}{\widetilde{#1}}

\NewDocumentCommand{\cont}{m}{C\left(#1\right)}
\NewDocumentCommand{\contc}{m}{C_c\left(#1\right)}
\NewDocumentCommand{\contz}{m}{C_0\left(#1\right)}
\NewDocumentCommand{\contb}{m}{C_b\left(#1\right)}

\NewDocumentCommand{\cfunz}{m}{\CCC_0\left(#1\right)}
\NewDocumentCommand{\red}{}{\rm red}
\NewDocumentCommand{\full}{}{\rm max}
\NewDocumentCommand{\ess}{}{\rm ess}
\NewDocumentCommand{\alg}{}{\rm alg}
\NewDocumentCommand{\redalg}{om}{%
  \IfNoValueTF{#1}
    {C^*_{\red}\left(#2\right)}
    {C^*_{\red,#1}\left(#2\right)}
}
\NewDocumentCommand{\redalgnopar}{om}{%
  \IfNoValueTF{#1}
    {C^*_{\red}(#2)}
    {C^*_{\red,#1}(#2)}
}
\NewDocumentCommand{\maxalg}{om}{%
  \IfNoValueTF{#1}
    {C^*_{\full}\left(#2\right)}
    {C^*_{\full,#1}\left(#2\right)}
}             
\NewDocumentCommand{\maxalgnopar}{om}{%
  \IfNoValueTF{#1}
    {C^*_{\full}(#2)}
    {C^*_{\full,#1}(#2)}
}
\NewDocumentCommand{\essalg}{m}{C^*_{\ess}\left(#1\right)}
\NewDocumentCommand{\essmaxalg}{m}{C^*_{\ess,\full}\left(#1\right)}
\NewDocumentCommand{\algalg}{om}{%
  \IfNoValueTF{#1}
    {\CC_{\alg}\left(#2\right)}
    {\CC_{\alg,#1}\left(#2\right)}
}
\NewDocumentCommand{\algalgnopar}{om}{%
  \IfNoValueTF{#1}
    {\CC_{\alg}(#2)}
    {\CC_{\alg,#1}(#2)}
}

\NewDocumentCommand{\singideal}{om}{%
  \IfNoValueTF{#1}
    {J_{#2}}
    {J_{#2}^{#1}}
}
\NewDocumentCommand{\condexp}{o}{%
  \IfNoValueTF{#1}
    {E}
    {E_{#1}}
}
\NewDocumentCommand{\redcondexp}{o}{%
  \IfNoValueTF{#1}
    {P}
    {P_{#1}}
}
\NewDocumentCommand{\esscondexp}{o}{%
  \IfNoValueTF{#1}
    {EL}
    {EL_{#1}}
}

\NewDocumentCommand{\tensor}{}{\otimes}

\NewDocumentCommand{\id}{o}{%
  \IfNoValueTF{#1}
    {{\rm id}}
    {{\rm id}_{#1}}
}

\NewDocumentCommand{\primidealspc}{m}{
  {\rm Prim}\left(#1\right)
}

\NewDocumentCommand{\idealsin}{m}{
  {\rm Ideals}\left(#1\right)
}

\NewDocumentCommand{\opensets}{m}{
  {\rm Open}\left(#1\right)
}

\begin{document}

\title[A Baum-Connes map for essential crossed products]{A Baum-Connes assembly map for essential semigroup crossed products}
\author[Diego Mart\'{i}nez]{Diego Mart\'{i}nez \(^{1}\)}
\dedicatory{A mi sobrina}
\address{Department of Mathematics, KU Leuven, Celestijnenlaan 200B, 3001 Leuven, Belgium.}
\email{diego.martinez@kuleuven.be}

\begin{abstract}
  We construct an equivariant \Eth-theory and a Baum-Connes assembly map at the level of Fell bundles of inverse semigroups over separable \cstar{}algebras.
  This generalizes previous work of several authors, and allows to discuss \Eth-theoretic matters in the context of Cartan pairs; maximal and essential \cstar{}algebras of non-Hausdorff groupoids; and Fell bundles over discrete groups and \'etale groupoids, among others.
  In order to do this we establish several functoriality properties for maximal, reduced and essential cross-sectional \cstar{}algebras associated with a (saturated) Fell bundle of an inverse semigroup.
  This allows to discuss when these algebras give rise to short exact sequences, generalizing the classical case of discrete groups.
  We also introduce the adequate notion of ``proper'' Fell bundle, or ``proper'' action of an inverse semigroup, and prove a weak containment property for these.
  Using these functoriality properties and these proper actions we then introduce (maximal, reduced and/or essential) equivariant \Eth-theory by means of adequately equivariant asymptotic morphisms, and construct a Baum-Connes assembly map that is both natural and reasonably well-behaved.
\end{abstract}



\thanks{{\(^{1}\)} Supported by projects G085020N and 1218726N funded by the Research Foundation Flanders (FWO)}

\maketitle

\section{Introduction} \label{sec:intro}

When one approaches the classification of (a certain class of) \cstar{}algebras via \Kth-theoretic means one immediately runs into the necessity to construct a large class of morphisms between \cstar{}algebras.
Typically speaking, the first class of morphisms \(\varphi \colon A \to B\) one tackles is that of \Star{}homomorphisms.
This notwithstanding, this class is, unfortunately, too small: there are in general many morphisms of abelian groups \(\Kgrp{A} \to \Kgrp{B}\) that cannot be obtained from \Star{}homomorphisms \(A \to B\).
A collateral consequence of Kasparov's phenomenal work \cites{Kasparov-1988,higson-kasparov-2000} introducing the (abelian) \KKth-groups \(\KKgrp{A}{B}\) solves this problem: the elements of \(\KKgrp{A}{B}\) are homotopy classes of ``generalized morphisms'' \(\varphi \colon A \asymparrow B\).
The original definition of this group, however, is technical and somewhat opaque to new readers, which is one of reasons Connes and Higson, along with Guentner and Trout~\cite{Guentner-Higson-Trout}, introduced the \Eth-theory groups \(\Egrp{A}{B}\).
Elements of this group are, by definition, homotopy classes of \emph{asymptotic morphisms} \(\varphi \colon \czeror \tensor A \tensor \CK \asymparrow \czeror \tensor B \tensor \CK\), that is, tuples \(\varphi = (\varphi_r)_{r \in [1,\infty)}\) of set-theoretic maps \(\varphi_r \colon A \to B\) that satisfy the \Star{}homomorphism conditions when \(r \to \infty\).

\KKth-theory and \Eth-theory share quite a lot of properties.
For instance, both are universal cohomology theories in their own way.
Moreover, both are covariant in the first variable and contravariant in the second; and any extension \(0 \to B \to E \to A \to 0\) canonically defines an element in \(\KKgrp{\czeror \tensor A}{B}\) and \(\Egrp{A}{B}\).
Furthermore, both can be defined even when the \cstar{}algebras \(A\) and \(B\) are equipped with an action of a, say discrete, group \(\Gamma \actson A, B\).
In this setting, there is a so-called ``descent functor''
\[
  {\rm descent} \colon \Egrp[\Gamma]{A}{B} \to \Egrp{A \rtimes_{\full} \Gamma}{B \rtimes_{\full} \Gamma}
\]
that is a morphism of abelian groups. This arrow sends the class of an equivariant asymptotic morphism \(\varphi \colon A \asymparrow B\) to the class of the induced asymptotic morphism \(\varphi \rtimes_{\full} \Gamma \colon A \rtimes_{\full} \Gamma \asymparrow B \rtimes_{\full} \Gamma\).
Lastly, both \Eth-theory and \KKth-theory enjoy a ``composition product'' (also called \emph{Kasparov} product), which allows to essentially compose elements:
\[
  \Egrp{A}{B} \times \Egrp{B}{C} \to \Egrp{A}{C}, \;\; \left(\left[\varphi\right]_h, \left[\psi\right]_h\right) \mapsto \left[\psi \circ \varphi\right]_h.
\]
This, after a lot of work and several key ideas~\cites{higson-kasparov-2000,higson-kasparov-trout-1998,Guentner-Higson-Trout,Kasparov-1988,tu-baum-connes-2004,tu-baum-connes-1999,TikuisisAaron2017QonC}, means that both these theories share two features of utmost importance in the study of \cstar{}algebras.
First, both can be used to define a so-called ``Universal Coefficient Theorem'' short exact sequence, which is typically written\footnote{\, The existence of this short exact sequence is a consequence of certain compatibility features of the maps \(\gamma\) and \(\kappa\) appearing in it. The reader should \emph{not} read this is as saying that the sequence always exists, regardless of \(A\) and \(B\).} as
\begin{equation} \label{eq:intro-uct}
  0 \to {\rm Ext}_\ZZ^1\left(K_\star\left(A\right), K_\star\left(B\right)\right) \xrightarrow{\kappa^{-1}} \Egrp{A}{B} \xrightarrow{\gamma} {\rm Hom}\left(K_\star\left(A\right), K_{\star+1}\left(B\right)\right) \to 0.
\end{equation}
Likewise, both can be used to define a ``Baum-Connes assembly map'', see \cites{yu_localization_1997,baum_classifying_1994-1,aparicio_baum-connes_2019,yu_coarse_2000,Guentner-Higson-Trout,higson-lafforgue-skand-2002-counter-bc,higson-kasparov-2000}:
\begin{equation} \label{eq:intro-baum-connes}
  \mu_{A, \Gamma} \colon \Egrp[\Gamma]{\CE \Gamma}{A} \to \Kgrp{A \rtimes_{\red} \Gamma}.
\end{equation}
Both \eqref{eq:intro-uct} and \eqref{eq:intro-baum-connes} are of fundamental importance. Indeed, one of the most important open problems in nuclear \cstar{}theory is whether \eqref{eq:intro-uct} is short exact when \(A\) is separable and nuclear, see, \emph{e.g.}\ \cite{blackadar-book-k-theory}*{23.15.12}, \cites{rosenberg-schochet-1987-uct,TikuisisAaron2017QonC,willett-yu-memoirs-2024}, and \cite{schafhauser-tikuisis-white-problems}*{Problem II}.
For the Baum-Connes assembly map, one of the most important open problems is whether it is an isomorphism of abelian groups for \emph{all} discrete groups \(\Gamma\), at least when \(A = \CC\), see \cite{aparicio_baum-connes_2019}.
This would provide an effective way to compute the \Kth-groups \(\Kgrp{A \rtimes_{\red} \Gamma}\).

Nevertheless, even more is true: the UCT short exact sequence and the Baum-Connes assembly map are inextricably intertwined.
The latter \eqref{eq:intro-baum-connes} was introduced by Baum and Connes as a generalization of the Atiyah-Singer theorem, and uses the usual \Kth-group \(\Kgrp{A \rtimes_{\red} \Gamma}\) as a receptacle of the index of some elliptic operators over \(A\) (this is the left hand side \(\Egrp[\Gamma]{\CE \Gamma}{A}\), for a more comprehensive discussion we refer to \cref{sec:equiv-eth}).
The map \(\gamma\) in the former \eqref{eq:intro-uct} can also be seen as an ``index'' map for the generalized morphism \([\varphi]_h \in \Egrp{A}{B}\), and the map \(\kappa \colon \ker(\gamma) \to {\rm Ext}_\ZZ^1(K_\star(A), K_\star(B))\) measures the deviation of \(\gamma\) being injective.
Thus, both \eqref{eq:intro-uct} and \eqref{eq:intro-baum-connes} discuss how much a certain index map fails to be an isomorphism.
Whence understanding \eqref{eq:intro-uct} and \eqref{eq:intro-baum-connes} in the most general setting possible is also of great importance.

\vspace{1.5mm}

In this paper we introduce the setting of \emph{equivariant \Eth-theory} for \emph{Fell bundles of inverse semigroups}.
This allows to discuss the short exact sequences \eqref{eq:intro-uct} and \eqref{eq:intro-baum-connes} for a large class of \cstar{}algebras: those arising as the maximal \(\maxalg{\CB}\) and the reduced \(\redalg{\CB}\), for a Fell bundle \(\CB = (B_s)_{s \in S}\) over an inverse semigroup \(S\).
The \emph{inverse semigroup} \(S\) should be regarded as the acting component: it generalizes the discrete group \(\Gamma\) appearing before.
Indeed, inverse semigroups allow for only locally defined symmetries, as opposed to the globally defined ones that (typically) group actions describe.
For instance, it is by now well known, and mostly due to the work of Exel~\cite{Exel:Inverse_combinatorial} and others \cites{exel-starling-2017,exel-pitts-grpds-2022,buss-exel-twisted-acts-2011}, that inverse semigroup actions on locally compact Hausdorff spaces are in correspondence with locally compact, \'etale groupoids \(G\) whose unit space is Hausdorff.
This, ultimately, means that our definition of \(\Egrp[S]{\CA}{\CB}\) will immediately yield a well-defined \Eth-theory for non-necessarily Hausdorff groupoids.
Notice that this follows a recent shift of interest in the community: from the Hausdorff case \cites{SimsNotes2020,Kranz:weak-containment,Kumjian:Fell_bundles,Exel:Inverse_combinatorial,Khoshkam-Skandalis-reg-rep,li-classification-2020}, to the study of non-Hausdorff groupoids, see \emph{e.g.}\ \cites{neshveyevetal2023,CEPSS-2019,szakacs-2021,szakacs-2023,exel-starling-2017,bkmk-2024-twited-grpds,timmermann_2011,bruce-li-2024,hume-2025,brix-gonzalez-hume-li-2025,gonzales-hume-2025,KMP2025,miller-steinberg-2025,miller-scarparo-2025}.

Even more so is true: it follows from the work of Exel-Buss-Meyer \cites{buss-exel-twisted-acts-2011,BussExel:Fell.Bundle.and.Twisted.Groupoids,Buss-Exel-Meyer:Reduced} that these Fell bundles generalize even the notion of a \emph{Cartan subalgebra}~\cites{Renault2008CartanSI,kumjian-diagonals-1986,barlak-li-2017,white_cartan_2018,bclmc24}, as they allow for \emph{twists}.
Hence, our approach also immediately yields a well-adjusted machinery to study \Eth-groups for (\emph{weak} and/or \emph{non-commutative}) Cartan pairs, for instance. 
We refer the reader to \cref{sec:apps} for a more detailed discussion about these topics.

\vspace{1.5mm}

\noindent \emph{Remark.}
  There has already been attempts to define an inverse semigroup equivariant \KKth-theory \cite{burgstaller-2016} and \Eth-theory for groupoids \cites{popescu-2004,kwanieski-li-skalski-2022}.
  Likewise, a Baum-Connes map has been defined for groupoids arising from singular foliations \cite{androulidakis-skandalis-2019}.
  Nevertheless, none of these prove all the necessary ingredients to properly define the equivariant category, or only work for Hausdorff groupoids, such as \cites{popescu-2004,kwanieski-li-skalski-2022,burgstaller-2016}.
  Moreover, none of these treat twists, let alone Fell bundles.
  Lastly, we briefly note in passing that we work exclusively with equivariant \Eth-theory, and this is \emph{not} due to a personal choice: see \cref{rem:why-not-kk}.

\vspace{1.5mm}

In order to introduce our attempt at inverse semigroup equivariant \Eth-theory one first needs to introduce a notion of \emph{equivariant asymptotic morphism} \(\varphi \colon \CA \asymparrow \CB\) between Fell bundles \(\CA = (A_s)_{s \in S}\) and \(\CB = (B_s)_{s \in S}\).
Such a morphism, of course, should yield an asymptotic morphism \(\maxalg{\varphi} \colon \maxalg{\CA} \asymparrow \maxalg{\CB}\), which then yields a ``descent functor''
\begin{align*}
  {\rm descent} \colon \Egrp[S]{\CA}{\CB} & \to \Egrp{\maxalg{\CA}}{\maxalg{\CB}}, \\
  \left[\varphi\right]_h & \mapsto \left[\maxalg{\varphi}\right]_h
\end{align*}
generalizing the one alluded to before.

Nevertheless, several hindrances here now appear.
The main one being that, by work of Buss and the author \cite{buss-martinez-ess-ame-2024}, and based on ideas of others \cites{Kwasniewski-Meyer:Pure_infiniteness,KwasniewskiMeyer2021}, there is \emph{another} possible choice of maximal \cstar{}algebra.
Indeed, one may choose to ignore, or ban, certain points \(D \subseteq G^{(0)}\) of the unit space of a groupoid when constructing the \cstar{}algebra.
The reason why one would do this is because this yields the \emph{essential} \cstar{}algebra \(\essalg{G}\) of \cites{exel-pitts-grpds-2022,Kwasniewski-Meyer:Pure_infiniteness,KwasniewskiMeyer2021}, which is now regarded to be the one with the ideal ideal structure.
This has to do with the aforementioned (non-)Hausdorffness of \(G\), and requires us to first shift the point of view.
Indeed, it is unclear what an \emph{essentially equivariant} asymptotic morphism \(\varphi \colon \CA \asymparrow \CB\) should be, for the definition of \(\essalg{G}\) is somewhat technical.
The first result we thus mention is an alternative description of \(\essalg{\CB}\), based on ideas in \cite{buss-martinez-ess-ame-2024}.

\begin{alphthm} [cf.\ \cref{thm:ess-alg-is-jd}] \label{thm-intro-ess}
  If \(\CB = (B_s)_{s \in S}\) is a Fell bundle over some type I \cstar{}algebra \(B\) then there is a choice of \(D \subseteq \primidealspc{B}\) such that the identity map \(\id \colon \maxalg{\CB} \to \maxalg{\CB}\) descends to a \Star{}isomorphism between \(\essalg{\CB}\), as in \cites{Kwasniewski-Meyer:Pure_infiniteness,KwasniewskiMeyer2021,exel-pitts-grpds-2022}, and \(\redalg[D]{\CB} \coloneqq \redalg{\CB}/\CN_{D}\), where \(\CN_{D}\) is the nucleus of the canonical conditional expectation \(\redalg{\CB} \to B^{**} \to B^{**}/J_{D}\).
\end{alphthm}

\cref{thm-intro-ess} generalizes \cite{buss-martinez-ess-ame-2024}*{Proposition 3.19} to the setting of Fell bundles over type I \cstar{}algebras.
From now on, given a choice of \(D_0 \subseteq \primidealspc{B}\) we write \(\maxalg[D_0]{\CB}\) and \(\redalg[D_0]{\CB}\) for the ``maximal'' and ``reduced'' \(D_0\)-completions of \(\CB\) (for a proper definition see \cref{def:cross-algebra}).
Morally speaking, \cref{thm-intro-ess} says there is a choice of points \(D_0 \subseteq \primidealspc{B}\) such that an \emph{essentially equivariant asymptotic morphism} \(\varphi \colon \CA \asymparrow \CB\) is a tuple of maps that satisfy some asymptotic morphism conditions and preserve some extra structure, namely some ideal \(J_{D_0}\) .
Following these ideas one obtains the following.

\begin{alphthm} [cf.\ \cref{thm:maximal-extensions,thm:reduced-extensions}] \label{thm-intro-func}
  A short exact sequence \(0 \to \CI \to \CA \to \CB \to 0\) that is \(D_0\)-equivariant (in the sense of \cref{def:all-equiv-morphism}) induces an extension
  \[
    0 \to \maxalg[D_0]{\CI} \to \maxalg[D_0]{\CA} \to \maxalg[D_0]{\CB} \to 0.
  \]
  The reduced analogue also exists; however, it can fail to be exact in the middle (and only in the middle).
\end{alphthm}

Likewise, we also prove that \(D_0\)-equivariant asymptotic morphisms \(\varphi \colon \CA \asymparrow \CB\) do induce asymptotic morphisms \(\maxalg[D_0]{\varphi} \colon \maxalg[D_0]{\CA} \asymparrow \maxalg[D_0]{\CB}\), and similarly for \(\redalg[D_0]{\variable}\) at least when they are ``\emph{reduced}'' in a certain way, see \cref{prop:inject-equiv-map-induces,def:all-equiv-morphism}.
This essentially provides the necessary mindset in order to define a \(D_0\)-equivariant \Eth-theory \(\Egrp[S,D_0]{\CA}{\CB}\), and it also produces the descent functor
\begin{align*}
  {\rm descent}_{D_0} \colon \Egrp[S,D_0]{\CA}{\CB} & \to \Egrp{\maxalg[D_0]{\CA}}{\maxalg[D_0]{\CB}}, \\
  \left[\varphi\right]_h & \mapsto \left[\maxalg[D_0]{\varphi}\right]_h
\end{align*}
Nevertheless, the assembly map is still far, for in order to define it one needs to discuss \emph{proper} actions \(S \actson X\), or \emph{proper} Fell bundles.
This is a notion that is certainly necessary when discussing Baum-Connes matters, see \cites{valette-2002-book,aparicio_baum-connes_2019,higson-kasparov-2000}.
For us, the main results regarding this are the following.

\begin{alphprop} \label{prop-intro-1}
  If \(S \actson X\) is a \(D_0\)-proper action on a locally compact, Hausdorff space \(X\), then there is a unique (up to homotopy) projection \(p_{c} \in \maxalg[D_0]{S \actson X}\).
\end{alphprop}

\begin{alphthm}[cf.\ \cref{lemma:proper-actions-give-weak-containment}] \label{prop-intro-2}
  Let \(\CB = (B_s)_{s \in S}\) be a Fell bundle over an \((S,X)\)-algebra \(B\).
  Assume \(S \actson X\) is \(D_0\)-proper.
  The canonical quotient map \(\maxalg[D_0]{\CB} \to \redalg[D_0]{\CB}\) is then a \Star{}isomorphism.
\end{alphthm}

These results give structural properties of \(\maxalg[D_0]{\CB}\) for proper Fell bundles.
Moreover, following the ideas previously discussed, one may define \(\Egrp[S, D_0]{\CA}{\CB}\) to be the set of homotopy classes of \(D_0\)-equivariant asymptotic morphisms \(\czeror \tensor \CA \tensor \CK \asymparrow \czeror \tensor \CB \tensor \CK\) (for some algebra of compact operators \(\CK\)).
This precisely generalizes the group case to our setting.
It also gives the following, whose phrasing we keep vague on purpose: we refer the reader to \cref{sec:equiv-eth} for precise statements. The reader should read the following as saying ``\(\Egrp[S,D_0]{\variable}{\variable}\) is well behaved''.

\begin{alphthm} \label{thm:intro-eth}
  \(\Egrp[S,D_0]{\CA}{\CB}\) is an abelian group; and it defines a category admitting a Kasparov product.
  Moreover, any \(D_0\)-equivariant extension \(0 \to \CB \to \CE \to \CA \to 0\) defines uniquely an element of \(\Egrp[S,D_0]{\CA}{\CB}\). 
\end{alphthm}

It should be mentioned that if \(S = \Gamma\) is a group then \(\Egrp[\Gamma,\emptyset]{\CA}{\CB}\) generalizes the classical group equivariant \Eth-theory of Guetner, Higson and Trout, see \cite{Guentner-Higson-Trout}, to the Fell bundle case.
Likewise, if \(S = \opensets{X}\) is the meet semi-lattice of open sets of a locally compact, Hausdorff \(X\) then \(\Egrp[\opensets{X},\emptyset]{\CA}{\CB}\) generalizes the \(X\)-equivariant \Eth-theory of Dadarlat and Meyer, see \cite{dadarlat-meyer-2012}.
This, in turn, should be compared with the \(X\)-equivariant \KKth-theory of Gabe in \cite{gabe-memoirs-2024}.

Furthermore, our point of view allows to define the left hand side of the Baum-Connes assembly map \(\Etop{S, D_0}{\CB}\) as the group generated by \(\Egrp[S,D_0]{\contz{Y}}{\CB}\) for proper, \(S\)-compact, \(S\)-spaces, quotiented by an appropriate equivalence relation (see \cref{thm:assembly} for details).
These ideas, after some work, yield the following.

\begin{alphthm}[cf.\ \cref{thm:assembly}] \label{thm:intro-assembly}
  The map
  \(
    \mu_{D_0, \CB} \colon \Etop{S, D_0}{\CB} \to \Kgrp{\redalg[D_0]{\CB}}
  \)
  locally defined as
  \begin{align*}
    \Egrp[S,D_0]{\contz{Y}}{\CB} & \xrightarrow{\text{descent}} \Egrp{\maxalg[D_0]{S \ltimes Y}}{\maxalg[D_0]{\CB}} \\
    & \xrightarrow{[p_{c_Y}]_0 \cdot} \Kgrp[0]{\maxalg[D_0]{\CB}} \\
    & \xrightarrow{\text{reg. rep.}} \Kgrp[0]{\redalg[D_0]{\CB}}
  \end{align*}
  is a homomorphism of abelian groups.
  Moreover, should \(S\) be a group and \(\CB\) be induced from an action \(S \actson B\) then \(\mu_{\emptyset, \CB}\) is the classical group Baum-Connes assembly map of \cite{Guentner-Higson-Trout}.
\end{alphthm}

\vspace{1.5mm}

\noindent \emph{Remark.}
  Notice that we have \emph{not} discussed the left hand side of \eqref{eq:intro-baum-connes}, for it requires us to discuss \(\CE \Gamma\), the universal proper \(\Gamma\) space, see \cite{aparicio_baum-connes_2019}.
  In the setting of inverse semigroup actions it is, unfortunately, not clear whether such an object exists at all.
  Whence we have to define the assembly map as in \cref{thm:intro-assembly}.

\vspace{1.5mm}

A natural consequence of \cref{thm:intro-assembly} is the explicit construction of a commutative diagram
\begin{equation*}
  \begin{tikzcd}[scale=50em]
    E^{\rm top}_{\star, \full}\left(G\right) \arrow{r}{\mu_{\full, G}} & \Kgrp{\maxalg{G}} \arrow{r}{} \arrow{d}{} & \Kgrp{\redalg{G}} \arrow{d}{} \\
    E^{\rm top}_{\star, \ess}\left(G\right) \arrow{u}{} \arrow{r}{\mu_{\ess, G}} & \Kgrp{\essmaxalg{G}} \arrow{r}{} & \Kgrp{\essalg{G}}
  \end{tikzcd}
\end{equation*}
for any \'etale groupoid \(G\).
Whether or not \(\mu_{\full, G}\) and \(\mu_{\ess, G}\) are isomorphisms will be studied elsewhere, but they cannot be expected to be neither injective nor surjective in full generality~\cites{higson-lafforgue-skand-2002-counter-bc,martinez-szakacs-2025}.

\vspace{1.5mm}

The paper is organized as follows.
\cref{sec:crossedprod} introduces the setting within which we will work, and constructs the algebras \(\maxalg[D_0]{\CB}\) and \(\redalg[D_0]{\CB}\).
\cref{sec:apps-ess} proves \cref{thm-intro-ess}. In \cref{sec:functoriality} we prove \cref{thm-intro-func}, and study the functoriality properties one needs when considering asymptotic morphisms of Fell bundles.
The slightly ad-hoc \cref{sec:proper-ame} introduces some technical notions, such as ``\(D_0\)-amenability'', that are needed for \cref{sec:proper}, where we introduce the notion of a ``proper'' Fell bundle, and prove \cref{prop-intro-1,prop-intro-2}.
In \cref{sec:equiv-eth} we define the equivariant \Eth-theory, and prove \cref{thm:intro-eth,thm:intro-assembly}.
Lastly, \cref{sec:apps} contains some discussions about the applications of these ideas to group actions, groupoids, Cartan pairs, and essential \cstar{}algebras.

\vspace{1.5mm}

\noindent \textbf{Acknowledgements:} we thank Alistair Miller for fruitful conversations on topics related to the text.

\section{Dynamical setup and its crossed products} \label{sec:crossedprod}
In this section we introduce the actions and spaces/algebras we study.
These notions are clearly inspired by previous literature, see \emph{e.g.}\ \cites{buss-meyer-2017-actions,buss-martinez-approx-prop-23,BussExelMeyer2017,buss-exel-twisted-acts-2011,Exel:Inverse_combinatorial,KwasniewskiMeyer2021}.
However, we will generalize most of the previous work, so it is best we construct all objects ``by hand''.

Recall that a \tzero space \(X\) is \emph{locally compact} if every point in \(X\) has a neighborhood base of compact sets (compacts subsets that contain \(x\) in their interior and intersect to the singleton \(\{x\}\)).
For all intents and purposes the reader is welcome to think of \(X\) as \(\primidealspc{B}\), which will be the primitive ideal space of a separable \cstar{}algebra \(B\).
We refer to \cite{blackadar-book-2010}*{Corollary II.6.5.7} for a proof that \(\primidealspc{B}\) is always (\emph{i.e.}\ regardless of whether \(B\) is separable) \tzero and locally compact.
Moreover, should \(B\) be separable (resp.\ unital) then \(\primidealspc{B}\) is second countable (resp.\ compact), as shown in \cite{blackadar-book-2010}*{Corollary II.6.5.7}.
We point out the following, which is known to experts (and will be used virtually without mention in the sequel).

\begin{lemma} \label{lemma:sober}
  Any locally compact, Baire, \tzero space \(X\) is sober, that is, every irreducible closed \(C \subseteq X\) is \(C = \closure{\{x\}}\) for exactly one \(x \in X\).
  In particular \(\primidealspc{B}\) is sober for any \cstar{}algebra \(B\).
\end{lemma}
\begin{proof}
  For the fact that \(X\) is sober see \cite{harnisch-kirchberg-prim} (note that \cite{harnisch-kirchberg-prim} uses the term ``point-complete'' instead of sober).
  Moreover, for the fact that \(\primidealspc{B}\) is Baire see \cite{blackadar-book-2010}*{Theorem~II.6.5.14}.
\end{proof}

Let, for a moment, \(X\) be locally compact and Hausdorff.
Traditionally Kasparov \cite{Kasparov-1988} defined an \emph{\(X\)-algebra} \(B\) as a \cstar{}algebra equipped with some non-degenerate \emph{structure map} \(\pi \colon \contz{X} \to \CZ(\multialg{B})\), where the latter is the center of the multiplier algebra of \(B\).
This definition, however, does not really work should one be interested in non-Hausdorff spaces, for \(\contz{X}\) does not allow for sufficiently many discontinuities.
This approach was later recovered, changed and expanded in several works, see \emph{e.g.}\ \cites{meyer-nest-2009,gabe-memoirs-2024} and references therein.
We particularly follow Gabe's path in \cite{gabe-memoirs-2024}.

\begin{definition}[see \cite{gabe-memoirs-2024}*{Definition 10.7}] \label{def:x-algebra}
  Let \(X\) be sober. An \emph{\(X\)-algebra} is a \cstar{}algebra \(B\) equipped with an order preserving \emph{structure map} \(\Phi_B \colon \opensets{X} \to \idealsin{B}\) from the lattice of open sets of \(X\) to the lattice of ideals in \(B\).
  The \(X\)-structure is \emph{continuous} if \(\Phi_B\) preserves arbitrary suprema and infima.
\end{definition}

In all examples of interest for us we will also have that \(\Phi_B(\emptyset) = 0\) and \(\Phi_B(X) = B\).
Note that these would be automatic should the \(X\)-structure be continuous.
Indeed, recalling that \(\inf \emptyset\) is the largest element of a partially ordered set, namely \(X\) or \(B\) in our cases, we have that \(\Phi_B(X) = \Phi_B(\inf \emptyset) = \inf \Phi_B(\emptyset) = \inf \emptyset = B\), where the second equality is due to the preservation of infima.
One similarly shows that \(\Phi_B(\emptyset) = 0\) whenever \(\Phi_B\) preserves suprema.

\begin{example} \label{ex:n-structure}
  As pointed out in \cite{gabe-memoirs-2024}*{Example 10.10}, any \cstar{}algebra \(B\) is canonically a \(\primidealspc{B}\)-algebra: we will revisit this in \cref{prop:setting-actions} and the subsequent \cref{def:s-x-algebra}.
  Nevertheless, \cref{def:x-algebra} is much more flexible.
  For instance, any choice of countably many ideals \(J_1, J_2, \dots \subseteq B\) equips \(B\) with a canonical \(\NN\)-algebra structure via \(\Phi_B(\{n\}) \coloneqq J_n\) and, similarly, \(\Phi_B(A) \coloneqq \closure{\sum_{n \in A} J_n}\) for all \(A \subseteq \NN\).
\end{example}
\begin{example} \label{example:cfunz}
  On the other end of the spectrum, consider:
  \begin{equation} \label{eq:def-cfunz}
    \cfunz{X} \coloneqq \spnc{f_1 \cdots f_k \; : \; f_i \colon \CU_i \to \CC \text{ and } \CU_i \subseteq X \text{ is open, and } f_i \in \contc{\CU_i}} \subseteq \ell^\infty\left(X\right).
  \end{equation}
  This is clearly an \(X\)-algebra in a natural manner, as one can regard \(\cfunz{X}\) as a \cstar{}sub-algebra of \(\ell^\infty(X)\).
  We also notice in passing that if \(X\) is \tzero then \(\cfunz{X}\) separates the points of \(X\), and if \(X\) is Hausdorff then \(\cfunz{X}\) is just the usual \(\contz{X}\). 
  Moreover, if \(f = f_1 \cdots f_k \in \cfunz{X}\) is a monomial then the map \(X \ni x \mapsto \abs{f(x)} \in \RR_+\) is lower semicontinuous, which will be (emotionally) relevant in \cref{rem:non-hausdorffness} and subsequent work. 
\end{example}

We highlight that, as the topological spaces we consider are all sober the approaches of \cites{meyer-nest-2009,gabe-memoirs-2024} are closely related.
This is proven in \cite{gabe-memoirs-2024}*{Example 10.14}: if \(\Phi_B\) above is \emph{upper semicontinuous} and \emph{finitely lower semicontinuous} then there is a dual continuous map \(\psi_B \colon \primidealspc{B} \to X\).
The correspondence between \(\psi_B\) and \(\Phi_B\) is 
\[
    \Phi_B\left(\CU\right) \coloneqq \bigcap \left\{P \in \primidealspc{B} : P \not\in \psi_B^{-1}\left(\CU\right) \right\}
\]
for all open sets \(\CU \subseteq X\), see \cite{gabe-memoirs-2024}*{Example 10.14}.
Given an \(X\)-algebra \(B\) we will often drop the map \(\Phi_B\) from the data (or assume it is implicitly defined).
Given \(x \in X\) we let \(B_x \coloneqq B/\Phi_B(X \setminus \closure{\{x\}})\).
The sobriety assumption plays a key role here, as it allows us to identify \(x\) and \(\closure{\{x\}}\), \emph{i.e.}\ it allows to see points uniquely (and not ``doubly'' or ``\emph{drunkenly}'', hence the name).

We observe the following for later reference.

\begin{lemma} \label{lemma:x-alg-in-product}
  Let \(X\) be a sober \tzero space. Any \cstar{}subalgebra \(B \subseteq \prod_{x \in X} B_x\) is an \(X\)-algebra.
  Conversely, if \(B\) is a \emph{faithful} \(X\)-algebra, in the sense that \(\norm{b} = \sup_{x \in X} \norm{b_x}\), then \(B\) embeds into \(\prod_{x \in X} B_x\).
\end{lemma}
\begin{proof}
  Let us first discuss the forward implication.
  For any subset \(\CU \subseteq X\) we have that the projection \(\pi_\CU \colon \prod_{x \in X} B_x \to \prod_{x \in X \setminus \CU} B_x\) is well defined, and restricts to a representation \(\pi_{\CU}|_B\) of \(B\).
  Thus, given an open set \(\CU \subseteq X\) we can consider \(\Phi_B(\CU) \coloneqq \ker(\pi_{\CU}|_B) \in \idealsin{B}\).
  The assignment \(\Phi_B\) is then order preserving, for if \(\CU \subseteq \CV\) are open sets of \(X\) then \(X \setminus \CU \supseteq X \setminus \CV\), so that \(\ker(\pi_\CU|_B) \subseteq \ker(\pi_\CV|_B)\).

  For the ``converse'', \(\norm{b} = \sup_{x \in X} \norm{b_x}\) exactly says the canonical map \(B \to \prod_{x \in X} B_x\) is isometric.
\end{proof}

The ``faithfulness'' notion of the \(X\)-structure is equivalent to saying that \(\bigcap_{x \in X} \Phi_B(X \setminus \closure{\{x\}}) = 0\), see \cite{blackadar-book-2010}*{II.6.5.5}.
In particular the \(X\)-structure is faithful whenever the structure map \(\Phi_B\) is upper semi-continuous and finitely lower semi-continuous, as then we can form the dual \(\psi_B \colon \primidealspc{B} \to X\), and
\[
  \bigcap_{x \in X} \Phi_B\left(X \setminus \closure{\left\{x\right\}}\right) = \bigcap_{x \in X} \bigcap_{P \not\in \psi_B^{-1}\left(X \setminus \closure{\left\{x\right\}}\right)} P = \bigcap_{P \in \primidealspc{B}} P = 0.
\]
\cref{lemma:x-alg-in-product} also allows to discuss the ``support'' of an element \(b\) of an \(X\)-algebra \(B\).

\begin{definition} \label{def:compactly-supported element}
  Let \(B\) be an \(X\)-algebra. We say \(b \in B\) is \emph{compactly supported} if its open/strict support \(\supp{b} \coloneqq \{x \in X : b_x \neq 0\}\) is contained in a compact subset of \(X\).
\end{definition}

We note in passing that we are not requiring that \(\{x : b_x \neq 0\}\) be a compact set, for \(X\) could well fail to have an abundance of closed compact sets.
We also observe the following.

\begin{lemma} \label{lemma:comp-supp-are-dense}
  The set of compactly supported elements of an \(X\)-algebra \(B\) forms a dense \Star{}sub-algebra.
\end{lemma}
\begin{proof}
  Take some positive \(b \in B\) and \(\varepsilon > 0\).
  The element \((b-\varepsilon)_+\) is positive and \(\varepsilon\)-close to \(b\), so it suffices to show \(\{x : \norm{b_x} \geq \varepsilon\}\) is compact in \(X\). This is implied by \cite{blackadar-book-2010}*{Proposition II.6.5.6 (iii)}.
\end{proof}

We now recall a theorem of Dauns and Hofmann~\cite{dauns-hofmann-1968}, whose statement has been adapted to the case of a general \(X\)-algebra in the sense of Gabe \cite{gabe-memoirs-2024}.
The proof is given in \cite{elliott-dauns-hofmann-1974}*{Theorem~3}.

\begin{theorem} [cf.\ \cites{dauns-hofmann-1968,elliott-dauns-hofmann-1974}] \label{thm-dauns-hofmann}
  Let \(B\) be a faithful \(X\)-algebra with structure map \(\Phi_B\), and let \(\CU \subseteq X\) be open.
  For any continuous bounded \(f \colon \CU \to \CC\) and \(b \in \Phi_B(\CU) \in \idealsin{B}\), there is a unique \(fb \in B\) such that
  \[
    \left(fb\right)\left(x\right) = f\left(x\right) b_x \in B_x \stackrel{\rm (def)}{=} B/\Phi_B(X \setminus \closure{\{x\}})
  \] 
  for all \(x \in \CU \subseteq X\).
\end{theorem}

Recall \(\cfunz{X}\) from \cref{example:cfunz}.
The following generalizes the ``Kasparov-way'' of looking at faithful \(X\)-algebras in the sense of Gabe \cite{gabe-memoirs-2024}.
In particular if \(X\) is Hausdorff and \(\Phi_B \colon \opensets{X} \to \idealsin{B}\) is continuous then this simply recovers the structure map \(\pi \colon \contz{X} \to \CZ(\multialg{B})\).
\begin{corollary}
  Let \(B\) be a faithful \(X\)-algebra. There is a canonical \Star{}homomorphism
  \begin{equation} \label{eq:dauns-hofmann}
    \pi \colon \cfunz{X} \to \CZ\left(B^{**}\right)
  \end{equation}
  with the feature that \(\pi(f) \in \CZ(\multialg{\Phi_B(\CU)})\) whenever \(f \in \contz{\CU}\).
  Conversely, should there be a \Star{}homomorphism \(\pi \colon \cfunz{X} \to \CZ(B^{**})\), then \(B\) is equipped with a canonical \(X\)-structure, where:
  \[
    \Phi_B \colon \opensets{X} \to \idealsin{B}, \;\; \Phi_B\left(\CU\right) \coloneqq \left\{b \in B : b \pi\left(\cfunz{\CU}\right) = 0\right\}.
  \]
\end{corollary}
\begin{proof}
  For the direct implication, the Dauns and Hofmann result above, see \cref{thm-dauns-hofmann}, allows us to see \(f \in \cfunz{X}\) as an element in \(\CZ(\multialg{\Phi_B(\CU)})\) as long as the \(X\)-structure of \(B\) is faithful and \(f\) is continuous and supported on the open set \(\CU \subseteq X\).
  As these functions generate \(\cfunz{X}\) this means that there is a canonical map as in \eqref{eq:dauns-hofmann} with the desired added feature.
  For the converse direction, simply notice that \(\Phi_B(\CU)\) thus defined is an ideal (for \(\pi\) lands in the center of \(B^{**}\)), and that these ideals are nested in one another.
  Whence \(\Phi_B \colon \opensets{X} \to \idealsin{B}\) is indeed order preserving.
\end{proof}

\subsection{Acting dynamics}

Having introduced the spaces \(X\) on which we will act we now introduce their dynamical counterparts.
\emph{Inverse semigroups} were introduced independently by Preston~\cites{Preston1954-1,Preston1954-2,Preston1954-3} and Wagner~\cites{Wagner1952,Wagner1953} in the 1950's, as a notion of ``local groups''.
They have turned out to be quite useful in the study of \cstar{}algebras, see \cites{buss-meyer-2017-actions,buss-martinez-approx-prop-23,BussExelMeyer2017,exel-nc-cartan-subalgs-2011,Exel:Inverse_combinatorial}, and many more.
We will use them here too.
A semigroup \(S\) is called \emph{inverse} if for all \(s \in S\) there is a unique \(s^* \in S\) such that \(ss^*s = s\) and \(s^*ss^* = s^*\).
Furthermore, an element \(e \in S\) is an \emph{idempotent}, or \emph{projection}, if \(e^2 = e\) (and thus \(e^* = e\) by uniqueness of the adjoint). The set of idempotents will be denoted by \(E\). Recall that \(E\) is commutative \cite{Lawson:InverseSemigroups}.
Inverse semigroups are naturally equipped with a partial order, where we say \(s \leq t\) if \(ts^*s = s\) (equivalently there is some \(e \in E\) with \(et = s\), and variants thereof). For idempotents \(e, f \in E\) we have that \(e \leq f\) if \(ef = fe = e\).
\begin{definition}[see \cite{chyuan_chung_inv_sem_2022}*{Definition 2.5}]
  An inverse semigroup \(S\) is \emph{quasi-countable} if there is some countable set \(K \subseteq S\) such that \(S = K E\) (equivalently, \(S\) is generated by \(K' \cup E\) as an inverse semigroup for some countable \(K' \subseteq S\), namely the set of finite words in \(K\)).
\end{definition}

Our dynamical systems will be given by actions of (often quasi-countable) inverse semigroups on (often separable) \cstar{}algebras. We recall these notions now, cf.\ \cites{buss-meyer-2017-actions,Buss-Exel-Meyer:Reduced,KwasniewskiMeyer2021,Kwasniewski-Meyer:Pure_infiniteness,buss-martinez-approx-prop-23}, among others.
\begin{definition}
  An \emph{action} \(S \actson X\) of \(S\) on a locally compact, sober, \tzero space \(X\) is a collection of partial homeomorphisms \(s \colon X_{s^*s} \to X_{ss^*}\) between open sets \(X_{s^*s} \subseteq X\) such that \(st = s \cdot t\) and \(\cup_{e \in E} X_e\) is dense in \(X\).
  We may also say \(X\) is an \emph{\(S\)-space.}
  An \emph{action} \(S \actson B\) of \(S\) on a \cstar{}algebra \(B\) is a collection of partial \Star{}isomorphisms \(s \colon B_{s^*s} \to B_{ss^*}\) between ideals \(B_{s^*s} \subseteq B\) such that \(st = s \cdot t\) and the span of \(\cup_{e \in E} B_e\) is dense in \(B\).
  We may also say \(B\) is an \emph{\(S\)-\cstar{}algebra.}
\end{definition}

In the case of an action of \(S\) on \(X\) we will usually denote the open set \(X_{s^*s}\) simply by \(s^*s\) (and similarly with \(ss^*\)). 
In any of the actions above the statement ``\(s \cdot t = st\)'' of course means that the composition of the partial maps \(s\) and \(t\) (in that order) equals the partial map \(st\) (in that order).
For instance, if \(S \actson B\) then this in particular implies that \(t^* B_{s^*stt^*} = B_{(st)^* st}\) and \(s B_{s^*s tt^*} = B_{st(st)^*}\), and these are ideals of \(B\).

\begin{remark}
  Much in the same spirit as Gabe's definition of an \(X\)-algebra, see \cref{def:x-algebra}, we could have chosen to forgo the use of points of \(X\) entirely and try to write the paper from a ``point-free topology'' perspective.
  In this way we would have defined an action of an inverse semigroup \(S\) on a semilattice \(\CL\) as a collection of order isomorphisms \(s \colon \CL_{s^*s} \to \CL_{ss^*}\) between ideal sub-lattices such that \(s \cdot t = st\).
  Nevertheless, we \emph{do} need to add a choice of set \(D_0 \subseteq X\) to the discussion, as it will be central for the functoriality properties of \cref{sec:functoriality}.
  It is, simply, unclear how to do this without ``seeing points'' of \(X\).
\end{remark}

We need to generalize the allowed actions even further, for we strive to prove results about \emph{twisted} inverse semigroup actions.
Note that these are needed in order to properly talk about Cartan subalgebras, for instance.
Instead of introducing twisted inverse semigroup actions \cite{buss-exel-twisted-acts-2011} we will recall the following, introduced by Sieben~\cites{Sieben1997,SiebenFellbundles}.
(For the notion of Hilbert bimodule we refer the reader to \cite{lance_hilbert_modules_1997}.)

\begin{definition} \label{def:fell-bundle}
Let \(S\) be an inverse semigroup, and let \(B\) be a \cstar{}algebra. A \emph{Fell bundle \(\CB\) over \(B\)} is a collection of Hilbert \(B\)-\(B\)-bimodules \((B_s)_{s \in S}\); and bimodule isomorphisms \(\mu_{s, t} \colon B_s \times B_t \to B_{st}\) for all \(s, t \in S\) such that:
\begin{enumerate}[label=(\roman*)]
  \item \(B_e\) is an ideal of \(B\) for all \(e \in E\) and \(\spn{\cup_{e \in E} B_e} \subseteq B\) is dense;
  \item \(\mu_{s, e} \colon B_s \times B_e \rightarrow B_{se}\) and \(\mu_{e, s} \colon B_e \times B_s \rightarrow B_{es}\) are the canonical maps, that is, the ones coming from the bimodule structure of \(B_s\); and
  \item for all \(s, t, r \in S\) the following diagram commutes:
    \begin{equation*}
      \begin{tikzcd}[scale=50em]
        B_s \times B_t \times B_r \arrow{d}{\id_{B_s} \times \mu_{t, r}} \arrow{r}{\mu_{s, t} \times \id_{B_r}} &[2.2em] B_{st} \times B_r \arrow{d}{\mu_{st, r}} \\
        B_s \times B_{tr} \arrow{r}{\mu_{s, tr}} & B_{str}.
      \end{tikzcd}
    \end{equation*}
\end{enumerate}
\end{definition}

Just as with group actions, inverse semigroup actions \(S \actson B\) can also be twisted and these, in turn, can be turned into Fell bundles.
A comparison between these Fell bundles (called \emph{regular} in \cite{buss-exel-twisted-acts-2011}) and our Fell bundles, and how they are related with twisted actions, can be found in \cite{buss-exel-twisted-acts-2011}.
The upshot is that our approach generalizes that of twisted actions.
We also refer the reader to \cites{Kranz:weak-containment,Kumjian:Fell_bundles} for analogue notions for Hausdorff \'etale groupoids.

We will systematically denote the Fell bundle by \(\CB = (B_s)_{s \in S}\).
\cref{def:fell-bundle} is actually about \emph{saturated} Fell bundles, as the maps \(\mu_{s,t}\) are assumed to be isomorphisms. We refer the reader to \cite{buss-martinez-approx-prop-23}*{Section 2} for a more comprenhensive discussion about this issue.
We will not explicitly write the ``multiplication'' maps \(\mu_{s,t}\), and simply say that \(b_s b_t \in B_s B_t \cong B_{st}\) whenever \(b_s \in B_s\) and \(b_t \in B_t\).
For instance, if \(S \actson B\) is a run-of-the-mill action then the collection \(B_s \coloneqq B_{ss^*}\) defines a canonical associated Fell bundle, where
\begin{equation} \label{eq:fell-bundle-from-action}
  B_s \times B_t \ni \left(a_s, a_t\right) \mapsto s \left(s^* a_s \cdot a_t\right) \in B_{st}.
\end{equation}
For the convenience of the reader and to improve legibility we will denote elements \(b\) in a fiber \(B_s\) by the subscript \(b_s\).

\begin{example}
  One important class of examples of Fell bundle arising as in \eqref{eq:fell-bundle-from-action} is the case when \(S \actson X\).
  This precisely corresponds to the \'etale groupoid case, see \cref{sec:apps} for a more detailed discussion.
\end{example}

\begin{notation}
  From now on, caligraphic letters, \emph{e.g.}\ \(\CA, \CB, \CI\), will denote Fell bundles over an inverse semigroup \(S\) (except for the algebra of bounded, adjointable operators on a Hilbert module, which will be denoted by \(\CL(H)\)).
  The non-caligraphic letters, \emph{e.g.}\ \(A, B, I\), will denote the unit fibers of these Fell bundles.
  In the same vein, by \(\cfunz{X}\) we will mean both the algebra from \cref{example:cfunz} and the Fell bundle \(\cfunz{X} \coloneqq (\cfunz{ss^*})_{s \in S}\) arising via \eqref{eq:fell-bundle-from-action} from the action \(S \actson X\).
  It will be clear from context whether we mean the Fell bundle or only the \cstar{}algebra.
\end{notation}

The following is a technical definition that has turned out to be quite fundamental in the study of inverse semigroup actions, see \emph{e.g.}\ \cites{Kwasniewski-Meyer:Pure_infiniteness,buss-martinez-approx-prop-23}.
It talks about when two elements \(s,t \in S\) are ``trivially the same''.
This never happens for a group \(\Gamma\) (unless \(\gamma = \rho \in \Gamma\)), but for groupoids it does: two different open bisections can intersect non-trivially.

\begin{definition} \label{def:ideals-i-s-t}
  Let \(\CB = (B_s)_{s \in S}\) be a Fell bundle. Given elements \(s, t \in S\) we let
  \[
    I_{s,t} \coloneqq \langle B_e : e \in E \text{ and } e \leq st^* \rangle \subseteq B
  \]
  (where \(\langle \variable \rangle\) denotes the ideal generated by \(\variable\)).
  We denote by \(1_{s,t} \in \multialg{I_{s,t}} \subseteq B^{**}\) the unit of \(I_{s,t}^{**} \subseteq B^{**}\).
\end{definition}

Within this generality the following result is one of the key components of the framework we will work in.
It is surely well known to experts, but having found no reference we produce a proof.
\begin{proposition} \label{prop:setting-actions}
  Let \(\CB = (B_s)_{s \in S}\) be a Fell bundle over \(B\).
  Then \(S\) acts on \(\primidealspc{B}\) (in the usual sense), and the actions are compatible, in the sense that for all open \(\CU \subseteq \primidealspc{B}\) and \(s \in S\):
  \[
    \Phi_B\left(s \cdot \left(\CU \cap \primidealspc{B}_{s^*s}\right)\right) = s \left(\Phi_B\left(\CU\right) \cdot B_{s^*s}\right),
  \]
  where \(\Phi_B \colon \opensets{\primidealspc{B}} \to \idealsin{B}\) is the usual map (see, \emph{e.g.}\ \cite{gabe-memoirs-2024}*{Example 10.10}).
\end{proposition}

\begin{proof}
  For future applications in mind we will prove something slightly stronger: there is a lower semi-continuous bundle of \cstar{}algebras and the actions are compatible in the above sense.
  The existence of the bundle is not new: given any primitive ideal \(P \in \primidealspc{B}\) we can form the quotient \cstar{}algebra \(B/P\).
  For any \(b \in B\) we may now construct \(\gamma(b) \colon \primidealspc{B} \to \prod_{P \in \primidealspc{B}} B/P\), where \(\gamma(b)(P) \coloneqq b + P\). Defining
  \[
    \sectionsof{B} \coloneqq \left\{\gamma\left(b\right) : b \in B\right\}
  \]
  then gives the desired sections of the bundle.
  For instance, the set \(\sectionsof{B}\) is clearly a linear subspace, and \(\{\gamma(b)(P) : \gamma(b) \in \sectionsof{B}\} = \{b+P : b \in B\}\) is dense in \(B/P\) (it is in fact \emph{all} of \(B/P\)).
  The fact that \(P \mapsto \norm{b + P}\) is lower semi-continuous for all \(b \in B\) is proven in \cite{blackadar-book-2010}*{Proposition~II.6.5.6~(i)}.

  We now move onto the only part of the statement that is not entirely clear, which is why \(S\) would act on \(\primidealspc{B}\) and \(\CB\), let alone compatibly. For this, just let
  \[
    \primidealspc{B}_{s^*s} \coloneqq \left\{P \in \primidealspc{B} : P \not\subseteq B_{s^*s}\right\} \cong \primidealspc{B_{s^*s}} ,
  \]
  which is an open subset of \(\primidealspc{B}\) because \(B_{s^*s}\) is an ideal of \(B\).
  Consider the map
  \begin{align*}
    s \colon \primidealspc{B}_{s^*s} & \to \primidealspc{B}_{ss^*}, \\
    P & \mapsto sP \coloneqq \spnc{b_s p b_{s^*} : b_s \in B_s, p \in P \text{ and } b_{s^*} \in B_{s^*}} \subseteq B_{ss^*}.
  \end{align*}
  We claim that this defines a partial homeomorphism sending primitive ideals contained in \(B_{s^*s}\) to primitives contained in \(B_{ss^*}\).
  First, note that \(B_{s^*s} = \sectionsof{B}_{s^*s}\) as in the statement.
  In fact, in order to see that \(P \mapsto sP\) is a homeomorphism recall that the closure of \(\FP \coloneqq \{P_i\}_i\) in \(\primidealspc{B}\) is, by definition, \(\closure{\FP} = \{Q : Q \supseteq \cap_i P_i\}\).
  If \(\closure{\FP} \subseteq \primidealspc{B}_{s^*s}\) then \(P_i \subseteq B_{s^*s}\) (modulo the identification \(\primidealspc{B}_{s^*s} \cong \primidealspc{B_{s^*s}}\)), so that \(s \FP \coloneqq \{s P_i\}_i\) makes sense.
  It suffices to show that \(s \closure{\FP} = \closure{s \FP}\), that is, that \(s\) is a (partial) bijection mapping closed sets to closed sets. (The inverse of \(s\) is of course given by \(s^*\).)
  This is an application of the classical Rieffel correspondence, due to the fact that \(B_s\) is a \(B_{ss^*}-B_{s^*s}\)-imprimitivity bimodule, and thus defines a homeomorphism between the primitive ideal spaces of \(B_{s^*s}\) and \(B_{ss^*}\).
\end{proof}

\begin{definition} \label{def:s-x-algebra}
  Let \(S \actson X\) be an action of an inverse semigroup on a locally compact, sober, \tzero space.
  An \emph{\((S, X)\)-\cstar{}algebra} is a \cstar{}algebra \(B\) equipped with an action \(S \actson B\) and an \(S\) equivariant structure map \(\Phi_B \colon \opensets{X} \to \idealsin{B}\) (we may also say that the actions are \emph{compatible}).
  By ``equivariance'' we mean that
  \begin{equation} \label{eq:equiv-struct-map}
    \Phi_B\left(s \left(\CU \cap s^*s\right)\right) = s \, \left(\Phi_B\left(\CU\right) \cdot B_{s^*s}\right) \;\; \text{ for all open } \; \CU \subseteq X \text{ and } s \in S. 
  \end{equation}
\end{definition}

\cref{prop:setting-actions} evidently inspired \cref{def:s-x-algebra}.
Notice that if \(\CB = (B_s)_{s \in S}\) is a Fell bundle over \(B\) then \(B\) is an \((S, \primidealspc{B})\)-algebra: the equivariant map \(\Phi_B \colon \opensets{\primidealspc{B}} \to \idealsin{B}\) is, simply, the canonical one.
By abuse of notation, we will sometimes say that a \cstar{}algebra \(B\) is an \((S, X)\)-algebra even if the data comes from a Fell bundle \(\CB = (B_s)_{s \in S}\), particularly in \cref{sec:functoriality}.
We observe the following only for later reference.

\begin{proposition} \label{type-i-iff-fibers-type-i}
  Let \(B\) be an \(X\)-algebra. If \(B\) is type I as a \cstar{}algebra; then \(B_x\) is type I for all \(x \in X\).
  Likewise, the converse holds if for every \(P \in \primidealspc{B}\) there is some \(x \in X\) such that \(P = \Phi_B(X \setminus \closure{\{x\}})\).
\end{proposition}
\begin{proof}
  The forward implication is clear: every quotient of a type I \cstar{}algebra is type I.
  For the converse, let \(\pi \colon B \to \CL(H)\) be an irreducible representation.
  Putting \(P \coloneqq \ker(\pi)\) we may choose some \(x \in X\) such that \(P = \Phi_B(X \setminus \closure{\{x\}})\).
  It hence follows that \(B_x = B/P = B/\ker(\pi)\) is type I, and \(\pi\) factors through \(B_x\).
  Thus \(\pi(B)\) contains the compact operators of \(H\), as desired.
\end{proof}

It should be highlighted that requiring the actions to be on type I \cstar{}algebras is rather stronger than requiring the quotients \(B_x\) be type I.

\begin{example}
  To show that it is, in fact, much stronger recall that \(\maxalg{\FF_\infty}\) is residually finite dimensional, that is, there is some embedding \(\maxalg{\FF_\infty} \subseteq \prod_{n \in \NN} M_{k(n)}\). The kernels of the \Star{}homomorphisms \(\pi_\ell \colon \maxalg{\FF_\infty} \subseteq \prod_{n \in \NN} M_{k(n)} \to M_{k(\ell)}\) equip the algebra with a type I \(\NN\)-structure as in \cref{ex:n-structure}. On the other hand, \(\maxalg{\FF_\infty}\) is, most definitely, not type I as a \cstar{}algebra.
\end{example}

\begin{remark} \label{rem:non-hausdorffness}
  There are two ``kinds'' of ``non-Hausdorffness'' we will have to deal with.
  Firstly, \(X\) may typically be \(\primidealspc{B}\), which itself may only be a \tzero space.
  On the other hand, the actions \(S \actson X\) also introduce some ``non-Hausdorffness'' which was not present in \(X\) at the start.
\end{remark}

\subsection{Cross-sectional algebras and crossed products}

We now turn to the construction of maximal and reduced cross-sectional \cstar{}algebras \(\maxalg{\CB}\) and \(\redalg{\CB}\) of a Fell bundle \(\CB = (B_s)_{s \in S}\).
This procedure is well known (at least the basics of it), and we refer the reader to \cites{buss-meyer-2017-actions,Buss-Exel-Meyer:Reduced,buss-martinez-approx-prop-23} for a long discussion about it.
Here we only recall some features in the following proposition.

\begin{proposition}
  Let \(\CB = (B_s)_{s \in S}\) be a Fell bundle over \(B\).
  The Banach spaces \(B_s\) embed canonically into \(\maxalg{\CB}\) and \(\redalg{\CB}\); and they span a dense sub-algebra, denoted by \(\algalg{\CB}\).
  Moreover, there is a quotient map \(\maxalg{\CB} \to \redalg{\CB}\) preserving every \(B_s\), that is, the quotient map restricts to an isomorphism of \(\algalg{\CB}\).
  Lastly, there is a faithful, weak conditional expectation \(\redcondexp \colon \redalg{\CB} \to B^{**}\).
  Moreover, on monomials, \(\redcondexp\) takes the form \(\redcondexp(b_s) = b_s 1_{s}\), where \(1_s \in B^{**}\) is the unit of the weakly closed ideal \(I_{s,1}^{**} = \langle B_e : e \leq s, e \in E\rangle\).
\end{proposition}
\begin{proof}
  The fact that \(\algalg{\CB}\) sits densely in \(\maxalg{\CB}\) is shown in \cite{Buss-Exel-Meyer:Reduced}*{Proposition 4.3}, see also \cite{buss-martinez-approx-prop-23}*{Lemma 2.28} for a short proof that the canonical map \(\maxalg{\CB} \to \redalg{\CB}\) is the identity of \(\algalg{\CB}\).
  The canonical weak conditional expectation \(\redcondexp \colon \redalg{\CB} \to B^{**}\) is constructed in \cites{Buss-Exel-Meyer:Reduced,buss-meyer-2017-actions}.
\end{proof}

However, we want to work on an even larger generality, as we need to also consider the \emph{essential} groupoid \cstar{}algebra \(\essalg{G}\). Because of this we must allow to ``ignore'' certain units \(D_0 \subseteq X\).
This is an idea similar to others in \cites{buss-martinez-ess-ame-2024,KwasniewskiMeyer2021} and references therein.
Thus, we henceforth work under the assumption that \(S \actson X\) is an action; for instance, it may be the action \(S \actson \primidealspc{B}\) induced from some Fell bundle \(\CB = (B_s)_{s \in S}\), see \cref{prop:setting-actions}.

\begin{definition}
  A subset \(D_0 \subseteq X\) is \emph{invariant} if \(sx \in D_0\) for all \(x \in D_0\) and \(s \in S\) such that \(x \in s^*s\).
\end{definition}
Both the empty set and \(X\) are clearly invariant, but there could be more.
For instance, every orbit is clearly invariant.
Note that we are \emph{not} asking any topological property of \(D_0\): it could be open, closed, Borel, meager, or none of these.
For our intents the invariant subset \(D_0\) will be ignored (or killed) when we construct our \cstar{}algebras \(\maxalg[D_0]{\CB}\) and \(\redalg[D_0]{\CB}\).
Hence, given an action \(S \actson X\) and some fixed \(D_0 \subseteq X\), we want to introduce functors \(\maxalg[D_0]{\variable}\) and \(\redalg[D_0]{\variable}\) from the category of \((S, X)\)-algebras such that the following hold.
\begin{enumerate} [label=(\alph*)]
  \item There is a canonical surjection \(\maxalg[D_0]{\variable} \to \redalg[D_0]{\variable}\). This will be immediate from the construction itself, see \cref{def:cross-algebra}.

  \item There is a \Star{}isomorphism \(\essalg{S \ltimes X} \cong \redalg[D]{\cfunz{X}}\), where \(S \rtimes X\) is the groupoid of germs of an action \(S \actson X\). Here, \(\cfunz{X}\) denotes the Fell bundle of the action \(S \actson X\).

  \item More generally, there is a choice of \(D \subseteq \primidealspc{B}\) such that \(\redalg[D]{\CB} \cong \essalg{\CB}\), where the latter is Kwa\'sniewski-Meyer's \emph{essential} \cstar{}algebra from \cite{Kwasniewski-Meyer:Pure_infiniteness}.
  This will be done in \cref{sec:apps-ess}.

  \item \(\maxalg[D_0]{\variable}\) is functorial with respect to some ``\(D_0\)\emph{-equivariant} (\emph{asymptotic}) \emph{morphisms}'' \(\CA \asymparrow \CB\). This will be carried out in \cref{sec:functoriality}.
\end{enumerate}

In order to do this the following is one of the key components. For this, recall that the reduced crossed product \(\redalg{\CB}\) comes equipped with a faithful (weak) conditional expectation
\begin{align*}
  \redcondexp \colon \redalg{\CB} \to B^{**}, \;\; \text{ where } \, b_s \mapsto b_s 1_s \in B_s \cdot I_{s,1}^{**} \in B^{**},
\end{align*}
see \cite{buss-martinez-approx-prop-23} for details.
The following is inspired by previous literature, see \cites{KwasniewskiMeyer2021,Kwasniewski-Meyer:Pure_infiniteness} and \cite{AraMathieu2003}*{Theorem 2.3.9}.

\begin{proposition} \label{prop:cond-exp-kills-dzero}
  Let \(\CB = (B_s)_{s \in S}\) be a Fell bundle over \(B\).
  Suppose \(B\) is an \((S, X)\)-algebra, and fix some invariant \(D_0 \subseteq X\).
  The set
  \[
    J_{D_0} \coloneqq \closure{\left\{b \in B^{**} : b(x) = 0 \;\; {{\rm for \, all }} \; x \in X \setminus D_0\right\}}
  \]
  forms an ideal of \(B^{**}\) (here \(b(x)\) denotes the image of \(b\) under the unique extension of \(B \ni b \mapsto b(x) \in B_x\)).
  It follows that \(\condexp_{D_0} \colon \redalg{\CB} \ni b \mapsto \redcondexp(b) + J_{D_0} \in B^{**}/J_{D_0}\) is a generalized conditional expectation~\cites{KwasniewskiMeyer2021,Kwasniewski-Meyer:Pure_infiniteness}.
\end{proposition}
\begin{proof}
  We will use the same notation as in the proof of \cref{def:s-x-algebra}.
  First note that the maps \(q_x \colon B \to B_x = B/\Phi_B(X \setminus \closure{\{x\}})\) are quotient maps of \cstar{}algebras.
  In particular they extend uniquely to normal quotient maps \(q_x^{**} \colon B^{**} \to B_x^{**}\), which we can bundle up to a \Star{}homomorphism \(q \colon B^{**} \to \prod_{x \in X} B_x^{**}\).
  The set \(J_{D_0}\) as in the statement is then clearly a subalgebra, as it equals
  \[
    J_{D_0} = \closure{\left\{b \in B^{**} : \supp{q\left(b\right)} \subseteq D_0\right\}},
  \]
  where \(\supp{q(b)} \coloneqq \{x \in X : q_x^{**}(b) \neq 0\}\).
  If \(b \in B^{**}\) and \(c \in J_{D_0}\) then \(\supp{q(bc)} \subseteq \supp{q(b)} \cap \supp{q(c)} \subseteq D_0\), which proves that \(J_{D_0}\) is a \Star{}ideal.
  Lastly, it is closed in norm (by construction).
  The fact that
  \[
    \condexp_{D_0} \colon \redalg{\CB} \ni b \mapsto \redcondexp(b) + J_{D_0} \in B^{**}/J_{D_0}
  \]
  is a generalized conditional expectation now follows, as it is the composition of a generalized conditional expectation (namely \(\redcondexp\)) and a \Star{}homomorphism, namely \(B^{**} \to B^{**}/J_{D_0}\).
\end{proof}

With the \cstar{}ideal \(J_{D_0}\) in \cref{prop:cond-exp-kills-dzero} built we are ready to introduce the reduced completions \(\redalg[D_0]{\CB}\) we will work with.

\begin{definition} \label{def:cross-algebra}
  Let \(\CB = (B_s)_{s \in S}\) be a Fell bundle over a \cstar{}algebra \(B\).
  Assume \(B\) is an \((S,X)\) algebra, and let \(D_0 \subseteq X\) be some invariant subset.
  We let
  \begin{align*}
    \redalg[D_0]{\CB} \coloneqq & \; \redalg{\CB}/\left\{a \in \redalg{\CB} : \condexp[D_0]\left(a^*a\right) = 0\right\} \\
    = & \; \redalg{\CB}/\left\{a \in \redalg{\CB} : \redcondexp\left(a^*a\right) \in J_{D_0}\right\}.
  \end{align*}
  Likewise, let \(\algalg[D_0]{\CB}\) be the image of the canonical dense subalgebra \(\algalg{\CB} \subseteq \redalg{\CB}\) under the quotient map \(\redalg{\CB} \to \redalg[D_0]{\CB}\).
\end{definition}

In order to show that \cref{def:cross-algebra} is well posed the following should be checked.

\begin{lemma} \label{lemma:description-j-d}
  With notation as in \cref{def:cross-algebra}, we have that
  \begin{align*}
    \left\{a \in \redalg{\CB} : \redcondexp\left(a^*a\right) \in J_{D_0}\right\} & = \left\{a \in \redalg{\CB} : \redcondexp\left(b^*a^*ab\right) \in J_{D_0} \; \text{ for all } \; b \in \redalg{\CB}\right\} \\
    & = \left\{a \in \redalg{\CB} : \redcondexp\left(abb^*a^*\right) \in J_{D_0} \; \text{ for all } \; b \in \redalg{\CB}\right\} \\
    & = \left\{a \in \redalg{\CB} : \redcondexp\left(aa^*\right) \in J_{D_0}\right\}.
  \end{align*}
  In particular the set is an ideal of \(\redalg{\CB}\).
\end{lemma}

\begin{proof}
  Let \(\CN_{D_0}^\ell\) be the set on the left hand side of the first row, and let \(\CN_{D_0}^r\) the set on the right hand side of the same first row.
  It is essentially clear (just add a unit to \(B\) and extend \(\CB\)) that \(\CN_{D_0}^\ell \supseteq \CN_{D_0}^r\), so take some \(a\) such that \(\redcondexp(a^*a) \in J_{D_0}\).
  Note that the set of \(b \in \redalg{\CB}\) such that \(\redcondexp(b^*a^*ab) \in J_{D_0}\) is a closed linear subspace (as \(J_{D_0}\) itself is closed).
  In particular it follows that we may assume that \(b = b_s \in B_s \subseteq \redalg{\CB}\) is some given monomial.
  In such case, notice that for all \(c_t \in B_t\):
  \[
    \redcondexp\left(b_s^* c_t b_s\right) \stackrel{\rm (def)}{=} b_s^* c_t b_s \cdot 1_{s^*ts,1} = b_s^* \cdot c_t 1_{ss^*t,1} \cdot b_s = b_s^* \cdot c_t 1_t \cdot b_s \stackrel{\rm (def)}{=} b_s^* \redcondexp\left(c_t\right) b_s,
  \]
  see \cite{buss-martinez-approx-prop-23}*{Proposition 2.17 (iii)}.
  As the monomials \(c_t\) span a dense set, it follows that \(\redcondexp(b_s^* a^*a b_s) = b_s^* \redcondexp(a^*a) b_s\),
  whose support is contained in \(s (\supp{\redcondexp(a^*a)} \cap s^*s) \subseteq s (D_0 \cap s^*s) \subseteq D_0\) by invariance of the latter.
  It hence follows that \(\redcondexp(b^*a^*ab) \in J_{D_0}\) for all \(b \in \redalg{\CB}\).
  The other equalities follow because \(\redcondexp\) is symmetric, see \cite{KwasniewskiMeyer2021}*{Theorem 3.22}.
  The ``in particular'' statement can now be shown simply by realizing that \(\CN_{D_0}^r\) is a left ideal.
\end{proof}

We define the maximal counterpart of \(\redalg[D_0]{\CB}\) in \cref{def:cross-algebra} via its universal property.
For this, recall the definition of a representation \(\pi \colon \CB \to \CL(H)\) of a Fell bundle~\cite{buss-martinez-approx-prop-23}, which amounts to a collection of bounded linear maps \(\pi_s \colon B_s \to \CL(H)\) with \(s \in S\) such that \(\pi_{st}(b_s b_t) = \pi_s(b_s) \pi_t(b_t)\) for all \(b_s \in B_s\) and \(b_t \in B_t\).

\begin{definition}
  We let \(\maxalg[D_0]{\CB}\) be the universal \cstar{}algebra of \(\algalg{\CB}\) with respect to \emph{\(D_0\)-representations}.
  This exactly means that if \(\pi \colon \maxalg{\CB} \to \CL(H)\) is a representation such that \(\pi(\{a \in \algalg{\CB} : \redcondexp(a^*a) \in J_{D_0}\}) = 0\), then \(\pi\) uniquely induces a map \(\maxalg[D_0]{\CB} \to \CL(H)\).
\end{definition}

It is clear from the construction that there is always a canonical quotient map from the maximal to the reduced crossed products: \(\maxalg[D_0]{\CB} \to \redalg[D_0]{\CB}\).
Indeed, the latter can be realized as the image of a \(D_0\)-representation in the sense above.
We now observe the following basic result, which describes \(\maxalg[D_0]{\CB}\) in an alternative way.

\begin{lemma} \label{lemma:alt-desc-max-dzero}
  If \(\maxalg[D_0]{\CB}\) is as above, then
  \(
    \maxalg[D_0]{\CB} \cong \maxalg{\CB}/\closure{\left\{a \in \algalg{\CB} : \redcondexp\left(a^*a\right) \in J_{D_0}\right\}}.
  \)
\end{lemma}

\begin{proof}
  Let \(\CN \coloneqq \{a \in \algalg{\CB} : \redcondexp(a^*a) \in J_{D_0}\}\).
  In order to prove the statement it suffices to show that \(\maxalg{\CB}/\closure{\CN}\) enjoys the universal property defining \(\maxalg[D_0]{\CB}\).
  Thus, let \(\pi \colon \CB \to \CL(H)\) be a \(D_0\)-representation, \emph{i.e.}\ a representation of the Fell bundle \(\CB\) such that \(\pi(\CN) = 0\).
  As any representation is contractive it follows that \(\pi(\closure{\CN}) = 0\), which means that \(\pi\) descends to a representation of \(\maxalg{\CB}/\closure{\CN}\).
\end{proof}

We observe the following, which is well known and will be a technical point later.
\begin{lemma} \label{lemma:approx-unit}
  If \(\CB = (B_s)_{s \in S}\) is a Fell bundle over \(B\), then the image of the canonical \Star{}homomorphism \(B \to \maxalg[D_0]{\CB}\) contains an approximate unit of the target algebra.
  The same holds for \(B \to \redalg[D_0]{\CB}\).
\end{lemma}
\begin{proof}
  First note that the map \(\pi_1 \colon B \to \maxalg{\CB}\) is a \Star{}homomorphism, as can be easily checked (it is in fact an embedding). The map \(\pi \colon B \to \maxalg[D_0]{\CB}\) in the statement is \(\pi_1\) composed with the canonical quotient \(\maxalg{\CB} \to \maxalg[D_0]{\CB}\) of \cref{lemma:alt-desc-max-dzero}.
  It thus is enough to show \(\{\pi_1(u_\lambda)\}_\lambda\) is an approximate unit for \(\maxalg{\CB}\) for all approximate units of \(B\), which is well known: \(\pi_1\) is non-degenerate.
\end{proof}

\begin{remark}
  The \Star{}homomorphism \(B \to \maxalg[D_0]{\CB}\) of \cref{lemma:approx-unit} may be non-injective.
  It is only injective if \(D_0 \subseteq X\) is appropriately small: for instance, if \(X\) is Hausdorff and \(D_0\) is meager.
\end{remark}

\begin{example}
  As aforementioned, the class of \cstar{}algebras \(\maxalg[D_0]{\CB}\) and \(\redalg[D_0]{\CB}\) includes all maximal and reduced twisted crossed products of actions \(\Gamma \actson B\); all \cstar{}algebras with weak Cartan sub-algebras; all quotients of \(\maxalg{\Gamma}\) for a virtually nilpotent group \(\Gamma\); all \cstar{}algebras with non-commutative Cartan sub-algebras; all twisted \'etale groupoid \cstar{}algebras; and possibly more.
  We will discuss these examples more in depth in \cref{sec:apps-ess,sec:apps}.
\end{example}

\subsection{Behavior with respect to tensor products and double duals}
Due to technical reasons, and with future applications in mind, one should check that considering \cstar{}algebras of the form \(\maxalg[D_0]{\CB^{**}}\) or \(\redalg[D_0]{A \tensor \CB}\) is a reasonable thing to do.
When \(D_0 = \emptyset\) the meaning of these is clear: they are simply the maximal (or reduced) \cstar{}algebras associated to the Fell bundles \(\CB^{**}\) or \(A \tensor \CB\) (see \cite{buss-meyer-2017-actions} for the former and \cref{lemma:fell-bundle-a-tensor-b} for the latter).
Nevertheless, the added variable \(D_0\) needs to be taken into account, so we observe how to do this here.

\begin{lemma} \label{lemma:double-dual-x-algebra}
  If \(B\) is an \(X\)-algebra then so is its double dual \(B^{**}\).
  In particular, if \(\CB = (B_s)_{s \in S}\) is a Fell bundle over an \((S,X)\)-algebra \(B\) and \(D_0 \subseteq X\) is invariant then \(\maxalg[D_0]{\CB^{**}}\) and \(\redalg[D_0]{\CB^{**}}\) make sense.
\end{lemma}
\begin{proof}
  It suffices to construct the structure map of \(B^{**}\), which is done as one first guesses:
  \[
    \Phi_{B^{**}} \colon \opensets{X} \to \idealsin{B^{**}}, \;\; \CU \mapsto \Phi_B\left(\CU\right)^{**} \subseteq B^{**},
  \]
  where \(\Phi_B \colon \opensets{X} \to \idealsin{B}\) is the structure map of \(B\).
  Observe that \(\Phi_B(\CU)\) is then an ideal of \(B\), so that \(\Phi_B(\CU)^{**}\) canonically embeds into \(B^{**}\).
  The map \(\Phi_{B^{**}}\) is order preserving, for so is \(\Phi_B\).
\end{proof}

The following shows how to construct ``trivial'' extensions of a Fell bundle (by a nuclear algebra).
\begin{lemma} \label{lemma:fell-bundle-a-tensor-b-trivial}
  Suppose \(A\) is a nuclear \cstar{}algebra, and that \(\CB = (B_s)_{s \in S}\) is a Fell bundle over \(B\).
  Then \(A \tensor \CB = (A \tensor B_s)_{s \in S}\) admits a canonical Fell bundle structure over \(A \tensor B\).
  Moreover, if \(B\) is an \((S,X)\) algebra then so is \(A \tensor B\).
\end{lemma}
\begin{proof}
  For the fact that \(A \tensor \CB\) forms a Fell bundle see \cite{buss-martinez-approx-prop-23}*{Lemma 5.2}.
  For the claim about \(A \tensor B\) simply note that \(\Phi_{A \tensor B}(\CU) \coloneqq A \tensor \Phi_B(\CU)\) is a well defined lattice morphism \(\opensets{X} \to \idealsin{A \tensor B}\) that is, moreover, compatible with the Fell bundle structure in \cite{buss-martinez-approx-prop-23}*{Lemma 5.2}. Indeed:
  \begin{align*}
    s \cdot \left(\Phi_{A \tensor B}\left(\CU\right) \cap A \tensor B_{s^*s}\right) = s \cdot \left(A \tensor \Phi_B\left(\CU\right) \cap A \tensor B_{s^*s}\right) = A \tensor B_{s \cdot \left(\CU \cap s^*s\right)} = \Phi_{A \tensor B}\left(s \cdot \left(\CU \cap s^*s\right)\right),
  \end{align*}
  as desired.
\end{proof}

With only slightly more effort we can generalize the previous construction a bit further: we can also allow non trivial actions too.

\begin{lemma} \label{lemma:fell-bundle-a-tensor-b}
  Suppose \(S \actson A\) is an action and \(\CB = (B_s)_{s \in S}\) is a Fell bundle.
  Let \(\CA\) be the Fell bundle associated to \(S \actson A\) in \eqref{eq:fell-bundle-from-action}.
  If \(A\) is nuclear then \(\CA \tensor \CB = (A_s \tensor B_s)_{s \in S}\) admits a canonical Fell bundle structure over \(A \tensor B\).
  Moreover, if \(A,B\) are \((S,X)\) algebras for a common \(X\) then so is \(A \tensor B\).
\end{lemma}

\begin{proof}
  Note that the expression \(\CA \tensor \CB\) a priori does not make sense, for it could well be the case that \(A_s \odot B_s\) admits several tensor products.
  This, however, is never so, see \cite{buss-martinez-approx-prop-23}*{Equation (5.3)} and recall that \(A\) is assumed to be nuclear, and hence so are its ideals \(A_{s^*s}\) for all \(s \in S\).
  We wish to define the multiplication table of the collection \((A_{s} \tensor B_s)_{s \in S}\), which is done as follows:
  \begin{align*}
    \left(A_{s} \tensor B_s\right) \times \left(A_{t} \tensor B_t\right) & \to A_{st} \tensor B_{st}, \\
    \left(a_{s} \tensor b_s\right) \left(a_{t} \tensor b_t\right) & \coloneqq a_s a_t \tensor b_s b_t.
  \end{align*}
  Checking that the above equips the collection \((A_{s} \tensor B_{s})_{s \in S}\) with a Fell bundle structure is routine.
  Suppose now that \(\Phi_A, \Phi_B \colon \opensets{X} \to \idealsin{A}, \idealsin{B}\) are the structure maps for \(A\) and \(B\) respectively, and suppose that they are compatible with the induced actions from \(S\).
  We wish to equip \(A \tensor B\) with a compatible \(X\)-structure which, again, is done as one think it is done:
  \[
    \Phi_{A \tensor B} \colon \opensets{X} \to \idealsin{A \tensor B}, \;\; \Phi_{A \tensor B}\left(\CU\right) \coloneqq \Phi_A\left(\CU\right) \tensor \Phi_B\left(\CU\right).
  \]
  The assignment \(\Phi_{A \tensor B}\) is order preserving, for both \(\Phi_A\) and \(\Phi_B\) are.
  Moreover, the \(X\)-structure of \(A \tensor B\) arising from \(\Phi_{A \tensor B}\) is compatible with the one coming from the Fell bundle \(\CA \tensor \CB\) (recall \eqref{eq:equiv-struct-map}):
  \begin{align*}
    \Phi_{A \tensor B}\left(s \cdot \left(\CU \cap s^*s\right)\right) & = \Phi_A\left(s \cdot \left(\CU \cap s^*s\right)\right) \tensor \Phi_B\left(s \cdot \left(\CU \cap s^*s\right)\right) = \left(s \cdot \left(\Phi_A\left(\CU\right) A_{s^*s}\right)\right) \tensor \left(s \cdot \left(\Phi_B\left(\CU\right) B_{s^*s}\right)\right) \\
    & = s \cdot \left(\Phi_A\left(\CU\right) A_{s^*s} \tensor \Phi_B\left(\CU\right) B_{s^*s}\right) = s \cdot \left(\Phi_{A \tensor B}\left(\CU\right) \left(A_{s^*s} \tensor B_{s^*s}\right)\right),
  \end{align*}
  as desired.
\end{proof}

Observe that it follows from \cref{lemma:fell-bundle-a-tensor-b-trivial} that we can consider the \cstar{}algebras \(\maxalg[D_0]{\czeror \tensor \CB \tensor \CK}\) and \(\redalg[D_0]{\czeror \tensor \CB \tensor \CK}\), where \(\CK\) is the algebra of compact operators of a separable Hilbert space.
For future applications, see \cref{sec:equiv-eth}, it is convenient to prove now that these ``trivial'' extensions are well-behaved.

\begin{proposition} \label{prop:trivial-extensions-maxalg}
  In the setting of \cref{lemma:fell-bundle-a-tensor-b}, if the action \(S \actson A\) is inner, then \(\maxalg[D_0]{A \tensor \CB} \cong A \tensor \maxalg[D_0]{\CB}\), and similarly for \(\redalg[D_0]{\variable}\).
  In particular this is the case if the action \(S \actson A\) is trivial.
\end{proposition}
\begin{proof}
  Recall that the action \(S \actson A\) is \emph{inner} if for all \(s \in S\) there is some unitary \(v_s \in \multialg{A_{s^*s}}\) such that \(s \cdot a_{s^*s} = v_s a_{s^*s} v_{s^*}\) for all \(a_{s^*s} \in A_{s^*s}\).
  Likewise, the action is \emph{trivial} if every \(s \in S\) acts like the identity on \(A\).\footnote{\, For the inverse semigroup reader, note that \(S\) could well have a zero element \(0\), in which case \(0\) would still act like the identity on \(A\).}

  For the proof, as usual, \(A\) will denote the \cstar{}algebra, while \(\CA\) will denote the Fell bundle arising from the action \(S \actson A\) as in \eqref{eq:fell-bundle-from-action}.
  Instead of explicitly constructing an isomorphism
  \[
    \Phi \colon \maxalg[D_0]{A \tensor \CB} \to A \tensor \maxalg[D_0]{\CB},
  \]
  we will first show that \(\maxalg{\CA} \tensor \maxalg[D_0]{\CB}\) satisfies the universal property defining \(\maxalg[D_0]{A \tensor \CB}\), and then show that \(\maxalg{\CA}\) is just \(A\).
  (Notice that \(A\) is nuclear and thus there is no confusion as to which tensor product \(\tensor\) one is using.)

  Let \(\pi \colon A \tensor \CB \to \CL(H)\) be a \(D_0\)-representation of the Fell bundle \(A \tensor \CB\).
  This amounts to linear bounded maps \(\pi_s \colon A_{ss^*} \tensor B_s \to \CL(H)\) satisfying certain covariant conditions.
  In particular \(\pi_{s^*s} \colon A_{s^*s} \tensor B_{s^*s} \to \CL(H)\) is a representation of the maximal tensor product, and thus yields two representations \(\pi^A_{s^*s}\) and \(\pi^B_{s^*s}\) with commuting images.
  We may now extend the \(\CB\)-family to \(\pi^B_s \colon B_s \to \CL(H)\) by decreeing \(\pi_s^B(b_s \cdot b_{s^*s}) \coloneqq \pi_s(b_s) \pi_{s^*s}^B(b_{s^*s})\) (here we use that the Fell bundle \(\CB\) is saturated, so that \(B_s = B_s B_{s^*s}\)).
  The family \(\pi^B \coloneqq \{\pi^B_s\}_{s \in S}\) is then a \(D_0\)-representation of \(\CB\), so it induces (by definition) a map \(\pi^B \colon \maxalg[D_0]{\CB} \to \CL(H)\).
  Since \(\spn{A_e : e \in E}\) is dense in \(A\) it also follows that a similarly defined family \(\{\pi_{s}^A\}_{s \in S}\) extends to a representation \(\pi^A \colon \CA \to \CL(H)\) whose range, moreover, commutes with that of \(\pi^B\) (for the same holds for \(\pi_{s^*s}^A\) and \(\pi_{s^*s}^B\)).
  By definition of the maximal tensor product, there is thus some representation \(\pi^A \tensor \pi^B \colon \maxalg{\CA} \tensor \maxalg[D_0]{\CB} \to \CL(H)\), as desired.

  Thus, in order to finish the proof it suffices to show that \(A \cong \maxalg{\CA}\), and this is where the innerness assumption comes into play.
  Notice that elements in \(\maxalg{\CA}\) are (limits of) finite linear combinations of monomials \(v_s a_{s^*s}\), subject to the usual commutation relations given in \eqref{eq:fell-bundle-from-action}.
  As \(v_s \in \multialg{A_{s^*s}}\) and \(a_{s^*s} \in A_{s^*s}\) it follows that every monomial generating \(\maxalg{\CA}\) is contained in \(A\), and since these are dense one obtains that \(A \cong \maxalg{\CA}\).
\end{proof}

\subsection{A brief interlude about non-equivariant E-theory} \label{sec:non-equiv-eth}
This section, unlike the previous ones, is only expository, and thus contains no proofs.
Reading through this is, nonetheless, a worthwhile endeavor, for this gives a glimpse onto the necessary ingredients one should harvest when constructing a Baum-Connes type assembly map, which will be properly done in \cref{sec:equiv-eth}.

Fix some Fell bundles \(\CA = (A_s)_{s \in S}\) and \(\CB = (B_s)_{s \in S}\) of an inverse semigroup over some \((S, X)\)-algebras \(A\) and \(B\).
Likewise, fix some invariant \(D_0 \subseteq X\).
One may want to consider an ``\(S\)-equivariant \Eth-theory'' group ``with respect to \(D_0\)'', tentatively termed \(\Egrp[S,D_0]{\CA}{\CB}\), via 
\begin{align*}
  \Egrp{\maxalg[D_0]{\CA}}{\maxalg[D_0]{\CB}} = \llbracket \czeror \tensor \maxalg[D_0]{\CA} \tensor \CK, \czeror \tensor \maxalg[D_0]{\CB} \tensor \CK \rrbracket,
\end{align*}
that is, the group of homotopy classes of asymptotic morphisms
\[
  \czeror \tensor \maxalg[D_0]{\CA} \tensor \CK \; \asymparrow \; \czeror \tensor \maxalg[D_0]{\CB} \tensor \CK.
\]  
The latter group should really be understood as the group of homotopy classes of \((S, X, D_0)\)-equivariant asymptotic morphisms \(\czeror \tensor \CA \tensor \CK \asymparrow \czeror \tensor \CB \tensor \CK\).
This is a point of view that would neatly justify the adornment ``\((S, D_0)\)'' of \(\Egrp[S,D_0]{\CA}{\CB}\).
This notwithstanding, in order to properly justify this nomenclature we would need to prove there is some ``descent map''
\[
  \Egrp[S,D_0]{\CA}{\CB} = \left\{\rm equivariant \; \Eth-theory\right\} \to \left\{\rm normal \; \Eth-theory\right\} = \Egrp{\maxalg[D_0]{\CA}}{\maxalg[D_0]{\CB}}
\]
that is well-behaved.
This, in turn, asks for asymptotic morphisms \(\CA \asymparrow \CB\) that are ``equivariant'', or even ``equivariant up to \(D_0\)'', to induce asymptotic morphisms \(\maxalg[D_0]{\CA} \asymparrow \maxalg[D_0]{\CB}\).
Likewise, should one have an extension \(0 \to \CI \to \CA \to \CB \to 0\) that is equivariant, then one ought to obtain an element in \Eth-theory as well.
In order to construct an equivariant category we shall thus need to work quite a bit: we need functoriality properties, and some descent map.
All of this discussion will be made precise and clear in \cref{sec:equiv-eth}.

\section{Essential algebras of Fell bundles} \label{sec:apps-ess}

In this section we show \cref{thm-intro-ess}, \emph{i.e.}\ that our construction \(\redalg[D_0]{\CB}\) really does include Kwa\'sniewski-Meyer's \emph{essential} \cstar{}algebra, see \cite{KwasniewskiMeyer2021}*{Definition 4.4}, along with the \emph{essential maximal} \cstar{}algebra defined by Buss and the author, see \cite{buss-martinez-ess-ame-2024}*{Definition 3.20}, at least when \(B\) is type I.
Fix a Fell bundle \(\CB = (B_s)_{s \in S}\) over some \cstar{}algebra \(B\).
Given some \(s \in S\) recall \cref{def:ideals-i-s-t}, which we now restrict to
\begin{equation} \label{eq:i-s-1}
  I_{1,s} \coloneqq \langle B_e : e \in E \text{ and } e \leq s \rangle.
\end{equation}
This is, by construction, an ideal in \(B_{s^*s} \subseteq B\), and thus corresponds to an open set in \(\primidealspc{B}\).
More concretely, it corresponds to \(\primidealspc{I_{1,s}} \cong \{P \in \primidealspc{B} : I_{1,s} \not\subseteq P\}\).
We let \(D_{s}\) be the boundary of \(\primidealspc{I_{1,s}}\) in \(\primidealspc{B_{s^*s}}\), that is:
\begin{equation} \label{eq:d-s}
  D_{s} \coloneqq \closure{\primidealspc{I_{1,s}}} \cap \primidealspc{B_{s^*s}} \setminus \primidealspc{I_{1,s}}.
\end{equation}
Such \(D_{s}\) is, by construction, a closed and nowhere dense subset of \(\primidealspc{B_{s^*s}}\), which is open.
Lastly, we let \(D \coloneqq \bigcup_{s \in S} D_{s}\), which is some subset of \(\primidealspc{B}\).
We observe the following.

\begin{lemma} \label{lemma:d-inv-countable}
  \(D\) is invariant. Moreover, if \(\CCC \subseteq S\) is such that \(S = \CCC E\), then \(D = \bigcup_{c \in \CCC} D_c\).
  In particular, if \(S\) is quasi-countable then \(D\) is Borel and meager in \(\primidealspc{B}\).
\end{lemma}

\begin{proof}
  These statements follow from basic considerations of inverse semigroup actions and have nothing to do with primitive ideal spaces.
  Thus, let \(X \coloneqq \primidealspc{B}\), and let \(S \actson X\) be the canonical action in \cref{prop:setting-actions}.
  We may let \(O_t \coloneqq \cup_{e \in E, e \leq t} e\), which is an open subset of \(t^*t\) (and also of \(tt^*\)).
  The sets \(D_t\) are the boundaries of \(O_t\) in \(t^*t\), that is, \(D_t = \closure{O_t} \cap t^*t \setminus O_t\) (by definition).
  Let \(x \in D_t \cap s^*s\) be given.
  We claim that \(sx \in D_{sts^*}\).
  We have that \(x \not\in O_t\), but there is a net \(\{x_\lambda\}_\lambda \subseteq O_t\) such that \(x_\lambda \to x\).
  Without loss of generality we may assume that \(\{x_\lambda\}_\lambda \subseteq s^*s\).
  It is not hard to check (see also \cite{buss-martinez-approx-prop-23}*{Remark 2.13 and Proposition 2.17}) that \(sx_\lambda \subseteq O_{sts^*}\).
  Indeed, if \(x_\lambda \in e\) for some idempotent \(e \leq t\), then \(sx_\lambda \in ses^*\), and the latter is an idempotent underneath \(sts^*\).
  Hence \(sx \in \closure{O_{sts^*}}\).
  The fact that \(sx \not\in O_{sts^*}\) is due to \(s \colon s^*s \to ss^*\) being a homeomorphism: should it be the case \(sx \in O_{sts^*}\) then \(x = s^*sx \in O_{s^*sts^*s} \subseteq O_{t}\), contradicting the assumption on \(x\).
  Lastly, we have to show that \(sx \in st^*s^*sts^* = (sts^*)^* (sts^*)\).
  Recall that \(x \in t^*t\) by assumption, which implies that \(sx \in st^*ts^*\).
  Moreover, \(x\) being in the closure of \(O_t\), which are fixed points, implies that \(tx = x\), which, in turn, means that \(sx \in st^*s^*sts^*\), as desired.

  The alternative description of \(D\) is similar. Let \(\CCC\) generate \(S\) modulo the idempotents.
  Clearly \(\bigcup_{c \in \CCC} D_c \subseteq D\), so it suffices to prove that \(D_t \subseteq \bigcup_{c \in \CCC} D_c\) for all \(t \in S\).
  Fix some \(t \in S\).
  By the assumption on \(\CCC\) there is some \(c \in \CCC\) such that \(t \leq c\).
  We claim that then \(D_t \subseteq D_c\).
  Take some \(x \in D_t = \closure{O_t} \cap t^*t \setminus O_t\).
  By \(t \leq c\) it follows that \(O_t \subseteq O_c\) and \(t^*t \subseteq c^*c\), implying \(x \in \closure{O_c} \cap c^*c\).
  Suppose it were the case that \(x \in O_{c}\), that is, suppose that \(x \in e\) for some idempotent \(e \leq c\).
  It would then follow that \(et\) is also an idempotent, for \(et = ect^*t = e t^*t\), which indeed is an idempotent (here we are using that \(e = ce = ec\)).
  Moreover, \(x \in et = et^*t\), as \(x \in e\) and \(x \in t^*t\). But this implies that \(x \in O_{t}\), contradicting the assumption on \(x\).

  The ``in particular'' statement follows from \(\primidealspc{B}\) being a Baire space.
\end{proof}

The following shows our functors \(\maxalg[D_0]{\variable}\) and \(\redalg[D_0]{\variable}\) generalize the main one of \cites{Kwasniewski-Meyer:Pure_infiniteness,KwasniewskiMeyer2021,exel-pitts-grpds-2022}, at least when \(B\) is type I and \(S\) is quasi-countable.
The latter condition is necessary for the study of separable algebras.

\begin{theorem} \label{thm:ess-alg-is-jd}
  Let \(\CB = (B_s)_{s \in S}\) be a Fell bundle over some separable type I \cstar{}algebra \(B\), where \(S\) is quasi-countable.
  Let \(D = \bigcup_{s \in S} D_{s}\) be as in \eqref{eq:d-s}.
  The identity map \(\id \colon \redalg{\CB} \to \redalg{\CB}\) descends to a \Star{}isomorphism \(\redalg[D]{\CB} \cong \essalg{\CB}\), where the latter is defined as in \cite{Kwasniewski-Meyer:Pure_infiniteness}.
\end{theorem}

The proof of \cref{thm:ess-alg-is-jd} is somewhat technical, and will necessitate some previous discussion.
The \emph{Borel algebra of \(B\)}, denoted by \(\borelalg{B}\), was introduced in \cite{davies-1968} as a ``Borel counterpart'' of the, in general, much larger \(B^{**}\).
For now we introduce \(\borelalg{B}\) as a natural (norm closed) subset of \(B^{**}\). 

\begin{definition} [cf.\ \cite{davies-1968}*{Section 3}]
  The \emph{\(\sigma\)-envelope of \(B\)} is the smallest \(\sigma\)-closed family \(\borelalg{B} \subseteq B^{**}\) of bounded linear functionals on the states of \(B\) containing \(B\), that is, if \(b_n \to b\) for a sequence \(\{b_n\}_{n \in \NN} \subseteq B \subseteq B^{**}\) then \(b \in \borelalg{B}\). 
\end{definition}

By definition we have that \(\borelalg{B} \subseteq B^{**}\) is a norm-closed subalgebra.
Recall that the \emph{atomic representation} \(\lambda \colon B \to \CL(H)\) of a \cstar{}algebra \(B\) is the sum of all the irreducible representations of \(B\).
By \cite{davies-1968}*{Theorem 3.2} we can define \(\borelalg{B}\) as the \(\sigma\)-closure of \(\lambda(B)\), that is, if \(\lambda(b_n) \to c\) in the weak operator topology for a \emph{sequence} \(\{b_n\}_{n \in \NN} \subseteq B\) then \(c \in \lambda(\borelalg{B})\).
A convenient sub-product of the construction of \(\borelalg{B}\) is the following simple observation which translates, in the commutative world, to the well known fact that the pointwise limit of continuous functions is Borel.

\begin{lemma} \label{lemma:ptw-limit-cont-is-borel}
  If \(b, \{u_n\}_{n \in \NN} \subseteq B\) are such that \(b u_n \to c \in B^{**}\) weakly, then \(c \in \borelalg{B}\).
\end{lemma}
The remarkably silly phrasing of this lemma is due to how it is used, see the proof of \cref{lemma:p-takes-values-borel}.
\begin{proof}[Proof of \cref{lemma:ptw-limit-cont-is-borel}]
  Observe \(b u_n \in B\) and that \(\pi(bu_n) \to \pi(c)\) weakly for all irreducible representations \(\pi \colon B \to \CL(H)\).
  It thus follows from the definition of \(\borelalg{B}\) that the limit of \(\{b u_n\}_{n \in \NN}\) also belongs to \(\borelalg{B}\).
\end{proof}

Essentially by construction and the fact that if \(B\) is type I then the canonical map \(\widehat{B} \to \primidealspc{B}\) is a homeomorphism, see \cite{blackadar-book-2010}*{Theorem IV.1.5.7}, we have that if \(B\) is type I then the canonical map \(q \colon B^{**} \to \prod_{P \in \primidealspc{B}} (B/P)^{**}\) is isometric on \(\borelalg{B}\) (see also \cite{davies-1968}*{Theorem 3.2}).
In particular, if \(\CB = (B_s)_{s \in S}\) is a Fell bundle over \(B\) and \(\redcondexp(\redalg{\CB}) \subseteq \borelalg{B}\) then we could, without any problems, assume that the weak conditional expectation \(\redcondexp\) takes the form of
\[
  \redcondexp \colon \redalg{\CB} \to B^{**} \to \prod_{P \in \primidealspc{B}} \left(B/P\right)^{**}
\]
by post-composing with \(q \colon B^{**} \to \prod_{P \in \primidealspc{B}} (B/P)^{**}\).
This is the content of the following lemma.

\begin{lemma}[cf.\ \cite{KwasniewskiMeyer2021}*{Lemma 3.21}] \label{lemma:p-takes-values-borel}
  If \(B\) is separable then \(\redcondexp \colon \redalg{\CB} \to B^{**}\) takes values in \(\borelalg{B}\).
\end{lemma}
\begin{proof}
  As \(\borelalg{B}\) is norm closed and \(\redcondexp\) is continuous it suffices to show that \(\redcondexp(b_s) \in \borelalg{B}\) for all monomials \(b_s \in B_s\) and \(s \in S\).
  For such monomials we have that \(\redcondexp(b_s) = b_s 1_s\) by definition of \(\redcondexp\).
  Recalling that \(1_s\) is the unit of \(I_{s,1}^{**}\) it is clear that \(1_s\) can be approximated strongly (hence also weakly) by elements \(\{u_{s, n}\}_{n \in \NN} \subseteq B\).
  The claim that \(b_s 1_s \in \borelalg{B}\) then follows from \cref{lemma:ptw-limit-cont-is-borel}.
\end{proof}

\begin{proof}[Proof of \cref{thm:ess-alg-is-jd}]
  We first recall the construction of \(\essalg{\CB}\).
  For this, we will combine \cite{Kwasniewski-Meyer:Pure_infiniteness} and \cite{AraMathieu2003}*{Theorem 2.9.3}.
  Recall that an ideal \(L \subseteq B\) is \emph{essential} if it intersects any other non-trivial ideal.
  Let
  \(
    J_{\ess} \coloneqq \closure{\left\{b \in B^{**} : b p_L = 0 \; \text{ for some essential ideal } \, L \subseteq B\right\}},
  \)
  where \(p_L \in B^{**}\) is the open central projection corresponding to \(L\).
  Similarly to \cref{prop:cond-exp-kills-dzero}, we denote by \(E_\CL\) the \emph{essential conditional expectation}, that is, the map
  \[
    E_\CL \colon \redalg{\CB} \xrightarrow{\redcondexp} B^{**} \to B^{**}/J_{\ess}, \;\; a \mapsto \redcondexp\left(a\right) + J_{\ess}.
  \]
  The essential \cstar{}algebra \(\essalg{\CB}\) is nothing but \(\redalg{\CB}\) modulo the nucleus of \(E_\CL\), \emph{i.e.}\ the ideal \(\{a : E_\CL(a^*a) = 0\} = \{a : \redcondexp(a^*a) \in J_{\ess}\}\).
  By \cref{lemma:p-takes-values-borel} and the discussion immediately prior to it, the map \(\widetilde{\redcondexp}\) defined by
  \[
    \widetilde{\redcondexp} \colon \redalg{\CB} \xrightarrow{\redcondexp} B^{**} \xrightarrow{q} \prod_{P \in \primidealspc{B}} \left(B/P\right)^{**}
  \]
  is faithful.
  Indeed, it is the composition of a faithful map and a map that is isometric on the image of the first one.
  For convenience let \(X \coloneqq \primidealspc{B}\), whose points will be denoted by \(x, y \in X\).
  Likewise, let \(B_x\) be the corresponding quotient \(B/P\) (where \(x = P\)).
  We wish to induce the expectation \(E_\CL\) down to \(\prod_{x \in X} B_x^{**}\).
  Similarly to the definition of \(J_{\ess}\) above, let \(\widetilde{J} \subseteq \prod_{x \in X} B_x^{**}\) be the ideal generated by all tuples \(b = (b_x)_{x \in X}\) with meager support, \emph{i.e.}\ \(\{x : b_x \neq 0\} \subseteq X\) is meager.
  Then let \(\widetilde{E_\CL}\) be the generalized conditional expectation
  \[
    \widetilde{E_\CL} \colon \redalg{\CB} \xrightarrow{P} B^{**} \xrightarrow{q} \prod_{x \in X} B_x^{**} \to \left(\prod_{x \in X} B_x^{**}\right)/\widetilde{J}.
  \]
  Likewise, note that the generalized conditional expectation \(\condexp[D]\) also ``lives'' in \(\prod_{x \in X} B_x^{**}\), and hence it makes sense to compare \(\widetilde{E_\CL}\) and \(\condexp[D]\).
  We claim that the nucleus of \(E_\CL, \widetilde{E_\CL}\) and \(\condexp[D]\) are all equal.
  Since the notation may be overwhelming, we hope the commuting diagram helps:
  \begin{equation} \label{eq:diagram:e-l}
    \begin{tikzcd}[scale=50em]
        \redalg{\CB} \arrow{r}{P} \arrow{dr}{E_\CL} & B^{**} \arrow{r}{q} \arrow{d}{} & \prod_{x \in X} B_x^{**} \arrow{r}{} \arrow{d}{} & \prod_{x \in X \setminus D} B_x^{**} \\
        & B^{**}/J_{\ess} \arrow{r}{q_\CL} & \left(\prod_{x \in X}B_x^{**}\right)/\widetilde{J}. & 
    \end{tikzcd}
  \end{equation}
  The nameless arrows are the canonical quotient maps.
  The top horizontal row is precisely \(\condexp[D]\); and \(\widetilde{E_\CL} = q_\CL \circ E_\CL\).
  We also note that the map \(q_\CL\) exists because \(q(J_{\ess}) \subseteq \widetilde{J}\).
  In order to check this, let \(y \in J_{\ess}\) be some element such that \(y p_L = 0\) for some essential ideal \(L \subseteq B\), where \(p_L \in B^{**}\) is the open central projection corresponding to \(L\).
  As \(L\) is essential we have that \(\primidealspc{L} \subseteq X\) is open and dense, so that \(\supp{y} = \supp{y (1-p_L)} \subseteq \supp{1-p_L} = X \setminus \primidealspc{L}\) is meager, as desired.

  The diagram \eqref{eq:diagram:e-l} essentially shows that the nucleus of \(E_\CL\) and \(\widetilde{E_\CL}\) are equal.
  Indeed, \(q\) is isometric on the image of \(\redcondexp\), cf.\ \cref{lemma:p-takes-values-borel}, so \(q_\CL\) is isometric on the image of \(E_\CL\).
  Thus, in order to prove that \(\redalg[D]{\CB} \cong \essalg{\CB}\) it suffices to show that the nucleus of \(\widetilde{E_\CL}\) and \(\condexp[D]\) are equal.
  For this, we will show that
  \begin{equation} \label{eq:want-ess}
    \left(q \circ \redcondexp\right)\left(\redalg{\CB}\right) \cap J_{D} = \left(q \circ P\right)\left(\redalg{\CB}\right) \cap \widetilde{J}.
  \end{equation}
  In order to simplify notation (and since it does not really matter, see \cref{lemma:p-takes-values-borel} and the immediately prior discussion), we will assume that \(\redcondexp = \widetilde{\redcondexp}\) already takes values in \(\prod_{x \in X} B_x^{**}\).
  Note that we have to intersect with the image of \(\redcondexp\) on both sides of \eqref{eq:want-ess}, otherwise the statement would be false.

  We start proving ``\(\subseteq\)'' in \eqref{eq:want-ess}.
  Let \(\redcondexp(a) \in J_D\) be given.
  This precisely means that \(\redcondexp(a)(x) = 0\) when \(x \not\in D\).
  We wish to show that \(\redcondexp(a) \in \widetilde{J}\), that is, we wish to show that \(\{x : \redcondexp(a)(x) \neq 0\}\) is meager in \(X\).
  But this is clear, for the former is contained in \(D\) by the assumption on \(a\), and \(D \subseteq X\) is meager by construction, since \(S\) is quasi-countable, see \cref{lemma:d-inv-countable}.

  We turn to the proof of ``\(\supseteq\)'' in \eqref{eq:want-ess}.
  Take some contraction \(a \in \redalg{\CB}\) such that \(\redcondexp(a) \in \widetilde{J}\).
  The latter precisely means that \(\supp{\redcondexp(a)} \subseteq X\) is meager.
  As \(X = \primidealspc{B}\) is Baire we have that any co-meager set, such as \(X \setminus \supp{\redcondexp(a)}\), is dense.
  It thus suffices to the prove the following.

  \begin{claim}
    If the complement of \(\supp{\redcondexp(a)}\) in \(X\) is dense then \(\redcondexp(a) \in J_D\).
  \end{claim}

  \begin{proof}
    The proof of this claim has a similar flavor to some arguments in a previous version of \cite{buss-martinez-ess-ame-2024}*{Section 3}, and has to do with continuity.
    Take some \(x \in X\) such that \(\redcondexp(a)(x) \neq 0\), \emph{i.e.}\ \(x \in \supp{\redcondexp(a)}\).
    By the density assumption there is some net \(\{x_\lambda\}_\lambda \subseteq X\) such that \(x_\lambda \to x\) and \(\redcondexp(a)(x_\lambda) = 0\).
    Expressing \(a = \sum_{s \in S} a_s\) for some \(a_s \in B_s\), we have that \(\redcondexp(a) = \sum_{s \in S} a_s 1_s\).
    Hence \(\redcondexp(a)(x) = \sum_{s : x \in O_{s}} (a_s 1_s)(x)\), where \(O_s \coloneqq \cup_{e \in E, e \leq s} \primidealspc{B_e}\), which we see as an open subset of \(X\) (in the same manner as in \cref{lemma:d-inv-countable}).
    Likewise, \((a_s 1_s)(x)\) is defined to be the image of \(a_s 1_s \in B^{**}\) under the map \(B^{**} \to B_x^{**}\).
    By the previous discussion we have that
    \begin{equation} \label{eq:cont-at-x}
      \redcondexp\left(a\right)\left(x\right) = \sum_{s : x \in O_s} \left(a_s 1_s\right)\left(x\right) \neq 0 = \sum_{s : x_\lambda \in O_s} \left(a_s 1_s\right)\left(x_\lambda\right) = \redcondexp\left(a\right)\left(x_\lambda\right).\footnote{\, Note we are being sloppy, as the left hand side of \eqref{eq:cont-at-x} happens in fiber \(x\), whereas the right hand one happens in fiber \(x_\lambda\).
    Nevertheless, we are only saying one side is \(0\) and the other is not, so the meaning should be clear.}
    \end{equation}
    It hence follows that \(x\) is a ``point of discontinuity'' of some \(a_{s_0} 1_{s_0}\), \emph{i.e.}\ there is some \(s_0 \in S\) such that \((a_{s_0} 1_{s_0})(x) \neq 0 = (a_{s_0} 1_{s_0})(x_\lambda)\).
    Note this implies that \(X \ni y \mapsto \norm{(a_{s_0} 1_{s_0})(y)} \in \RR_{\geq 0}\) is not lower semi-continuous at \(x\), immediately implying, by \cite{blackadar-book-2010}*{Proposition II.6.5.6}, that \(a_{s_0} 1_{s_0} \not\in B\).
    However this is not enough: we have to explicitly use that \((a_{s_0} 1_{s_0})(x) \neq 0 = (a_{s_0} 1_{s_0})(x_\lambda)\).
    First of all this non-continuity implies \(x \in s_0^*s_0\). Secondly \(x \not\in O_{s_0}\), for if that were the case then \(x_\lambda \in O_{s_0}\) by openness of the latter, and then \(a_{s_0}\) would be continuous at \(x\), in the sense that \eqref{eq:cont-at-x} would not hold.
    Thirdly, \(x \in \closure{O_{s_0}}\), for, again, otherwise both terms of \eqref{eq:cont-at-x} would be \(0\).
    Whence \(x \in s_0^*s_0 \cap \closure{O_{s_0}} \setminus O_{s_0} = D_{s_0} \subseteq D\), as desired.
  \end{proof}

  By the previous discussions it follows that \eqref{eq:want-ess} holds whenever \(B\) is a type I separable \cstar{}algebra.
  As the left and right hand sides of \eqref{eq:want-ess} are exactly the nucleus of \(E_D\) and \(\widetilde{E_\CL}\) the identity map \(\id \colon \redalg{\CB} \to \redalg{\CB}\) descends to a \Star{}isomorphism \(\redalg[D]{\CB} \cong \essalg{\CB}\), finishing the proof.
\end{proof}

In light of \cref{thm:ess-alg-is-jd} we give the following.
\begin{definition}
  \(\essmaxalg{\CB}\) is the \cstar{}algebra \(\maxalg[D]{\CB}\), where \(D \subseteq \primidealspc{B}\) is as in \eqref{eq:d-s}.
\end{definition}

The following will be expanded in \cref{sec:apps}.
\begin{example}
  Let \(G\) be an \'etale groupoid whose unit space \(X \coloneqq G^{(0)}\) is Hausdorff.
  Let \(S \coloneqq {\rm Bis}(G)\) be the inverse semigroup of open bisections of \(G\).
  The canonical action \(S \actson X\) has \(G\) as its groupoid of germs (cf.\ \cref{sec:apps}).
  If we denote by \(\cfunz{X}\) the Fell bundle associated to \(S \actson \contz{X}\) via \eqref{eq:fell-bundle-from-action} then \cref{thm:ess-alg-is-jd} shows that \(\essmaxalg{G} \cong \maxalg[D]{\cfunz{X}}\), where \(D \subseteq X\) is the set of dangerous units, and \(\essmaxalg{G}\) is defined as in \cite{buss-martinez-ess-ame-2024}.
\end{example}

We also point the following out, which we hope will enlighten the notions we are studying.
\begin{example}
  Let \(\Gamma\) be a discrete group acting by automorphisms on a \cstar{}algebra \(A\). By means of \eqref{eq:fell-bundle-from-action} one may construct a Fell bundle \(\CA\), and then \(\maxalg{\CA} \cong A \rtimes_{\full} \Gamma\), and similarly for their reduced counterparts.
  In fact, these isomorphisms are extensions of a canonical algebraic isomorphism at the level of dense sub-algebras.

  On the other hand, suppose \(A \cong \contz{X}\), where \(X\) is a locally compact and Hausdorff space. Fix some invariant \(D_0 \subseteq X\).
  Then the \cstar{}algebra \(\maxalg[D_0]{\CA}\) is just the maximal \cstar{}algebra of the groupoid \(\Gamma \ltimes X\) restricted to \(X \setminus D_0\) (which fails to be locally compact should \(D_0\) be dense).
  Likewise, \(\redalg[D_0]{\CA}\) is the (in general only Hausdorff) completion of the canonical dense sub-algebra \(\algalg{\CA} \cong \contz{X} \rtimes_{\alg} \Gamma\) under the norm induced from the representation
  \[
    \bigoplus_{x \in X \setminus D_0} \lambda_x \colon \contz{X} \rtimes_{\alg} \Gamma \to \bigoplus_{x \in X \setminus D_0} \CB\left(\ell^2\left(\Gamma\right)\right).
  \]
\end{example}

\begin{example}
  If \(S = E\) is an inverse semigroup containing only idempotents then one can check that \(\maxalg{\CB} \cong \redalg{\CB} \cong B\) for any Fell bundle \(\CB = (B_s)_{s \in S}\) over a \cstar{}algebra \(B\).
  If \(D_0 \subseteq X\) is non-empty then \(\maxalg[D_0]{\CB}\) is the image of the map \(B \to \prod_{x \in X \setminus D_0} B_x\) which, should \(D_0\) be (topologically) large, may fail to be isometric.
\end{example}

\section{Equivariance, exact sequences and functoriality} \label{sec:functoriality}
We now study functoriality properties of the crossed products \(\maxalg[D_0]{\variable}\) and \(\redalg[D_0]{\variable}\) introduced before.
Notice that by the discussion in \cref{sec:apps-ess} these functoriality properties immediately have applications for \(\essalg{\CB}\).
We are mostly interested in when extensions \(0 \to \CI \to \CA \to \CB \to 0\) of Fell bundles (which we regard as \((S, X)\)-\cstar{}algebras for a common \(X\)) give rise to extensions
\[
  0 \to \maxalg[D_0]{\CI} \to \maxalg[D_0]{\CA} \to \maxalg[D_0]{\CB} \to 0
\]
(and similarly for the reduced crossed counterparts \(\redalg[D_0]{\variable}\)).
This subsection uses and/or generalizes certain results in \cite{Kwasniewski-Meyer:Pure_infiniteness}*{Section 4}.
In fact, the case when \(D_0\) is empty is contained (modulo \emph{asymptotic} morphisms) in \cite{Kwasniewski-Meyer:Pure_infiniteness}*{Proposition 4.15}.
On the other hand, the cases when \(D_0\) is non-empty require more care and, therefore, it is best we do them by hand.
Thus, fix some Fell bundles \(\CI, \CA, \CB\) over \(S\), which we regard as \((S, X)\)-algebras for a common \(X\).
Observe that the structure maps
\[
  \Phi_I, \Phi_A, \Phi_B \colon \opensets{X} \to \idealsin{I}, \idealsin{A}, \idealsin{B}
\]
are all equivariant in the sense of \eqref{eq:equiv-struct-map}.
We start recalling the notion of non-asymptotic morphism, see \cite{Kwasniewski-Meyer:Pure_infiniteness}*{Definition 4.1}.

\begin{definition} \label{def:non-asymp-morphism}
  A \emph{morphism} \(\varphi \colon \CA \to \CB\) is a collection of linear maps \(\varphi_s \colon A_s \to B_s\) such that \(\varphi_{s_1s_2}(a_{s_1} a_{s_2}) = \varphi_{s_1}(a_{s_1}) \varphi_{s_2}(a_{s_2})\) for all \(a_{s_1} \in A_{s_1}, a_{s_2} \in A_{s_2}\) and \(s_1, s_2 \in S\).
  We say \(\varphi\) is \emph{reduced} if \(\varphi_1(I_{s,1}^\CA) \supseteq \varphi_{s^*s}(A_{s^*s}) \cdot I_{s,1}^\CB\) for all \(s \in S\), where \(I_{s,1}^\CA\) is the ideal in \eqref{eq:i-s-1} for \(\CA\), and similarly for \(\CB\).
\end{definition}

\begin{example}
  The easiest case of \cref{def:non-asymp-morphism} is that of a discrete group \(\Gamma\) acting on a \cstar{}algebra \(A\) by automorphisms.
  One can then construct an associated Fell bundle \(\CA\) via \eqref{eq:fell-bundle-from-action}.
  In this setting, a \Star{}homomorphism \(\varphi \colon A \to B\) between \(\Gamma\)-\cstar{}algebras is equivariant if and only if it induces a morphism \(\CA \to \CB\) between the associated Fell bundles in the sense of \cref{def:morphism}.
  Moreover, such a morphism is automatically reduced, for \(I_{\gamma,1_\Gamma}^\CB = 0\) unless \(\gamma = 1_\Gamma \in \Gamma\), in which case \(I_{1_\Gamma,1_\Gamma}^\CA = A\).
\end{example}

\begin{example} \label{ex:opensets-x}
  On the other end of the spectrum, let \(X\) be (as usual) some locally compact, Baire, \tzero space. Let \(S \coloneqq \opensets{X}\), equipped with the operation \(\CU \cdot \CV \coloneqq \CU \cap \CV\).
  Then every element in \(S\) is a projection, so that \(E = S\).
  An \(S\)-\cstar{}algebra then reduces to a \cstar{}algebra \(A\) equipped with some order preserving structure map \(\Phi_A \colon \opensets{X} = E \to \idealsin{A}\), \emph{i.e.}\ it is an \(X\)-algebra in the sense of Gabe~\cite{gabe-memoirs-2024}, see \cref{def:x-algebra}.
  In the same vein, a morphism \(\varphi \colon A \to B\), in the sense of \cref{def:morphism}, is just a \Star{}homomorphism \(\varphi \colon A \to B\) that is \(X\)-equivariant in the sense of \cite{gabe-memoirs-2024}*{Definition 10.7}:
  \[
    \varphi\left(\Phi_A\left(\CU\right)\right) \subseteq \Phi_B\left(\CU\right)
  \]
  for all \(\CU \in S = \opensets{X}\).
  Again, such a morphism is also automatically reduced. 
\end{example}

The main functoriality properties of morphisms \(\CA \to \CB\) in the sense of \cref{def:non-asymp-morphism} is that surjective morphisms always behave nicely.

\begin{proposition}[see \cite{Kwasniewski-Meyer:Pure_infiniteness}*{Proposition 4.2 and Remark 4.6}] \label{prop:technical-red-map}
  Any surjective morphism \(\varphi \colon \CA \to \CB\) between Fell bundles is reduced, and induces \Star{}homomorphisms
  \[
    \maxalg[\emptyset]{\varphi} \colon \maxalg[\emptyset]{\CA} \to \maxalg[\emptyset]{\CB} \;\; \text{ and } \;\; \redalg[\emptyset]{\varphi} \colon \redalg[\emptyset]{\CA} \to \redalg[\emptyset]{\CB},
  \]
  where \(a_s \delta_s \mapsto \varphi_s(a_s) \delta_s\) for all \(a_s \in A_s\).
\end{proposition}
\begin{proof}
  This is a consequence of \cite{Kwasniewski-Meyer:Pure_infiniteness}*{Proposition 4.2 and Remark 4.6}: by \cite{Kwasniewski-Meyer:Pure_infiniteness}*{Remark 4.6} the condition that \(\varphi\) be surjective is sufficient to guarantee that \cite{Kwasniewski-Meyer:Pure_infiniteness}*{Proposition 4.2 (2)} is met, \emph{i.e.}\ \(\varphi\) is reduced in the sense of \cref{def:non-asymp-morphism}.
\end{proof}

The notion of reduced morphism in \cref{def:non-asymp-morphism} is needed to guarantee that the induced \Star{}homomorphism descends to \(\redalg{\variable}\), as \(\varphi \colon \CA \to \CB\) always induces a \Star{}homomorphism on the level of \(\maxalg{\variable}\).
Thus, this is only to warn the reader that \emph{some} condition is needed to induce a \Star{}homomorphism, cf.\ \cref{prop:inject-equiv-map-induces}. 
We go onto the vastly more subtle study of asymptotic morphisms.

\begin{definition} \label{def:morphism}
  An \emph{asymptotic morphism} \(\varphi \colon \CA \asymparrow \CB\) is a collection of maps \(\varphi_s \colon A_s \asymparrow B_s\) such that \(\varphi_{s^*s} = (\varphi_{s^*s,r})_{r \in [1, \infty)} \colon A_{s^*s} \asymparrow B_{s^*s}\) is a usual asymptotic morphism for all \(s^*s \in E\) and \(\varphi_s = (\varphi_{s,r})_{r \in [1, \infty)} \colon A_s \asymparrow B_s\) is a family of  maps such that
  \begin{itemize}
    \item \(\varphi_{s,r}(a_s + b_s) - \varphi_{s,r}(a_s) - \varphi_{s,r}(b_s)\);
    \item \(\varphi_{s_1s_2, r}(a_{s_1} a_{s_2}) - \varphi_{s_1,r}(a_{s_1}) \varphi_{s_2,r}(a_{s_2})\);
    \item \(\varphi_{s,r}(a_s) - \varphi_{t,r}(a_s)\), where \(s \leq t\) and \(a_s \in A_s \subseteq A_t\);
  \end{itemize}
  all go to \(0\) when \(r \to \infty\), for all variables appearing.
  We say a morphism \(\varphi\) is \emph{\(D_0\)-equivariant} if \(\varphi_1^{**}(J_{D_0, \CA}) \subseteq J_{D_0, \CB}\), where \(\varphi_1^{**} \colon A^{**} \to B^{**}\) is the unique normal extension of \(\cup_{e \in E} \varphi_e \colon \spn{\cup_{e \in E} A_e} \to B\) and \(J_{D_0,\CA}, J_{D_0,\CB}\) are as in \cref{prop:cond-exp-kills-dzero} for the Fell bundles \(\CA\) and \(\CB\) respectively.
\end{definition}

The notion of \(D_0\)-equivariant morphism, see \cref{def:morphism}, only works for actual morphisms, as opposed to \emph{asymptotic} ones.
However, in future applications we will need to construct elements in \Eth-theory on the level of \(\maxalg[D_0]{\variable}\) and \(\redalg[D_0]{\variable}\), so we do need to generalize this notion to \emph{\(D_0\)-equivariant asymptotic morphisms}.
For this, it is worth remembering that an asymptotic morphism \(\varphi = (\varphi_r)_{r \in [1,\infty)} \colon A \asymparrow B\) between usual \cstar{}algebras can be regarded as a \Star{}homomorphism \(\varphi_\fc \colon A \to \contb{[1, \infty), B}/\contz{[1, \infty), B}\).
For simplicity we let the latter \cstar{}algebra be
\[
  B_\fc \coloneqq \contb{\left[1, \infty\right), B}/\contz{\left[1, \infty\right), B}.\footnote{\, We would like to thank G\'abor Szab\'o for this phenomenal naming convention. Also, \(\contb{\left[1,\infty\right), B}\) denotes the \cstar{}algebra of continuous bounded functions \([1,\infty) \to B\).}
\]
We say an asymptotic morphism \(\varphi \colon A \asymparrow B\) is \emph{injective} if the induced map \(\varphi_\fc \colon A \to B_\fc\) is an embedding.
The following is a generalization of \cref{prop:technical-red-map} to asymptotic morphisms and generalized crossed products \(\maxalg[D_0]{\variable}\) and \(\redalg[D_0]{\variable}\).
For it, we denote by \(\varphi_{1,\fc} \colon A \to B_\fc\) the unique norm-continuous extension of \(\cup_{e \in E} \varphi_e \colon \spn{\cup_{e \in E} A_e} \to B_\fc\).
Similarly, we denote by \(\varphi_{1,\fc}^{**} \colon A^{**} \to B_\fc^{**} \coloneqq (B_\fc)^{**}\) its unique normal extension to the double duals.

The following lengthy statement encapsulates the technicality needed.
\begin{proposition} \label{prop:inject-equiv-map-induces}
  Let \(\varphi \colon \CA \asymparrow \CB\) be an asymptotic morphism of Fell bundles.
  \begin{enumerate}
    \item \(\varphi\) induces an asymptotic morphism \(\maxalg{\varphi} \colon \maxalg{\CA} \asymparrow \maxalg{\CB}\).
    \item \label{prop:inject-equiv-map-induces:red} If \(\varphi\) is either a surjective morphism or \(\varphi_{e,\fc}(A_e) \supseteq B_e \cdot \varphi_{s^*s,\fc}(A_{s^*s})\) for all \(s \in S\) and \(e \in E\) with \(e \leq s\), then \(\varphi\) induces an asymptotic morphism \(\redalg{\varphi} \colon \redalg{\CA} \asymparrow \redalg{\CB}\).
    \item \label{prop:inject-equiv-map-induces:dzero-equiv} If \(\varphi\) is ``\(D_0\)-equivariant'', in the sense that
      \[
        \varphi_{1,\fc}^{**}\left(J_{D_0}^\CA\right) \subseteq J_{D_0,\fc}^\CB,
      \]
      then \(\varphi\) also induces an asymptotic morphism \(\maxalg[D_0]{\varphi} \colon \maxalg[D_0]{\CA} \asymparrow \maxalg[D_0]{\CB}\).
    \item \label{prop:inject-equiv-map-induces:dzer-equi-red} If \(\varphi\) is \(D_0\)-equivariant and \(\varphi_{e,\fc}(A_e) \supseteq B_e \cdot \varphi_{s^*s,\fc}(A_{s^*s})\) for all \(s \in S\) and \(e \in E\) with \(e \leq s\), then \(\varphi\) also induces an asymptotic morphism \(\redalg[D_0]{\varphi} \colon \redalg[D_0]{\CA} \asymparrow \redalg[D_0]{\CB}\).
    \item \label{prop:inject-equiv-map-induces:injective} In the setting of \cref{prop:inject-equiv-map-induces:dzer-equi-red}, if \(\varphi\) is an injective morphism and \(\varphi(A)\) is an ideal of \(B\) then \(\maxalg[D_0]{\varphi}\) and \(\redalg[D_0]{\varphi}\) are injective.
    \item \label{prop:inject-equiv-map-induces:morphisms} If \(\varphi\) is a \(D_0\)-equivariant morphism (resp.\ surjective morphism) then so are \(\maxalg[D_0]{\varphi}\) and \(\redalg[D_0]{\varphi}\).
  \end{enumerate}
\end{proposition}

In lieu of the yet-to-be-proved \cref{prop:inject-equiv-map-induces} we tacitly give the following.
\begin{definition} \label{def:all-equiv-morphism}
  An asymptotic morphism \(\varphi \colon \CA \asymparrow \CB\) is \emph{\(D_0\)-equivariant} if \(\varphi_{1,\fc}^{**}\left(J_{D_0}^\CA\right) \subseteq J_{D_0,\fc}^\CB\), and it is \emph{reduced} if \(\varphi_{e,\fc}(A_e) \supseteq B_e \cdot \varphi_{s^*s,\fc}(A_{s^*s})\) for all \(s \in S\) and \(e \in E\) with \(e \leq s\).
\end{definition}

\begin{proof}[Proof of \cref{prop:inject-equiv-map-induces}]
  The first part of the proof is in spirit mighty similar to \cite{KwasniewskiMeyer2021}*{Proposition 4.2}.
  Let \(\varphi = (\varphi_{s,r})_{s \in S, r \in [1,\infty)}\) be the starting asymptotic morphism.
  We will first prove the existence of \(\maxalg{\varphi} = \maxalg[\emptyset]{\varphi}\).
  In order to do this, recall that the Banach spaces \(B_s\) isometrically embed into \(\maxalg{\CB}\) for every \(s \in S\).
  In particular it follows that \(\varphi\) gives a representation
  \[
    \left(\varphi_s\right)_{s \in S} \colon \CA \to \maxalg{\CB}_\fc.
  \]
  By the universal property of \(\maxalg{\CA}\) the tuple \((\varphi_s)_{s \in S}\) integrates to a \Star{}homomorphism
  \(\widehat{\varphi} \colon \maxalg{\CA} \to \maxalg{\CB}_\fc\).
  Thus, \(\widehat{\varphi}\) can be seen as an asymptotic morphism \(\maxalg{\varphi} \coloneqq \widehat{\varphi} \colon \maxalg{\CA} \asymparrow \maxalg{\CB}\).

  \

  We now prove that \(\maxalg{\varphi}\) descends to all the asymptotic morphisms we are looking for.
  We first show it descends to the reduced crossed products as in \cref{prop:inject-equiv-map-induces:red}.
  By the definition of \(\redalg{\variable}\) it suffices to show that
  \begin{equation} \label{eq:want-red-nn}
    \maxalg{\varphi}\left(\CN_\CA\right) \subseteq \left(\CN_\CB\right)_\fc,\footnote{\, Throughout this proof, if \(C \subseteq B\) is some subset then \(C_\fc\) is the canonical copy of \(C\) embedded in \(B_\fc\), whose elements are (classes of) constant nets \([c]_{t \in [1,\infty)}\), with \(c \in C\).}
  \end{equation}
  where \(\CN_\CA \coloneqq \{a \in \maxalg{\CA} : \redcondexp_\CA(a^*a) = 0\}\) is the nucleus of the usual weak conditional expectation \(\redcondexp_\CA \colon \maxalg{\CA} \to A^{**}\), and similarly for the Fell bundle \(\CB\).\footnote{\, Technically speaking, we usually regard \(\redcondexp_\CA\) as a map \(\redalg{\CA} \to A^{**}\). Here we pre-compose it with the canonical quotient \(\maxalg{\CA} \to \redalg{\CA}\). We thus slightly abuse notation, as when talking about their nucleus we explicitly write where the maps take values from. Whence causing no confusion.}
  Recall the definition of the ideals \(I_{s,t}^\CA \subseteq A\) from \cref{def:ideals-i-s-t} (we use the superscript to distinguish between the ones from \(\CA\) and the ones from \(\CB\)).
  Likewise, we let \(1_{s,t}^\CA\) and \(1_{s,t}^\CB\) be the units of \(\multialg{I_{s,t}^\CA}\) and \(\multialg{I_{s,t}^\CB}\) respectively.
  We see these algebras embedded into \(A^{**}\) and \(B^{**}\) respectively in the canonical way.
  We will first show that a similar statement to \cite{Kwasniewski-Meyer:Pure_infiniteness}*{Proposition 4.2 (2)} is met, that is, we will prove
  \begin{equation} \label{eq:varphi-sends-i-st-a-to-b}
    \varphi_{1,\fc}^{**}(1_{1,s}^\CA) = \varphi_{1,\fc}^{**}(1_{1,s^*s}^\CA) \cdot 1_{1,s}^{\CB} \in B^{**}_\fc.
  \end{equation}
  As mentioned in \cite{Kwasniewski-Meyer:Pure_infiniteness}*{Remark 4.6} the inequality ``\(\leq\)'' is always satisfied.
  Nevertheless, we show this here as well for convenience of the reader, and because their notations and ours are vastly different.
  By normality of \(\varphi_{1,\fc}^{**}\) we have that
  \begin{equation} \label{eq:i-dont-have-more-names}
    \varphi_{1,\fc}^{**}\left(1_{1,s}^\CA\right) \stackrel{\rm (def)}{=} \varphi_{1,\fc}^{**} \left(\vee_{e \in E, e \leq s} 1_{1,e}^\CA\right) = \vee_{e \in E, e \leq s} \, \varphi_{1,\fc}^{**} \left(1_{1,e}^\CA\right) \leq \vee_{e \in E, e \leq s} \, 1_{1,e}^\CB = 1_{1,s}^\CB.
  \end{equation}
  Similarly, \(\varphi_{1,\fc}^{**}(1_{1,s}^\CA) \leq \varphi_{1,\fc}^{**}(1^\CA_{1,s^*s}) \leq 1^\CB_{1,s^*s}\).
  (In these equations we see \(1_{1,s}^\CB \in B^{**}_\fc\) as classes of constant nets.)
  The inequality ``\(\geq\)'' follows from the surjectivity of \(\varphi\) or the other assumption in \cref{prop:inject-equiv-map-induces:red}.
  Indeed, if \(\varphi\) is a surjective (non-asymptotic) morphism then \(\varphi_{1}^{**}(1_{1,e}^\CA) = 1_{1,e}^\CB\) for all \(e \in E\).
  This is due to the fact that \(\varphi_e \colon A_e \to B_e\) is a surjective \Star{}homomorphism.
  Thus, it follows that
  \[
    \varphi_{1}^{**}\left(1_{1,s}^\CA\right) = \vee_{e \in E, e \leq s} \varphi_{1}^{**}\left(1_{1,e}^\CA\right) = \vee_{e \in E, e \leq s} 1_{1,e}^\CB = 1_{1,s}^\CB.
  \]
  In the remaining case, the only non-equality in \eqref{eq:i-dont-have-more-names} (when intersected with \(\varphi_{s^*s,\fc}(A_{s^*s})\)) will actually be an equality if \(\varphi_{e,\fc}(A_e) \supseteq B_e \cdot \varphi_{s^*s,\fc}(A_{s^*s})\) for all \(e \in E\) such that \(e \leq s\), but this is precisely the condition appearing in item \cref{prop:inject-equiv-map-induces:red}.
  This finishes the proof of \eqref{eq:varphi-sends-i-st-a-to-b}.
  Notice that by \eqref{eq:varphi-sends-i-st-a-to-b} it follows that the diagram
  \begin{equation*}
    \begin{tikzcd}
      \maxalg{\CA} \arrow{r}{\maxalg{\varphi}} \arrow{d}{\redcondexp_\CA} & \maxalg{\CB}_\fc \arrow{d}{\redcondexp_{\CB,\fc}} \\
      A^{**} \arrow{r}{\varphi_{1,\fc}^{**}} & B_\fc^{**}
    \end{tikzcd}
  \end{equation*}
  commutes.
  In particular it follows that the image of \(\CN_\CA\) under \(\maxalg{\varphi}\) has to be contained in the nucleus of \(\redcondexp_{\CB,\fc}\), which is precisely what \eqref{eq:want-red-nn} states.
  Thus item \cref{prop:inject-equiv-map-induces:red} holds, as desired.

  \

  Knowing \cref{prop:inject-equiv-map-induces:red}, \emph{i.e.}\ that \(\maxalg{\varphi}\) descends to \(\redalg{\varphi} \colon \redalg{\CA} \asymparrow \redalg{\CB}\), we go directly to \cref{prop:inject-equiv-map-induces:dzer-equi-red}.
  For this, we must show that \(\varphi\) descends even further to \(\redalg[D_0]{\CA} \asymparrow \redalg[D_0]{\CB}\), for which it is enough to show that
  \begin{equation} \label{eq:want-red-dzero}
    \redalg{\varphi}\left(\left\{a \in \redalg{\CA} : \redcondexp_\CA\left(a^*a\right) \in J_{D_0}^\CA\right\}\right) \subseteq \left\{b \in \redalg{\CB}_\fc : \redcondexp_{\CB,\fc}\left(b^*b\right) \in J_{D_0,\fc}^\CB\right\}
  \end{equation}
  (at least modulo \(\contz{[1,\infty), \redalg{\CB}}\)).
  Instead of doing this directly, observe that in order to show \eqref{eq:want-red-dzero} it is enough to recall that \(D_0\)-equivariance precisely means that \(\varphi_{1,\fc}^{**}(J_{D_0}^\CA) \subseteq J_{D_0,\fc}^{\CB}\).
  This implies that there is an induced map \(A^{**}/J_{D_0}^\CA \to B^{**}_\fc/J_{D_0,\fc}^\CB\) making that the diagram
  \begin{equation} \label{eq:red-algs-diagram}
    \begin{tikzcd}
      \redalg{\CA} \arrow{r}{\Lambda^{\CA}_{D_0}} \arrow{d}{\redalg{\varphi}} & \redalg[D_0]{\CA} \arrow{r}{\condexp_{D_0}^\CA} & A^{**}/J_{D_0}^\CA \arrow{d}{} \\
      \redalg{\CB}_\fc \arrow{r}{\Lambda^{\CB}_{D_0,\fc}} & \redalg[D_0]{\CB}_\fc \arrow{r}{\condexp_{D_0,\fc}^\CB} & B_\fc^{**}/J_{D_0,\fc}^\CB
    \end{tikzcd}
  \end{equation}
  commute. In order to prove \eqref{eq:want-red-dzero} let us take some \(a^*a \in \redalg{\CA}\) such that \(\redcondexp_\CA(a^*a) \in J_{D_0}^\CA\).
  By the commutativity of \eqref{eq:red-algs-diagram} we have that \(\Lambda_{D_0,\fc}^\CB(\redalg{\varphi}(a^*a)) = 0\) by faithfulness of \(\condexp[D_0]^\CB\).
  But this precisely means that \(\redalg{\varphi}(a^*a) \in \ker(\Lambda_{D_0,\fc}^\CB) = \{b \in \redalg{\CB}_\fc : \redcondexp_{\CB,\fc}(b^*b) \in J_{D_0,\fc}^\CB\}\), as desired.

  \
  We now come back to item \cref{prop:inject-equiv-map-induces:dzero-equiv}, and show that \(\maxalg{\varphi}\) also induces a map on the level of \(\maxalg[D_0]{\variable}\).
  For this, simply recall that \(\algalg{\CA}\) is naturally a dense subalgebra of \(\maxalg{\CA}\) (and similarly for \(\CB\)).
  Thus we have a canonical algebra morphism
  \[
    \algalg[D_0]{\varphi} \colon \algalg{\CA} \to \algalg{\CB}_\fc \subseteq \maxalg{\CB}_\fc \to \maxalg[D_0]{\CB}_\fc.
  \]
  By the universal property of \(\maxalg[D_0]{\CA}\) this induces a map \(\maxalg[D_0]{\varphi}\) as long as \(\algalg[D_0]{\varphi}\) is a \(D_0\)-representation of \(\CA\).
  This precisely asks whether
  \[
    \algalg[D_0]{\varphi}\left(\left\{a \in \algalg{\CA} : \redcondexp_{\CA}\left(a^*a\right) \in J_{D_0}^\CA\right\}\right) = 0.
  \]
  But notice that this exactly happens by an argument similar to the one around diagram \eqref{eq:red-algs-diagram}.

  \

  Lastly, we turn to the proofs of \cref{prop:inject-equiv-map-induces:injective,prop:inject-equiv-map-induces:morphisms}.
  For \cref{prop:inject-equiv-map-induces:morphisms} we may just follow the previous constructions, noticing that whenever \(\varphi\) is assumed to be a morphism then \(\maxalg{\varphi}\) is a \Star{}homomorphism.
  Thus all of its descendants are also \Star{}homomorphisms.
  Likewise, if \(\varphi\) is surjective then so is \(\maxalg{\varphi} \colon \maxalg{\CA} \to \maxalg{\CB}\), for the span of the Banach spaces \(B_s\) is dense in \(\maxalg{\CB}\), and we can hit any \(b_s = \varphi_s(a_s) \in \varphi_s(A_s)\).
  Thus all the descendants \(\maxalg[D_0]{\varphi}\) and \(\redalg[D_0]{\varphi}\) are also surjective.
  
  Let us investigate \cref{prop:inject-equiv-map-induces:injective}.
  We have to show that
  \[
    \maxalg[D_0]{\varphi} \colon \maxalg[D_0]{\CA} \to \maxalg[D_0]{\CB}
  \]
  is injective.
  For this we instead show that \(\maxalg[D_0]{\varphi}(\maxalg[D_0]{\CA})\) enjoys the universal property of \(\maxalg[D_0]{\CA}\), implying that \(\maxalg[D_0]{\varphi}\) is isometric.
  Take any \(D_0\)-representation \(\pi \colon \maxalg[D_0]{\CA} \to \CL(H)\).
  We have to show that it factors through a representation of \(\maxalg[D_0]{\varphi}(\maxalg[D_0]{\CA})\).
  To simplify notation let us denote \(\maxalg[D_0]{\varphi}\) simply by \(\rho\). Assume, without loss of generality, that \(\pi\) is non-degenerate.
  We extend \(\pi\) to a representation of \(\CB\) via
  \[
    \pi_{\CB,s}\left(b_s\right) \cdot \pi\left(a\right) h \coloneqq \pi\left(b_s \rho\left(a\right)\right) h, \;\; \text{ where } h \in H.
  \]
  It is precisely here that the ad-hoc condition that \(\varphi(A)\) be an ideal of \(B\) is used: it makes the expression \(b_s \rho(a)\) be an element of \(\maxalg[D_0]{\CA}\).
  This \(\pi_\CB = (\pi_{\CB, s})_{s \in S}\) can be checked to be a representation of \(\CB\), and thus integrates to a \Star{}homomorphism \(\pi_\CB \colon \maxalg{\CB} \to \CL(H)\).
  Moreover, we may restrict \(\pi_\CB\) to \(\maxalg{\varphi}(\maxalg{\CA})\), thus yielding a \Star{}homomorphism
  \[
    \pi_\CA \coloneqq \pi_\CB|_\CA \colon \maxalg{\varphi}\left(\maxalg{\CA}\right) \to \CB\left(H\right).
  \]
  We claim that \(\pi_\CA\) actually descends to the desired representation of \(\rho(\maxalg[D_0]{\CA})\).
  For this, it is enough to show that \(\pi_\CB\) is a \(D_0\)-representation ``in the part of \(\CB\) coming from \(\CA\)''.
  What this precisely means is that by the description in \cref{lemma:alt-desc-max-dzero} it suffices to show that
  \[
    \pi_\CB\left(\left\{b \in \rho\left(\maxalg{\CA}\right) : \redcondexp_\CB\left(b^*b\right) \in J_{D_0}^\CB\right\}\right) = 0.
  \]
  Letting \(b = \sum_{s \in F} \varphi_s(a_s) \in \algalg{\CB}\) be such an element, we have that \(\pi_\CB(b) = \sum_{s \in F} \pi_\CB(\varphi_s(b_s)) = \sum_{s \in F} \pi(a_s) = \pi(\sum_{s \in F} a_s) = 0\) by the \(D_0\)-assumption on \(\pi\).
  Lastly, the injectivity claim in item \cref{prop:inject-equiv-map-induces:injective} for \(\redalg[D_0]{\varphi}\) is proven similarly to \cite{KwasniewskiMeyer2021}*{Proposition 4.2}.
  If \(\varphi\) is injective then so is \(\varphi_{1}^{**}\), and thus the claim follows from the built-in faithfulness feature of \(\condexp[D_0]^\CA\) and \(\condexp[D_0]^\CB\) and the commutativity of \eqref{eq:red-algs-diagram}.
\end{proof}

\cref{prop:inject-equiv-map-induces} will be used to construct an \(S\)-equivariant \Eth-theory in \cref{sec:equiv-eth}.
In the case we consider actual morphisms we have the following two results, describing the maximal/reduced scenarios.

\begin{theorem} \label{thm:maximal-extensions}
  Let \(0 \to \CI \xrightarrow{\iota} \CA \xrightarrow{q} \CB \to 0\) be a short exact sequence of Fell bundles.\footnote{\, This in particular implies that \(I \cong \iota(I)\) is an ideal in \(A\).}
  Suppose that both \(\iota\) and \(q\) are \(D_0\)-equivariant morphisms.
  Then there is an induced extension
  \[
    0 \to \maxalg[D_0]{\CI} \to \maxalg[D_0]{\CA} \to \maxalg[D_0]{\CB} \to 0
  \]
  of maximal crossed products.
\end{theorem}
\begin{proof}
  It is observed in \cite{Kwasniewski-Meyer:Pure_infiniteness}*{Section 4} that the morphisms \(\iota\) and \(q\) are always reduced in the sense of \cref{def:all-equiv-morphism}.
  The induced maps \(\widehat{\iota} \coloneqq \maxalg[D_0]{\iota}\) and \(\widehat{q} \coloneqq \maxalg[D_0]{q}\) are provided by \cref{prop:inject-equiv-map-induces}.
  Hence we only need to show that the sequence
  \[
    0 \to \maxalg[D_0]{\CI} \xrightarrow{\widehat{\iota}} \maxalg[D_0]{\CA} \xrightarrow{\widehat{q}} \maxalg[D_0]{\CB} \to 0
  \]
  is exact in the middle.
  The image of the map \(\widehat{\iota}\) is clearly contained in the kernel of \(\widehat{q}\), for the image of \(\widehat{\iota}\) is generated by monomials of the form \(\iota_s(i_s)\), with \(i_s \in I_s\).
  These get mapped to \(0\) under \(\widehat{q}\):
  \[
    \widehat{q}\left( \widehat{\iota}\left(i_s\right) \right) = \widehat{q}\left(\iota_s(i_s)\right) = q_s\left(\iota_s\left(i_s\right)\right) = 0.
  \]
  Hence, it follows that we only have to show that \(\ker(\widehat{q}) \subseteq {\rm im}(\widehat{\iota})\).
  A way to show this is to show that
  \[
    C \coloneqq \maxalg[D_0]{\CA}/{\rm im}\left(\widehat{\iota}\right) = \maxalg[D_0]{\CA}/\maxalg[D_0]{\CI}
  \]
  enjoys the universal property defining \(\maxalg[D_0]{\CB}\), \emph{i.e.}\ we have to show that if \(\pi \colon \CB \to \CL(H)\) is some \(D_0\)-representation of \(\CB\) then \(\pi\) integrates to a \Star{}homomorphism \(\widetilde{\pi} \colon C \to \CL(H)\).
  Hence, take any such \(D_0\)-representation \(\pi \colon \maxalg[D_0]{\CB} \to \CL(H)\).
  Pre-composing with \(\widehat{q}\) yields a representation \(\rho \coloneqq \pi \circ \widehat{q} \colon \maxalg[D_0]{\CA} \to \CL(H)\).
  We claim that \(\rho(\maxalg[D_0]{\CI}) = 0\).
  Again, this we only have to check on monomials, and it is clear: \(\rho(i_s) = \pi(\widehat{q}(i_s)) = 0\).
  Hence \(\rho\) descends to a representation \(\widetilde{\rho} \colon C \to \CL(H)\), proving that \(C\) does enjoy the universal property defining \(\maxalg[D_0]{\CB}\).
\end{proof}

\cref{thm:maximal-extensions} deals with the functoriality properties we need for maximal crossed products. As usual, their reduced counterparts are considerably more subtle, so we only observe the following.

\begin{theorem} \label{thm:reduced-extensions}
  Let \(0 \to \CI \xrightarrow{\iota} \CA \xrightarrow{q} \CB \to 0\) be a short exact sequence of Fell bundles.
  Suppose that both \(\iota\) and \(q\) are \(D_0\)-equivariant morphisms.
  Then there is a sequence
  \[
    0 \to \redalg[D_0]{\CI} \to \redalg[D_0]{\CA} \to \redalg[D_0]{\CB} \to 0
  \]
  which can only fail to be exact in the middle: the image of \(\redalg[D_0]{\CI} \to \redalg[D_0]{\CA}\) is always contained in the kernel of \(\redalg[D_0]{\CA} \to \redalg[D_0]{\CB}\).
\end{theorem}

\begin{proof}
  The proof is the same as that of \cref{thm:maximal-extensions} just adorning with ``\red'' instead of ``\full'' and skipping the second half of the proof.
\end{proof}

One may reduce the study of the reduced crossed products to the following.

\begin{definition}
  We say an action \(S \actson X\) is \emph{\(D_0\)-exact} if any short exact sequence \(0 \to \CI \to \CA \to \CB \to 0\) of Fell bundles over \(S\) that is \(D_0\) equivariant induces a short exact sequence
  \begin{equation} \label{eq:ext-func-properties-red}
    0 \to \redalg[D_0]{\CI} \to \redalg[D_0]{\CA} \to \redalg[D_0]{\CB} \to 0.
  \end{equation}
\end{definition}
For some (rather non-explicit, and with \(D_0 = \emptyset\)) examples of exact actions we refer the reader to \cite{Kwasniewski-Meyer:Pure_infiniteness}*{Examples 4.16 and 4.17}.
For some more tame ones we give the following.

\begin{example}
  Let \(S \actson X\) be a trivial action, that is, \(sx = x\) for all \(x \in s^*s\).
  The groupoid \(S \ltimes X\) is a group bundle, and it comes equipped with a unique \'etale topology restricting to the given one in \(X\).
  Exactness of the action \(S \actson X\) is precisely the same as that for all open \(\CU \subseteq X\) the short exact sequence \(S \ltimes \CU \to S \ltimes X \to S \ltimes F\), where \(F \coloneqq X \setminus \CU\), induces a short exact sequence
  \[
    0 \to \redalg{S \ltimes \CU} \to \redalg{S \ltimes X} \to \redalg{S \ltimes F} \to 0.
  \]
  The above may be non exact: see \cite{martinez-szakacs-2025}*{Section 4} for a particularly badly behaved example.
  The maximal analogue, however, is always exact.
\end{example}

\begin{example}
  Take two actions \(S \actson X, Y\) on locally compact Hausdorff spaces, and let \(\rho \colon Y \to X\) be a continuous surjection.
  Were \(\rho\) to be equivariant, \emph{i.e.}\ \(\rho(sy) = s \rho(y)\) for all \(y \in s^*s\), then it would also induce, via Gelfand duality, a Fell bundle morphism \(\widehat{\rho} \colon \cfunz{X} \to \cfunz{Y}\) via \eqref{eq:fell-bundle-from-action} and \cref{prop:inject-equiv-map-induces}.
  In such case one would often say that the \Star{}homomorphism \(\widehat{\rho} \colon \cfunz{X} \to \cfunz{Y}\) is \emph{equivariant}.
  Moreover, this would induce a \Star{}homomorphism \(\essalg{S \ltimes X} \to \essalg{S \ltimes Y}\) if \(\widehat{\rho}\) is assumed to also be \(D\)-equivariant, \emph{e.g.}\  if \(\rho\) sends the dangerous units of \(S \ltimes Y\) to the dangerous units of \(S \ltimes X\) (in the sense of \cite{buss-martinez-ess-ame-2024}, and where \(D \subseteq X\) is as in \eqref{eq:d-s}).
\end{example}

\section{A notion of amenability and a left regular representation} \label{sec:proper-ame}
We now introduce an appropriate notion of ``\(D_0\)-amenable'' action \(S \actson X\).
The following should be compared with \cref{def:ideals-i-s-t}.

\begin{definition} \label{def:ideals-i-s-t-dzero}
  Let \(\CB = (B_s)_{s \in S}\) be a Fell bundle, and assume \(\Phi_B \colon \opensets{X} \to \idealsin{B}\) is some equivariant map. Let \(D_0 \subseteq X\) be invariant.
  Given elements \(s, t \in S\) recall (see \cref{def:ideals-i-s-t}) that \(1_{s,t} \in \multialg{B} \subseteq B^{**}\) is the unit of \(I_{s,t}^{**}\).
  If \(J_{D_0}\) is as in \cref{prop:cond-exp-kills-dzero}, then we let
  \[1_{s,t}^{D_0} \coloneqq 1_{s,t} + J_{D_0} \in B^{**}/J_{D_0}.\]
\end{definition}

Given a Fell bundle \(\CB = (B_s)_{s \in S}\) its double dual \(\CB^{**} \coloneqq (B_s^{**})_{s \in S}\) is the collection of double duals of the spaces \(B_s\) themselves.
This \(\CB^{**}\) is clearly a Fell bundle in its own right, see \cref{lemma:double-dual-x-algebra}.
Given some invariant \(D_0 \subseteq X\), consider
\[
  B_{s}^{D_0} \coloneqq B_s^{**}/B_s^{**} \cdot J_{D_0}
\]
where \(J_{D_0}\) is as in \cref{prop:cond-exp-kills-dzero}.
We can also equip \(\CB^{D_0} \coloneqq (B_s^{D_0})_{s \in S}\) with a Fell bundle structure:
\begin{align*}
  B_{s}^{D_0} \times B_{t}^{D_0} \to B_{st}^{D_0}, \;\; \left(b_s + B_s^{**} \cdot J_{D_0}\right) \cdot \left(b_t + B_t^{**} \cdot J_{D_0}\right) \coloneqq b_s b_t + B_{st}^{**} \cdot J_{D_0} \in B_{st}^{**}/B_{st}^{**} \cdot J_{D_0}.
\end{align*}

\begin{lemma} \label{lemma:artifice-d0-fell-bundle}
  The data \(\CB^{D_0}\) above defines a Fell bundle.
\end{lemma}
\begin{proof}
  We will first show that \(J_{D_0}\) is \(\CB\) invariant, \emph{i.e.}\ \(B_s^{**} \cdot J_{D_0} = J_{D_0} \cdot B_{s}^{**}\) for all \(s \in S\).
  Let \(b_s \in B_s^{**}\) and \(a \in J_{D_0}\) be given.
  Consider the element:
  \[
    sa \coloneqq \left[X \supseteq s^*s \ni x \mapsto a\left(sx\right) \in B_{sx}^{**}\right] \in \prod_{x \in X} B_x^{**} \subseteq B^{**}.
  \]
  where the last inclusion follows from the fact that \(B^{**}\) quotients onto \(\prod_{x \in X} B_x^{**}\) and the former is a von Neumann algebra.
  It then follows that \(sa \in J_{D_0}\), since \(D_0\) is invariant and the support of \(a\) itself was contained in \(D_0\).
  Moreover \(b_s a = (sa) b_s\), proving that \(J_{D_0}\) is \(\CB^{**}\) invariant.
  The fact that \(\CB^{D_0}\) is a Fell bundle now follows routinely.
  Indeed, given \(b_s, b_t, b_r\) in \(B_s^{**}, B_t^{**}, B_r^{**}\) respectively we get that
  \[
    \left(b_s + B_s^{**} \cdot J_{D_0}\right) \left(\left(b_t + B_t^{**} \cdot J_{D_0}\right) \cdot \left(b_r + B_r^{**} \cdot J_{D_0}\right)\right) = b_s b_t b_r + B_{str}^{**} \cdot J_{D_0},
  \]
  proving associativity of the product. Likewise, the multiplication maps defining the Fell bundle are clearly isomorphisms, for the same holds for the double dual \(\CB^{**}\) itself.
\end{proof}

It is \emph{not remotely true} that \(\redalg[D_0]{\CB} \cong \redalg{\CB^{D_0}}\), for the former is usually separable while the latter almost never is.
The Fell bundle \(\CB^{D_0}\) is only an ad-hoc artifice we can use to generalize \cite{Exel:Amenability} and \cite{buss-martinez-approx-prop-23}*{Definition 4.3} to our setting, where we can ignore certain units \(D_0 \subseteq X\).
In particular the only interest of \cref{lemma:artifice-d0-fell-bundle} is that expressions such as \((b_s + J_{D_0}) b_1\) now make sense, where \(b_s \in B_s\) and \(b_1 \in B\).

\begin{definition} \label{def:approx-property}
  Let \(\CB = (B_s)_{s \in S}\) be a Fell bundle over \(B\). We denote by:
  \[
    B_\CCC \coloneqq C^* \left(\left\{1_{s,t} b : s, t \in S \text{ and } b \in B_{t^*ss^*t}\right\}\right) \subseteq B^{**}.
  \]
  We say \(\CB\) has the \emph{\(D_0\)-approximation property} if there is a net of finitely supported maps \(\xi_i \colon S \to B_\CCC/J_{D_0}\) such that \(\xi_i(s) \in B_{\CCC, ss^*} + J_{D_0} \coloneqq B_{\CCC} \cap B_{ss^*}^{**} + J_{D_0}\) for all \(s \in S\); and
  \begin{enumerate}
    \item \(\norm{\sum_{r,t \in S} \, 1_{r,t}^{D_0} \, \xi_i(r)^* \xi_i(t)} \leq 1\) for all \(i\); and
    \item \(\sum_{r, t} \, 1_{r, st}^{D_0} \, \xi_i(r)^* a \xi_i(t) \to b\) for all \(s \in S\) and \(a \in B_s\).
  \end{enumerate}
  We say \(S\) acts \emph{\(D_0\)-amenably} on \(X\) if the induced Fell bundle \(\cfunz{X}\) has the \(D_0\)-approximation property.
\end{definition}

\begin{example}
  In the case that \(B \cong \contz{X}\) and \(D \subseteq X\) is as in \eqref{eq:d-s} it follows that \(B_{\CCC}\) is the \cstar{}algebra whose spectrum is, by definition, the unit space of the \emph{Hausdorff cover} of \(S \ltimes X\), see \cite{brix-gonzalez-hume-li-2025}.
  Moreover, we refer the reader to \cref{sec:apps} for a discussion around the relationship between \(S \actson X\) being amenable and the groupoid of germs \(S \ltimes X\) being (topologically) amenable.
\end{example}

The main observations about Fell bundles and/or actions with the approximation property are the following. Firstly, they yield nuclear crossed products \cite{buss-martinez-approx-prop-23}*{Theorem 6.3}.
Moreover, should \(X\) be second countable and \(S\) be quasi-countable then we may assume that \(\{\xi_i\}_i\) is a sequence (as opposed to a net).
Lastly, if \(\CB\) is ``closed''~\cites{buss-martinez-approx-prop-23,Kwasniewski-Meyer:Pure_infiniteness} then \(B_\CCC = B\).
In that case monomials of the form \(1_{s,t} b\), where \(b \in B_{t^*ss^*t}\), are really contained in \(B\).
This is the ``Hausdorff'' situation, where the added non-Hausdorffness mentioned in \cref{rem:non-hausdorffness} does not appear.

\begin{example}
  It would certainly be convenient should it be true that if a Fell bundle \(\CB = (B_s)_{s \in S}\) has the \(D_0\)-approximation property then the action \(S \actson X\) is \(D_0\)-amenable.
  This, however, is false.
  For instance, we may choose any amenable action \(\FF_2 \actson B\) of the free group on \(2\) generators on a simple \cstar{}algebra, such as the Bernoulli action \(\FF_2 \actson \tensor_{\gamma \in \FF_2} M_{2^\infty}\).
  Then the action \(\FF_2 \actson \primidealspc{B} = \{\rm pt\}\) cannot be amenable, for \(\FF_2\) is non-amenable.
\end{example}

We now give a canonical left regular representation of \(\redalg[D_0]{\CB}\).
\begin{proposition} \label{prop:left-reg-rep}
  Let \(\CB = (B_s)_{s \in S}\) be a Fell bundle over an \((S,X)\) algebra, and let \(D_0 \subseteq X\) be some invariant subset.
  Given \(a = \sum_{s \in S} a_s\) and \(b = \sum_{s \in S} b_s \in \bigoplus_{s \in S} B_s\), consider:
  \[
    \innerprod{a}{b}_{D_0} \coloneqq \condexp[D_0]\left(a^*b\right) = \sum_{s, t \in S} a_s^*b_t 1_{s,t}^{D_0} \in B^{**}/J_{D_0},
  \]
  which is a semi-inner product with values in \(B^{**}/J_{D_0}\).
  Let \(\ell^2_{D_0}(S; \CB)\) be the \(B^{**}/J_{D_0}\)-Hilbert bimodule obtained by completing \(\bigoplus_{s \in S} B_s^{**}/B_s^{**} J_{D_0}\) with respect to \(\innerprod{\variable}{\variable}_{D_0}\).
  Then the map
  \begin{equation}
    \Lambda_{D_0} \colon \CB = \left(B_s\right)_{s \in S} \to \CL\left(\ell_{D_0}^2\left(S; \CB\right)\right); \;\; \text{ where } \; \Lambda_{D_0,s}\left(b_s\right) \widehat{a_t} \coloneqq \widehat{b_s a_t},
  \end{equation}
  where \(\widehat{a_t}\) is the image of \(a_t\) under the canonical
  \[
    \bigoplus_{s \in S} B_s \to \bigoplus_{s \in S} B_s^{**} \to \bigoplus_{s \in S} B_s^{**}/B_{s}^{**} J_{D_0} \to \ell^2_{D_0}(S; \CB),
  \]
  defines a \(D_0\)-representation \(\Lambda_{D_0} \colon \maxalg[D_0]{\CB} \to \CL(\ell^2_{D_0}(S; \CB))\). In fact, its image is precisely \(\redalg[D_0]{\CB}\).
\end{proposition}
\begin{proof}
  The fact that \(\innerprod{\variable}{\variable}_{D_0}\) is a semi inner product is rather self-explanatory, as is the fact that \(\Lambda_{D_0}\) defines a representation.
  In fact, \(\Lambda_{D_0}\) is a \(D_0\) representation, as in, we have that
  \[
    \Lambda_{D_0}\left(\left\{a \in \algalg{\CB} : \redcondexp\left(a^*a\right) \in J_{D_0}\right\}\right) = 0,
  \]
  where, as usual, \(\redcondexp \colon \redalg{\CB} \to B^{**}\) is the usual weak conditional expectation.
  In order to show this, let \(a = \sum_{s \in F} a_s\) be such that \(\redcondexp(a^*a) \in J_{D_0}\), that is, \(\supp{\redcondexp(a^*a)} \subseteq D_0\).
  Take some \(b_t \in B_t\), and let \(\widehat{b_t}\) be its image in \(\ell^2_{D_0}(S; \CB)\).
  We aim to show that \(\widehat{a b_t} = 0\), for which it suffices to show that \(\condexp[D_0](b_t^*a^*ab_t) = 0\).
  This, however, holds precisely due to \(\condexp[D_0](a^*a) = 0\) and \cref{lemma:description-j-d}.
  Thus \(\Lambda_{D_0}\) integrates to a representation \(\Lambda_{D_0} \colon \maxalg[D_0]{\CB} \to \CL(\ell^2_{D_0}(S; \CB))\).

  In order to show that the image of \(\Lambda_{D_0}\) is isomorphic to \(\redalg[D_0]{\CB}\) we first aim to reconstruct the faithful generalized conditional expectation \(\condexp[D_0] \colon \redalg[D_0]{\CB} \to B^{**}/J_{D_0}\) at the level of \(\CL(\ell^2_{D_0}(S; \CB))\).
  In order to do this, notice that there are canonical bounded maps, the last of which may in general be non-isometric:
  \[
    B_t \cdot I_{t,1} \subseteq B_t \to B_t^{**} \to \ell^2_{D_0}\left(S; \CB\right).
  \]
  Thus there is a unique weakly continuous extension to \(B_t^{**} I_{t,1}^{**} \to \ell^2_{D_0}(S; \CB)\): call the image \(F_t\).
  We claim that \(F_t\) is a complemented module for all \(t \in S\). 
  For this, observe that the canonical inclusion \(F_t \subseteq \ell^2_{D_0}(S; \CB)\) is the dual of the (projection) map \(p_t \colon \ell^2_{D_0}(S; \CB) \to F_t\) defined via \(p_t a_s \coloneqq a_s 1_{s,t}^{D_0} 1_{t,1}^{D_0}\) for all \(s \in S, a_s \in B_s\).
  Indeed, we have that
  \[
    \innerprod{p_t a_s}{a_t 1_{t,1}^{D_0}}_{D_0} = a_s^* a_t 1_{s,t}^{D_0} 1_{t,1}^{D_0} = \innerprod{a_s}{a_t 1_{t,1}^{D_0}}_{D_0}
  \]
  for all \(a_s, a_t\) in \(B_s^{**}, B_t^{**}\) respectively.
  Moreover, the submodule generated by all \(F_t\) is also complemented, for it is the image of the join projection of all \(p_t\), denoted it by \(p \coloneqq \vee_{t \in S} p_t\).
  In particular it follows that
  \[
    \Phi \colon \maxalg[D_0]{\CB} \xrightarrow{\Lambda_{D_0}} \CL\left(\ell^2_{D_0}\left(S; \CB\right)\right) \xrightarrow{p \variable p} \CL\left(p \cdot \ell^2_{D_0}\left(S; \CB\right)\right)
  \]
  is a well defined completely positive and contractive map.
  A computation shows that \(\Phi\) restricts to \(\condexp[D_0]\) on \(\algalg[D_0]{\CB}\), \emph{i.e.}\ on the span of the images of \(B_s \to \maxalg[D_0]{\CB}\):
  \[
    \Phi\left(b_s\right) \cdot \widehat{a_t} = p \Lambda_{D_0}\left(b_s\right) p \cdot \widehat{a_t} = p \cdot \widehat{b_s a_t 1_{t,1}} = \widehat{1_{s,1} b_s a_t 1_{t,1}} = \condexp[D_0]\left(b_s\right) \cdot p \widehat{a_t}.
  \]
  Hence \(\Phi\) is \(\condexp[D_0]\), for the former subspace \(\algalg[D_0]{\CB}\) is dense.
  In particular it follows that the norm of all \(b \in \algalg[D_0]{\CB}\) in \(\redalg[D_0]{\CB}\) is larger than the norm of \(\Lambda_{D_0}(b)\) in \(\ell^2_{D_0}(S; \CB)\).
  This precisely means that \(\Lambda_{D_0}\) descends (or induces) a representation \(\Lambda_{D_0}^{\red} \colon \redalg[D_0]{\CB} \to \CL(\ell^2_{D_0}(S; \CB))\).
  The fact that \(\Lambda_{D_0}^{\red}\) is an embedding follows as well, as it is isometric on \(\algalg[D_0]{\CB}\) (due to the previous discussion) and the latter is dense.
\end{proof}

\begin{remark} \label{rem:more-left-reg}
  The statement and proof of \cref{prop:left-reg-rep} should be regarded more as a method to construct a representation of \(\redalg[D_0]{\CB}\).
  By this we mean that they also work when \(\ell^2_{D_0}(S; \CB)\) is constructed as the completion of any Fell sub-bundle \(\widetilde{\CB} \subseteq \CB^{**}\) as long as the unit fiber of \(\widetilde{\CB}\) contains the image of \(\redcondexp \colon \redalg{\CB} \to B^{**}\).
  Thus, for instance, the same proof can be carried out in the case of groupoids from the point of view of the Hausdorff covers of \cite{brix-gonzalez-hume-li-2025}, or from the point of view of \cite{tu-baum-connes-2004}*{Sections 3 and 7}, or from the point of view of the Borel algebra of \cref{sec:apps-ess}.
\end{remark}

\begin{example}
  In the case that \(D = \emptyset\) then the map \(\Lambda_{D}\) above generalizes the usual left regular representation of an \'etale groupoid presented in previous literature, such as \cites{buss-martinez-ess-ame-2024,Kwasniewski-Meyer:Pure_infiniteness,KwasniewskiMeyer2021}.
  By this we mean that \(\Lambda_\emptyset \colon \essalg{S \ltimes X} \to \CL(\ell^2(\algalg{G}))\) is the usual left regular representation. (Here \(\ell^2(\algalg{G})\) is the completion of \(\algalg{G}\) under the inner product \(\xi^*\zeta|_{G \setminus GD}\), where \(D \subseteq G^{(0)}\) is as in \eqref{eq:d-s}.)
\end{example}

\section{Properness and cut-off functions} \label{sec:proper}
In this section we introduce the adequate notion of \emph{proper} action \(S \actson X\), along with so-called \emph{cut-off functions} \(c \colon X \to [0,1]\).
We remark the former notion is rather different from the one studied classically.
Indeed, one should compare \cite{tu-baum-connes-2004}*{Definition 2.9} and \cref{def:proper-grpd}, and notice both reduce to the same one in the Hausdorff situation, which is the one used in \cite{tu-baum-connes-1999}*{Definition 3.4}.

\subsection{Proper actions}
Henceforth we work under the convention that \(S\) is a discrete inverse semigroup acting on a locally compact, Baire, sober, \tzero space \(X\).
Moreover, \(D_0 \subseteq X\) will be a fixed invariant subset.

\begin{definition} \label{def:proper-grpd}
  Let \(S \actson X\).
  A set \(F \subseteq S\) is \emph{finitely upper bounded} if it is bounded above by a finite set, \emph{i.e.}\ there is some finite \(\CF \subseteq S\) such that for all \(s \in F\) and \(x \in s^*s\) there is some \(t \in \CF\) and idempotent \(e \in E\) such that \(x \in e t^*t\); \(se = te\) and \(\closure{(se)^*se} = \closure{s^*se} \subseteq t^*t\).
  The action \(S \actson X\) is \emph{\(D_0\)-proper} if for all compact sets \(K_0, K_1 \subseteq X\) the set
  \[
    \left\{s \in S : s\left(K_0 \setminus D_0\right) \cap K_1 \neq \emptyset\right\} \subseteq S
  \]
  is finitely upper bounded.
  We say the action is \emph{proper} if it is \(\emptyset\)-proper.
\end{definition}

\begin{remark}
  The condition that \(\closure{s^*s e} \subseteq t^*t\) is crucial, particularly in \cref{lemma:proper-action-orbit-space}.
  It is often not needed should \(X\) be totally disconnected and Hausdorff, as then \(s^*s\) is often assumed to be open and compact (and thus closed too).
  In particular it is not needed for universal or tight groupoids of inverse semigroups.
\end{remark}

Heuristically, the more units \(D_0 \subseteq X\) we ignore the easier it is to meet our definitions.
In fact, the following is obvious from the definition, but should be mentioned nevertheless.

\begin{lemma}
  Let \(S \actson X\) and \(D_0 \subseteq D_1 \subseteq X\) be invariant. If \(X\) is \(D_0\)-proper then it also is \(D_1\)-proper.
\end{lemma}

\begin{example} \label{ex:0-1-Gamma}
  Let \(\Gamma\) be an infinite discrete group. Consider the groupoid \(G \coloneqq [0,1) \sqcup \Gamma\), equipped with the obvious multiplication and the unique \'etale topology such that \(x_k \to \gamma\) for all \(\gamma \in \Gamma\) and net \((x_k)_k \subseteq [0,1)\) converging to \(1\).
  This is an \'etale groupoid with only isotropy; and \(X\) is compact, but \(G\) is \emph{not}. Therefore the action \(S \coloneqq \Gamma \sqcup \{0\} \actson [0,1]\) is not \(\emptyset\)-proper, but it \emph{is} \(\{1\}\)-proper.
  It also \(D_0\)-proper for any \(D_0 \subseteq [0,1]\) containing \(1\).
\end{example}

\begin{example} \label{ex:0-1-Gamma:non-loc-hausd}
  Let \(Z \coloneqq [0,1) \sqcup \RR\), equipped with the topology uniquely determined by
  \begin{enumerate}
    \item the subspace topology of both \([0, 1)\) and \(\RR\) in \(Z\) being the usual topology; and
    \item \(x_k \to r\) for all \(r \in \RR\) and any net \((x_k)_k \subseteq [0,1)\) converging to \(1\).
  \end{enumerate}
  In particular, observe that \(Z\) is not locally Hausdorff.
  The semigroup \(S \coloneqq \ZZ \sqcup \{0\}\) canonically acts on \(Z\) via \(k \cdot r \coloneqq k+r \in \RR\) (and trivially on \([0,1)\)).
  We claim this action is, in fact, proper.
  For this, note that every compact \(K \subseteq Z\) is contained in a subset of the form \(\widetilde{K} \coloneqq [0,1) \sqcup [-k,k] \subseteq Z\), which is also compact.
  Moreover,
  \[
    \left\{s \in S : s \widetilde{K} \cap \widetilde{K} \neq \emptyset\right\} \subseteq \left\{0\right\} \sqcup \left\{-k-1, -k, \dots, k, k+1\right\} \subseteq \left\{0\right\} \sqcup \ZZ = S,
  \]
  which is a finite upper bound witnessing the properness of the action.
\end{example}

We are particularly interested in the structure of the orbit space of a \(D_0\)-proper action.
For this, we say \(x_1 \sim x_2\) if there is some \(s \in S\) such that \(sx_1 = x_2\).\footnote{\, Here, and from now on, if we ever write \(sx\) in the text it will be implicitly assumed that this is a possible thing to do, \emph{i.e.}\ \(x \in s^*s\).}
This is clearly an equivalence relation, and we are interested in when the orbit space \(X/S \coloneqq X/\sim\) is Hausdorff.
Needless to say, there is no reason for this to happen should \(X\) itself be non-Hausdorff, but the following shows that one can control the ``non-Hausdorffness'' of \(X/S\): it is (at most) that of \(X\) itself.
Before we state this we say \(x \in X\) is \emph{dangerous} if there is a net \(\{x_\lambda\}_\lambda \subseteq X\) converging to both \(x\) and to some \(y \in X\) with \(x \neq y\).
In particular, open neighborhoods of \(x\) and \(y\) will always intersect and, thus, these points ``cannot be separated'' (at least in a Hausdorff manner).

\begin{lemma} \label{lemma:proper-action-orbit-space}
  Let \(S \actson X\) be a proper action on a locally compact \tzero space.
  The quotient map \(X \to X/S\) is open; the graph of \(\sim_S\) is closed in \(X \times X\); and if \([x] \in X/S\) is dangerous then \(x\) is dangerous.

  Moreover, if \(S \actson X\) is \(D_0\) proper then the graph of \(\sim_S\) is ``\(D_0\)-closed'', in the sense that if \((x_\lambda, s_\lambda x_\lambda) \to (x, y)\) for points \(\{x_\lambda\}_\lambda, \{s_\lambda x_\lambda\}_\lambda, x, y \not\in D_0\), then \(x \sim_S y\).
\end{lemma}
\begin{proof}
  We first show that the quotient map \(\pi \colon X \to X/S\) is open.
  Take some open set \(\CU \subseteq X\). Note that \(\pi^{-1}(\pi(\CU)) = S \CU = \cup_{s \in S} \, s \cdot (\CU \cap s^*s)\), which is a union of open sets of \(X\), and thus open.
  By definition of the quotient topology it follows that \(\pi(\CU)\) is open too.

  We now study the closedness of \(R \coloneqq \{(sx, x) : x \in s^*s \subseteq X\}\) (as a subset of \(X \times X\)).
  Suppose that \((s_n x_n, x_n) \to (y, x)\) for some net.
  We have to show that there is some \(t \in S\) such that \(y = tx\).
  As \(X\) is assumed to be locally compact there are compact neighborhoods \(K_x\) and \(K_y\) of \(x\) and \(y\) respectively.
  By properness there is some finite \(\CF \subseteq S\) such that
  \[
    \left\{s \in S : s K_x \cap K_y \neq \emptyset\right\}
  \]
  is bounded above by \(\CF\).
  This precisely means that if \(z \in K_x\) and \(sz \in K_y\) then there is some \(t \in \CF\) and idempotent \(e \in E\) such that \(z \in e t^*t\); \(se = te\) and the closure of \(s^*s e\) is contained in \(t^*t\).
  Since \(x_n \to x\) and \(sx_n \to y\), for all large \(n\) there are \(t_n \in \CF\) and \(e_n \in E\) such that \(s_n e_n = t_n e_n\) and \(x_n \in e_n t_n^*t_n\). As there are infinitely many \(n\) and finitely many \(t_n \in \CF\), without loss of generality, and by passing to a sub-net if necessary, we may assume that \(t x_n \to y\) for a henceforth fixed \(t \in \CF\).
  By continuity of \(t\) (as a partial map of \(X\)) it follows that it is enough to show that \(y \in t^*t\).
  Note that \(y \in \closure{t^*t}\), as the latter set contains \(t x_n\) and these converge to \(y\).
  In fact, \(y \in \closure{e_n s_n^*s_n} \subseteq t^*t\) from the properness of the action, which implies that \(R\) is indeed closed.

  In the case when \(X\) is Hausdorff we are now done: the quotient map \(X \to X/S\) is open, and \(R\) is closed (in \(X \times X\)).
  This is then sufficient to show that \(X/S\) is Hausdorff.
  Nevertheless, in our setting \(X\) is only a \tzero space, so there is still work to do.
  Take some \([x] \in X/S\), and suppose that some representative \(x \in X\) is not dangerous. We will show that neither is \([x]\).
  For this, suppose that \([x_\lambda] \to [x], [y]\). By choosing different representatives if needed we may suppose that \(x_\lambda \to x\) and, as \(x\) is not dangerous, \(x\) is the only limit point of \(\{x_\lambda\}_\lambda\).
  Moreover, there are \(s_\lambda\) such that \(s_\lambda x_\lambda \to y\).
  Again by local compactness of \(X\) choose some compact neighborhoods \(K_x, K_y\) of \(x\) and \(y\) respectively.
  We may assume, without loss of generality, that \(s_\lambda x_\lambda \in K_y\) for all \(\lambda\).
  Then \(x_\lambda = s_\lambda^* s_\lambda x_\lambda \to x\), which implies (by an argument using properness of the action similar to the one in the previous paragraph) that we may assume \(s \coloneqq s_\lambda = s_{\lambda'}\) for all large enough \(\lambda, \lambda'\).
  But then \(x_\lambda = s^* s x_\lambda \to x, s^*y\). As \(x\) is not dangerous it follows that \(x = s^*y\), so that \([x] = [y]\).

  The ``moreover'' statement is proven similarly to the fact that \(\sim_S\) is closed in \(X \times X\), just using points outside of \(D_0\) as opposed to arbitraty points in \(X\).
\end{proof}

\subsection{Borel cut-off functions}
We proceed with the definition of a \emph{cut-off function}, which is reminiscient, and should be compared with, the one in \cite{higson-kasparov-2000}*{Theorem 8.1} or \cite{tu-baum-connes-2004}*{Section 6}.
Classically (\emph{i.e.}\ in the case of group actions \(\Gamma \actson X\)) the existence of a cut-off function is (essentially) a consequence of \cref{lemma:proper-action-orbit-space}.
This also works in the setting of Hausdorff groupoids, as stated (and not proved) in \cite{tu-baum-connes-1999}*{Lemme 10.5}.
In the setting of non-Hausdorff groupoids this approach is no longer desirable, as shown in \cite{tu-baum-connes-2004}*{Section 6}, and also explored below.
This essentially is the same problem faced in \cites{buss-martinez-ess-ame-2024,renault-2013}, and has to do with the discontinuities hinted on in \cref{rem:non-hausdorffness}.

For the following, given \(x \in X\) and \(s_1, s_2 \in S\) such that \(x \in s_1^*s_1 \cap s_2^*s_2\), we say \(s_1 \sim_x s_2\) whenever they are trivially equal on a neighborhood around \(x\), meaning that there is some idempotent \(e \in E\) such that \(x \in e\) and \(s_1 e = s_2 e\).
This is an equivalence relation, and we let \(R_x\) be a set of representatives of \(\sim_x\). 

\begin{proposition} \label{prop:cut-off-functions}
  Let \(S \actson X\) be an action on a locally compact, Hausdorff space \(X\). Suppose \(X\) is
  \begin{itemize}
    \item \emph{\(S\)-compact}; meaning that \(X = SK\) for some relatively compact \(K \subseteq X\); and
    \item the action \(S \actson X\) is proper.
  \end{itemize}
  Then \(S \actson X\) admits a \emph{Borel cut-off function}, that is, a Borel map \(c \colon X \to [0,1]\) such that \(\sum_{[s]_x \in R_x} c(sx)^2 = 1\) for all \(x \in X\) and whose support is contained in a compact set.
  In fact, \(c\) may be taken to continuous on \(X \setminus D\), where \(D \subseteq X\) is as in \eqref{eq:d-s}.
\end{proposition}

\begin{remark}
  Observe that, by properness of the action, the sum \(\sum_{[s]_x \in R_x} c(sx)^2\) is, in fact, finite for all \(x \in X\).
  Moreover, the assignment \(([s]_x, x) \mapsto sx\) really is well defined, and hence so is the sum.
\end{remark}

\begin{proof}[Proof of \cref{prop:cut-off-functions}]
  By \cref{lemma:proper-action-orbit-space} we have that the quotient map \(\pi \colon X \to X/S\) is an open surjection between locally compact, Hausdorff spaces.
  In particular there is a collection \(\{\xi_i\}_{i \in I} \subseteq \contc{X}\) of functions \(0 \leq \xi_i \leq 1\) such that \(\CV_i \coloneqq \pi(f_i^{-1}(0,1])\) forms a locally finite open cover of \(X/S\).
  Consider \(\xi(x) \coloneqq \sum_{i \in I} \xi_i(x)\); and \(\zeta(x) \coloneqq \sum_{[s]_x \in R_x} \xi(sx)\).
  Notice that the sum appearing in the definition of \(\xi\) is actually finite for all \(x \in X\), for \(\CV_i\) is locally finite.
  Moreover, let \(c(x) \coloneqq (\xi(x)/\zeta(x))^{1/2}\).
  We claim that \(c \colon X \to [0,1]\) meets the requirements of the statement.
  First of all, note that \(\xi\) is continuous, whereas \(\zeta\) is (in general) Borel.
  Nevertheless, if \(x_\lambda \to x\) with \(\{x_\lambda\}_\lambda, x \not\in D\), then
  \[
    \zeta\left(x_\lambda\right) = \sum_{[s]_{x_\lambda} \in R_{x_\lambda}} \xi(sx_\lambda) \to \sum_{[s]_x \in R_x} \xi(sx) = \zeta\left(x\right),
  \]
  proving the continuity asserted in the ``in fact'' statement.
  Furthermore, notice the support of \(c\) is contained in that of \(\xi\), which is relatively compact.
  Thus we only have to check that \(\sum_{[s]_x \in R_x} c(sx)^2 = 1\):
  \[
    \sum_{\left[s\right]_x \in R_x} c\left(sx\right)^2 = \sum_{\left[s\right]_x \in R_x} \frac{\xi\left(sx\right)}{ \zeta\left(sx\right) } = \sum_{\left[s\right]_x \in R_x} \frac{\xi\left(sx\right)}{\sum_{\left[t\right]_{sx} \in R_{sx}} \xi\left(tsx\right)} = \sum_{\left[s\right]_x \in R_x} \frac{\xi\left(sx\right)}{\sum_{\left[p\right]_{x} \in R_{x}} \xi\left(px\right)} = 1,
  \]
  as desired.
\end{proof}

\cref{prop:cut-off-functions} takes care of the case whenever \(X\) is Hausdorff and \(D_0 = \emptyset\).
Nevertheless, typically speaking \(X = \primidealspc{B}\), so we need to generalize it a bit further.
Luckily, the added non-Hausdorffness in \cref{rem:non-hausdorffness} has been taken care of by the Borel nature of \(c\).

\begin{proposition} \label{prop:cut-off-functions-primidlspc}
  Let \(S \actson X\) be a \(D_0\)-proper action on a locally compact, \(S\) compact, \tzero space \(X\).
  The action \(S \actson X\) admits a Borel \(D_0\)-cut-off function.
\end{proposition}
\begin{proof}
  Reading through the proof of \cref{prop:cut-off-functions} we only need to prove that similar functions \(\{\xi_i\}_{i \in I}\) exist even if \(\pi \colon X \to X/S\) is a map between \tzero spaces.
  Of course, there is no reason why \(\xi_i\) would be continuous in such case, but we have forsaken those aspirations already.
  Thus we only need to show that such \(\xi_i\) exist if they are assumed Borel, which one can do by hand.
  As \(X\) is \(S\)-compact there is some compact set \(K \subseteq X\) such that \(X = SK\).
  Moreover, as \(S\) is quasi-countable there is some countable \(\CCC \subseteq S\) such that \(S = \CCC E\).
  Arbitrarily choosing an enumeration \(\CCC = \{c_1, c_2, \dots\}\) it follows that \(X = \cup_{n \in \NN} c_n(K \cap c_n^*c_n)\).
  We aim to construct a locally finite open cover of \(X\) of subsets contained in compact subsets.

  For every \(x \in X\) let \(\CU_x \subseteq X\) be an open neighborhood of \(x\) contained in some compact set, which exists by local compactness of \(X\) as a space.
  The open cover \(\{\CU_x\}_{x \in X}\) is not necessarily locally finite, so it will need to be refined.
  As \(K\) is compact (and contained in \(X\)) there are points \(x_1, \dots, x_\ell \in X\) such that \(K \subseteq \CU_{x_1} \cup \dots \cup \CU_{x_\ell}\), and the latter is contained in a compact set, for it is a finite union of sets with that same property.
  Put \(K_n \coloneqq \cup_{j = 1}^n c_j (K \cap c_j^*c_j)\).
  This is a countable exhaustion of \(X\) by sets contained in compact sets so, in particular, by properness of the action we have that the sets
  \[
    \left\{s \in S : s\left(K_n \setminus D_0\right) \cap K_n \neq \emptyset\right\}
  \]
  are finitely upper bounded for all \(n \in \NN\), say by \(\{t_{n,1}, \dots, t_{n,m_n}\} \subseteq S\).
  Furthermore, by the previous discussion we have that
  \[
    X = \bigcup_{n \in \NN} c_n \left(K \cap c_n^*c_n\right) = \bigcup_{n \in \NN} \bigcup_{i = 1}^\ell c_n \left(\CU_{x_i} \cap c_n^*c_n\right) = \bigcup_{n \in \NN} \bigcup_{i = 1}^\ell \bigcup_{j = 1}^{m_n} t_{n,j} \left(\CU_{x_i} \cap t_{n,j}^* t_{n,j}\right),
  \]
  that is, the former expression forms an open cover of \(X\).
  We claim that it is locally finite (on \(X \setminus D_0\)) and every subset in it is contained in a compact set.
  The latter condition is clear, for the same is true for \(\CU_{x_i}\) itself and every \(s \in S\) defines a partial homeomorphism of \(X\).
  For the former condition, suppose \(z = c_n x = c_m y\), with \(x \in \CU_{x_i} \cap c_{n}^*c_{n}\) and \(y \in \CU_{x_i} \cap c_{m}^*c_{m}\), is some element in some open set of the cover.
  Say that \(m \geq n\).
  By the properness assumption it follows that \(c_m^* c_n\) coincides with some \(t_{m,j_0}\) on a neighborhood of \(y\).
  In particular \(z\) may only be contained in the sets \(\{t_{m,j} \CU_{x_i}\}_{i,j}\), of which there are only finitely many. This proves local finiteness of the open cover \(\{\CV_{n, i, j}\}_{n, i, j} \coloneqq \{t_{n, j} (\CU_{x_i} \cap t_{n,j}^*t_{n,j})\}_{n, i, j}\).

  Given the open cover \(\{\CV_{n, i, j}\}_{n, i, j}\), let \(\xi_{1,1,1}\) be the characteristic function of \(\CV_{1,1,1}\), and  iteratively construct \(\xi_{n,i,j}\) as the characteristic function of \(\CV_{n,i,j} \setminus (\bigcup_{(n',i',j') < (n,i,j)} \CV_{n,i,j})\).
  We may now apply the same reasoning as in \cref{prop:cut-off-functions}, yielding an appropriately Borel cut-off function \(c \colon X \to [0,1]\).
\end{proof}

These cut-off functions are unique up to homotopy, and they also in fact define canonical projections.

\begin{proposition} \label{prop:hom-cut-off}
  Let \(S \actson X\) be an action on a locally compact, \tzero space \(X\).
  Any two Borel cut-off functions \(c_0, c_1 \colon X \to [0,1]\) are homotopic. Moreover, the formula
  \[
    p_c \left(\left[s, x\right]\right) \coloneqq c\left(sx\right) \, c\left(x\right), \;\;\; \text{ where } x \in s^*s \setminus D_0,
  \]
  defines a projection in \(\maxalg[D_0]{S \ltimes X}^{**}\) and \(\redalg[D_0]{S \ltimes X}^{**}\).
  Lastly, this projection is unique up to homotopy, so its \Kth-theory class is well defined.
\end{proposition}
\begin{proof}
  For the fact that any two cut-off functions \(c_0, c_1\) are homotopic consider:
  \[
    c_r^2 \coloneqq r c_1^2 + \left(1-r\right) c_0^2,
  \]
  where \(r \in [0,1]\). Clearly \(c_r\) is a continuous path from \(c_0\) to \(c_1\), and all \(c_r\) are cut-off functions.
  Let us thus show that \(p\) defines a projection.
  For this, first observe that as \(c\) is Borel we have that \(c \in B_c(X) \subseteq B_c(S \ltimes X)\), where we are using the notation of \cite{buss-martinez-ess-ame-2024}*{Section 4}.\footnote{\, There really is nothing deep going on here. We have that \(S \ltimes X\) is an \'etale groupoid, and \(B_c(S \ltimes X)\) is then just (by definition) the functions \(f \colon S \ltimes X \to \CC\) that are Borel and whose support is contained in a compact set of \(S \ltimes X\). We equip this with the usual convolution product, yielding a \Star{}algebra.}
  Recalling the construction of the groupoid of germs \(S \ltimes X\) and the convolution product of \(B_c(S \ltimes X)\), we have that
  \begin{align*}
    \left(p_c p_c\right) \left(\left[s,x\right]\right) = \sum_{\left[t, rx\right] \left[r,x\right] = \left[s,x\right]} p_c\left(\left[t, rx\right]\right) p_c\left(\left[r,x\right]\right) = \sum_{\left[tr,x\right] = \left[s,x\right]} c\left(sx\right) c\left(rx\right)^2 c\left(x\right) = p_c\left(\left[s,x\right]\right) 
  \end{align*}
  for all arrows \([s,x] \in S \ltimes X\).
  In particular it follows that \(p_c\) is a self-adjoint Borel projection whose support (as a function on \(S \ltimes X\)) is contained in a compact set (namely the orbit of whichever compact set \(\supp{c}\) is contained in).
  Thus \(p_c \in B_c(S \ltimes X)\), which canonically embeds into the \cstar{}algebras in the statement.
  Lastly, the fact that the \Kth-theory of \(p_c\) does not depend on \(c\) follows since the \(\{c_r\}_{r \in [0,1]}\) considered above are a homotopy, and thus \(\{p_{c_r}\}_{r \in [0,1]}\) defines a continuous path connecting any two given \(p_{c_0}\) and \(p_{c_1}\).
\end{proof}

In the case of an action \(S \actson X\) on a Hausdorff space we need slightly more: we need to prove that the canonical projection \(p\) may be chosen to be contained in \(\maxalg{S \ltimes X}\).
\begin{proposition} \label{prop:proj-proper-action}
  Let \(S \actson X\) be a \(D_0\)-proper action on a locally compact, Hausdorff space \(X\).
  Then the projection \(p_c\) defined before may be chosen to be contained in \(\maxalg[D_0]{S \ltimes X}\).
\end{proposition}
\begin{proof}
  This is a consequence of the methods of \cite{tu-baum-connes-2004}*{Proposition 6.5}, but in order to apply it we have to first observe that the groupoid of germs (also known as transformation groupoid) \(S \ltimes X\) is proper, in the sense of \cite{tu-baum-connes-2004}*{Definition 2.9}, if and only if \(S \actson X\) is proper in our sense.
  As this is more related to applications, we leave the proof to \cref{lemma:proper-iff-proper}.
  Note that \cite{tu-baum-connes-2004}*{Proposition 6.5} only deals with the case when \(D_0 = \emptyset\), but the proof is point-by-point, so one may choose to ignore \(D_0\) and arrive at the desired result.
\end{proof}

\begin{example}
  As discussed, it is self-evident that our definition of proper action generalizes the classical one used by, \emph{e.g.}\ \cite{higson-kasparov-2000}, should \(S\) be a group.
  More generally, if \(S = \opensets{X}\) is as in \cref{ex:opensets-x}, the canonical action \(S \actson X\) is proper.
  Indeed, any set \(F \subseteq S\) is covered by the singleton \(\{X\} \subseteq S\).
  If \(X\) is compact (so that it is also \(S\)-compact in the sense of \cref{prop:cut-off-functions}), a canonical choice for cutoff function is just the constant map \(X \ni x \mapsto 1 \in [0,1]\).
  The projection \(p_c\) is then the identity.
\end{example}

We end the sub-section by proving that \(\maxalg[D_0]{\CB} \to \redalg[D_0]{\CB}\) is isometric when \(\CB\) is proper.

\begin{theorem} \label{lemma:proper-actions-give-weak-containment}
  Let \(\CB = (B_s)_{s \in S}\) be a Fell bundle over an \((S,X)\)-algebra \(B\).
  Assume \(S \actson X\) is \(D_0\)-proper.
  The canonical quotient map \(\maxalg[D_0]{\CB} \to \redalg[D_0]{\CB}\) is then a \Star{}isomorphism.
\end{theorem}

The proof requires the following technical point, which will also come in handy later.
\begin{lemma} \label{lemma:cover-s-by-ti}
  Suppose \(s\) is ``covered by \(t_1, \dots, t_k\)'', meaning that \(\{s\}\) is bounded above by \(\{t_1, \dots, t_k\}\) in the sense of \cref{def:proper-grpd}.
  Then \(1_{s^*s}^{D_0} \leq \vee_{j = 1, \dots, k} 1_{s, t_j}^{D_0}\) (as projections in \(B^{**}/J_{D_0}\)).
\end{lemma}
\begin{proof}
  Translating from central open projections in \(B^{**}\) to open sets in \(X\) we have to show that
  \[
    s^*s \setminus D_0 \subseteq \bigcup_{j = 1}^k O_{s,t_j} \setminus D_0,
  \]
  where \(O_{s,t} = \bigcup_{e \in E, e \leq s^*t} e \subseteq s^*s \cdot t^*t\).
  Take some \(x \in s^*s \setminus D_0\) on the left hand side.
  By the condition in the statement there is some \(j_0 \in \{1, \dots, k\}\) and idempotent \(e \in E\) such that \(x \in e t_{j_0}^*t_{j_0}\), \(se = t_{j_0} e\) and \(\closure{s^*se} \subseteq t_{j_0}^*t_{j_0}\).
  This precisely means that \(x\) is contained in the \(j_0\) sub-set of the right hand side.
\end{proof}

\begin{proof}[Proof of \cref{lemma:proper-actions-give-weak-containment}]
  Classically, this result would be shown by proving that if the action \(\Gamma \actson \primidealspc{B}\) is proper then it has the approximation property (\emph{i.e.}\ it is amenable), and thus has the ``weak containment property'', see \cite{buss-martinez-approx-prop-23}.
  In this setting, however, it is not yet known whether the \(D_0\)-approximation property, in the sense of \cref{def:approx-property}, implies the ``\(D_0\) weak containment property'' (\emph{i.e.}\ what we want to show), so it is best we show \(\Lambda_{D_0} \colon \maxalg[D_0]{\CB} \to \redalg[D_0]{\CB}\) is isometric again, by hand.

  Let \(\unitof{S} \coloneqq \{1\} \sqcup S\) be the minimal unitization of \(S\), and likewise let \(\unitof{\CB}\) be the minimal unitization of the Fell bundle \(\CB\).
  This means that \(\unitof{B_s} \coloneqq B_s\) for all \(s \in S\) and \(\unitof{B}_1 \coloneqq \unitof{B}\) is the minimal unitization of the \cstar{}algebra \(B\).
  This is a Fell bundle with the obvious operations, and we have that
  \[
    \maxalg[D_0]{\CB} \subseteq \maxalgnopar[D_0]{\unitof{\CB}} \;\; \text{ and } \;\; \redalg[D_0]{\CB} \subseteq \redalgnopar[D_0]{\unitof{\CB}}
  \]
  canonically, that is, the inclusions are just given by the identity maps \(B_s \to B_s\) on the fibers coming from \(\CB\).
  As the canonical quotient \(\Lambda_{\unitof{\CB}} \colon \maxalgnopar[D_0]{\unitof{\CB}} \to \redalgnopar[D_0]{\unitof{\CB}}\) extends \(\Lambda_{D_0}\) we may assume that \(B\) is unital.\footnote{\, Note that the Fell bundle \(\unitof{\CB}\) is not proper, but it does in the open set corresponding to \(\primidealspc{B}\). We hope this causes no confusion.}

  Recall we say \(b \in B\) is \emph{compactly supported}, cf.\ \cref{def:compactly-supported element}, if the set \(\supp{b} \coloneqq \{x \in X : b(x) \neq 0\}\) is contained in a compact subset of \(X\).
  By \cref{lemma:comp-supp-are-dense} the \Star{}algebra \(B_c\) formed by these elements in dense in \(B\).
  We first show that the ``compactly supported Fell bundle'' \(\CB_c \coloneqq (B_{c,s})_{s \in S}\),\footnote{\, This is not a Fell bundle, for \(B_c\) is not a \cstar{}algebra, nor is \(B_{c,s}\) a Banach space. Nevertheless, on a moral level we will regard \(\CB_{c}\) as an ``algebraic'' Fell bundle.} where \(B_{c,s} \coloneqq B_s \cdot B_c\), is dense in \(\CB\).
  By this we precisely mean that any \(a = \sum_{s \in F} a_s \in \algalg[D_0]{\CB}\) can be approximated up to a given \(\varepsilon > 0\) by some \(b = \sum_{s \in F} b_s\) such that \(b_s^*b_s \in B_c\) for all \(s \in F\) (equivalently \(b_s \in B_{c,s}\)).
  In order to do this, simply note that \(B_{c,s} = B_s \cdot B_c\) is dense in the Banach space \(B_s\).
  Thus we may approximate \(a_s \in B_s\) by some \(b_s \in B_{c,s}\) up to \(\varepsilon/\abs{F}\).
  The so obtained finite linear combination \(b = \sum_{s \in F} b_s\) satisfies that \(\norm{a - b} \leq \sum_{s \in F} \norm{a_s - b_s} \leq \varepsilon\), as desired.

  In order to show that the left regular representation \(\maxalg[D_0]{\CB} \to \redalg[D_0]{\CB}\) is isometric it suffices to prove, by the previous paragraph, that it is isometric in the dense sub-algebra spanned by the elements \(b\) of compact support.
  Thus, let \(b = \sum_{s \in F} b_s \in \algalg[D_0]{\CB_c}\) be some finite linear combination, where \(b_s = a_s a_{c,s} \in A_s A_c\) is as before.
  Let \(K \subseteq X\) be some compact set such that \(\bigcup_{s \in F} \supp{b_s^*b_s} \subseteq K\) (such a set exists by the choice of \(b\)).
  Suppose that \(1 + b \in \algalg[D_0]{\CB}\) is invertible as an element in \(\redalg[D_0]{\CB}\).
  We aim to prove that the inverse of \(1+b\) is then actually contained in \(\algalg[D_0]{\CB^{**}}\), so that \(1+b\) is also invertible in \(\maxalg[D_0]{\CB^{**}}\).
  This would finish the proof.
  The inverse of \(1+b\) is given by \((1+b)^{-1} = \sum_{n = 0}^\infty (-1)^n b^n \in \redalg[D_0]{\CB}\).
  By \(D_0\)-properness of the action the set
  \[
    \left\{s \in S : s \left(K \setminus D_0\right) \cap K \neq \emptyset\right\}
  \]
  is finitely upper bounded, say by \(\{t_1, \dots, t_k\} \subseteq S\).
  This in particular means that any \(s \in F\) is locally below some \(t_i\).
  For the following, and to simplify notation, we let \(B_t^{**}/J_{D_0}\) be \(B_t^{**}/J_{D_0} B_t^{**}\) (note that the former technically does not make sense).

  \begin{claim}
    \(b^n \in (B_{t_1}^{**} + \cdots + B_{t_k}^{**})/J_{D_0}\) for all \(n \in \NN\).
  \end{claim}
  \begin{proof}
    It suffices to show that all products \(b_{s_1} \cdots b_{s_n}\) are contained in \((B_{t_1}^{**} + \cdots + B_{t_k}^{**})/J_{D_0}\) for all finitely many \(s_1, \dots, s_n \in F\).
    We proceed by induction.
    If \(n = 1\) we may represent \(b_s = a_s a_{c,s} \in B_s \cdot B_c\) and proceed iteratively (in \(j = 1, \dots, k\)).
    Put \(b_{t_1} \coloneqq a_s a_{c,s} \cdot 1_{s, t_1}^{D_0}\) and, given \(b_{t_1}, \dots, b_{t_\ell}\), put
    \[
      b_{t_{\ell + 1}} \coloneqq a_s a_{c,s} \cdot 1_{s, t_\ell}^{D_0} \left(1 - 1_{s, t_1}^{D_0}\right) \cdots \left(1 - 1_{s, t_{\ell-1}}^{D_0}\right).
    \]
    As \(s\) is bounded above by \(\{t_1, \dots, t_k\}\) it follows that \(1_{s^*s}^{D_0} \leq \vee_{j = 1, \dots, k} 1_{s, t_j}^{D_0}\) by \cref{lemma:cover-s-by-ti}, which implies that \(b_s = \sum_{j = 1}^k b_{t_j}\).
    Each of the monomials of this sum belongs to (at least) one \(B_{t_j}^{**}/J_{D_0}\) by construction.

    For the induction step it is enough to show that \(b_s \cdot (b_{t_1} + \cdots + b_{t_k}) \in (B_{t_1}^{**} + \dots + B_{t_k}^{**})/J_{D_0}\) whenever \(b_{t_j} \in B_{t_j}^{**}/J_{D_0}\) is of compact support.
    Whence it is sufficient to show this for \(b_s \cdot b_{t_j}\), in which case it follows from the very same argument before.
  \end{proof}

  The claim implies \((1+b)^{-1} \in (B_{t_1}^{**} + \cdots + B_{t_k}^{**})/J_{D_0}\).
  This means \(1+b\) is invertible in \(\maxalg[D_0]{\CB^{**}}\), as desired, for the latter \(\maxalg[D_0]{\CB^{**}}\) contains both \((B_{t_1}^{**} + \cdots + B_{t_k}^{**})/J_{D_0}\) and \(\maxalg[D_0]{\CB}\) canonically.
\end{proof}

\section{A Baum-Connes assembly map on inverse semigroup equivariant E-theory} \label{sec:equiv-eth}

With the work in the previous sections done we are now in a position to define a Baum-Connes assembly map.
For this, we will follow the blueprint of \cite{Guentner-Higson-Trout}*{Chapters 6 and 10}, where this is achieved for (say discrete and countable) groups. 
We start by making precise the desired equivariant \Eth-theory of \cref{sec:non-equiv-eth}.

\begin{remark}
  Because of the discussion around the notion of ``reduced'' morphism in \cref{def:all-equiv-morphism} we will restrict ourselves to discussing only ``maximal'' \Eth-theory. A ``reduced'' analogue will not be studied here.
  For an example of why this is needed see \cite{Kwasniewski-Meyer:Pure_infiniteness}*{Example 4.7}: it is due to the discontinuities of \cref{rem:non-hausdorffness}.
\end{remark}

\begin{definition}
  Two \(D_0\)-equivariant (in the sense of \cref{def:all-equiv-morphism}) asymptotic morphisms \(\varphi_0, \varphi_1 \colon \CA \asymparrow \CB\) are \emph{homotopic} if there is a \((S,D_0)\)-equivariant morphism \(\varphi \colon \CA \to (\cont{[0,1]} \tensor \CB)_\fc\) with end points \(\varphi_0\) and \(\varphi_1\) at \(0\) and \(1\) respectively.
\end{definition}

Suffice to say that being homotopic is an equivalence relation: the usual proof goes through equally well.
We may now give the following.

\begin{definition}
  An \((S,X)\)-\emph{Hilbert space} is a Hilbert space \(H\) equipped with representations \(S \to {\rm PIsom}(H)\) (via partial isometries) and \(\opensets{X} \to {\rm Proj}(H)\) (via projections) that are, moreover, compatible in the following covariant way
  \[
    s p_{\CU} s^* = p_{s \cdot (\CU \cap s^*s)}
  \]
  for all \(s \in S\) and \(\CU \subseteq X\) open set.
  The \emph{standard \((S,X,D_0)\)-Hilbert space} is \(\CH_{S,D_0} \coloneqq \ell^2(\NN) \tensor \ell^2(S \ltimes (X \setminus D_0))\), equipped with the usual left regular representations:
  \[
    s \cdot \delta_n \tensor \delta_{\left[t,x\right]} \coloneqq \delta_n \tensor \delta_{\left[st,x\right]}
  \]
  if \(tx \in s^*s\), and \(0\) otherwise, and \(p_{\CU}\) is the projection onto \(\ell^2(\NN) \tensor \ell^2(\{[s,x] : x \in \CU\})\) for all \(\CU \in \opensets{X}\).
\end{definition}

The following contains the key fact about the standard \(S\)-Hilbert space \(\CH_{S,D_0}\).
\begin{lemma} \label{lemma:k-sx}
  If \(\CK_{S,D_0} \coloneqq \CK(\CH_{S,D_0})\), then there is a, unique up to homotopy, map
  \[
    \CK_{S,D_0} \oplus \CK_{S,D_0} \subseteq \CK\left(\CH_{S,D_0} \oplus \CH_{S,D_0}\right) \cong \CK_{S,D_0}.
  \]
\end{lemma}
\begin{proof}
  The proof is essentially the same as the one given in \cite{Guentner-Higson-Trout}*{Lemma 6.3} and discussion thereafter. We sketch it here for convenience.

  First we show any isometry \(V \colon \CH_{S,D_0} \to \CH_{S,D_0}\) is connected to the identity.
  Writing \(\CH_{S,D_0} = \ell^2(\NN) \tensor \CH\), and seeing \(\ell^2(\NN)\) as a trivial \((S,X)\) Hilbert space, we may choose a \Star{}strongly continuous path of isometries \(\{W_r\}_{r \in (0,1]}\) such that \(W_rW_r^* \to 0\) when \(r \to 0\) in the strong operator topology.
  Then \((1 \tensor W_r) V (1 \tensor W_r)^* + 1 - 1 \tensor W_rW_r^*\) connects \(V\) to \(1\).

  It follows that any two spatial \Star{}homomorphisms \(\CK_{S,D_0} \to \CK_{S,D_0}\) are homotopic, and since \(\CH_{S,D_0} \oplus \CH_{S,D_0} \cong \CH_{S,D_0}\) equivariantly it follows that the map
  \[
    \CK_{S,D_0} \oplus \CK_{S,D_0} \subseteq \CK\left(\CH_{S,D_0} \oplus \CH_{S,D_0}\right) \to \CK_{S,D_0}
  \]
  is unique up to homotopy.
\end{proof}

\begin{definition} \label{def:equiv-eth}
  Let \(\CA\) and \(\CB\) be Fell bundles over \((S, X)\)-algebras \(A\) and \(B\) respectively. Given any invariant \(D_0 \subseteq X\), we let \(\llbracket \CA, \CB \rrbracket_{S,D_0}\) be the set of homotopy classes of \(D_0\)-equivariant asymptotic morphisms \(\CA \asymparrow \CB\).
  In the same vein, we let
  \[
    \Egrp[S,D_0]{\CA}{\CB} \coloneqq \llbracket \czeror \tensor \CA \tensor \CK_{S,D_0}, \czeror \tensor \CB \tensor \CK_{S,D_0} \rrbracket_{S,D_0},
  \]
  where \(\czeror \tensor \CA \tensor \CK_{S,D_0}\) is the trivial Fell bundle extension of \(\CA \tensor \CK_{S,D_0}\), and similarly for \(\CB \tensor \CK_{S,D_0}\) (recall \cref{lemma:fell-bundle-a-tensor-b-trivial}).
\end{definition}

\begin{example}
  \cref{def:equiv-eth} evidently generalizes the group case as discussed in, for instance, \cite{Guentner-Higson-Trout}*{Definition 6.8}.
  On the other end of the spectrum, recall \cref{ex:opensets-x}, and let \(S \coloneqq \opensets{X}\) and \(D_0 \coloneqq \emptyset \subseteq X\).
  Then asymptotic morphisms \(\varphi \colon \CA \asymparrow \CB\) are (asymptotically equivalent to) \emph{\(X\)-equivariant} asymptotic morphisms \(\varphi \colon A \asymparrow B\) in the sense of \cite{dadarlat-meyer-2012}*{Definition 2.4}.
  Whence our definition of \Eth-theory generalizes that of Dadarlat and Meyer in \cite{dadarlat-meyer-2012}*{Definition 2.14}.
\end{example}

\begin{example}
  The ``ideal related'' \(X\)-\KKth-theory defined by Gabe in \cite{gabe-memoirs-2024}*{Section 12} can be seen as a particular case of \(\Egrp[S,D_0]{\CA}{\CB}\) at least if the algebras \(A\) and \(B\) are nuclear.  
  Indeed, the group \(\Egrp[\opensets{X},\emptyset]{A}{B}\) can then be seen as \(\KKgrp[nuc]{X; A}{B}\), see \cite{gabe-memoirs-2024}*{Definition 12.11}.
  Asymptotic morphisms \(\varphi \colon \CA \asymparrow \CB\) are just \(X\)-asymptotic morphisms \(A \asymparrow B\), see \cref{ex:opensets-x}.
  These, in turn, are given by Cuntz pairs that are ``weakly \(X\)-equivariant'' in the sense of \cite{gabe-memoirs-2024}*{Definition 12.17} whenever \(A\) and \(B\) are nuclear.
\end{example}

One should check that \(\Egrp[S,D_0]{\variable}{\variable}\) is a well behaved functor: we start showing it is an abelian group.

\begin{lemma}
  The set \(\llbracket \CA, \czeror \tensor \CB \tensor \CK_{S,D_0}\rrbracket\) is an abelian group, where given \(\varphi, \psi \colon \CA \asymparrow \czeror \tensor \CB \tensor \CK_{S,D_0}\) the sum \([\varphi]_h + [\psi]_h\) is defined as the homotopy class of
  \[
    \varphi \oplus \psi \colon \CA \asymparrow \left(\czeror \tensor \CB \tensor \CK_{S,D_0}\right) \oplus \left(\czeror \tensor \CB \tensor \CK_{S,D_0}\right) \cong \czeror \tensor \CB \tensor \CK_{S,D_0},
  \] 
  where the last isomorphism comes from any choice of embedding \(\CK_{S,D_0} \oplus \CK_{S,D_0} \subseteq \CK_{S,D_0}\) as in \cref{lemma:k-sx}.
\end{lemma}
\begin{proof}
  We only sketch the proof due to the importance of the result, but the method is standard.
  Observe that \(\varphi \oplus \psi\) does define a \(D_0\)-equivariant asymptotic morphism.
  Thus, all we have to do is to prove that the operation is commutative and that every element admits an inverse.
  The former follows in the same way as it does for groups, see \cite{Guentner-Higson-Trout}*{Lemma 6.6}: the maps
  \[
    \left(T, S\right) \mapsto T \oplus S \;\; \text{ and } \;\; \left(T,S\right) \mapsto S \oplus T
  \]
  are homotopy-equivalent.
  The latter condition is also proved as it is in groups: the inverse of \(\varphi \colon \CA \asymparrow \czeror \tensor \CB \tensor \CK_{S,D_0}\) is \((\rho \tensor 1 \tensor 1) \circ \varphi\), where \(\rho(f)(x) \coloneqq f(-x)\) for all \(f \in \czeror\). 
\end{proof}

\begin{corollary}
  \(\Egrp[S,D_0]{\CA}{\CB}\) is an abelian group.
\end{corollary}

We now check several properties of the functor \(\Egrp[S,D_0]{\variable}{\variable}\).
We are particularly interested in the composition product; constructing a ``descent'' functor and proving that extensions give rise to elements in equivariant \Eth-theory.

\begin{proposition} \label{prop:eth-composition}
  The assignment
  \[
    \Egrp[S,D_0]{\CA}{\CB} \times \Egrp[S,D_0]{\CB}{\CCC} \to \Egrp[S,D_0]{\CA}{\CCC}, \;\; \left(\left[\varphi\right]_h, \left[\psi\right]_h\right) \mapsto \left[\psi \circ \varphi\right]_h,
  \]
  where \([\psi \circ \varphi]_h\) is the homotopy class of (possibly a reparametrization of) the composition of the \((S, D_0)\)-equivariant asymptotic morphisms \(\varphi\) and \(\psi\) is a well defined composition law.
\end{proposition}
\begin{proof}
  One should check that the composition of \(D_0\)-equivariant asymptotic morphisms is really a \(D_0\) equivariant asymptotic morphism.
  The ``asymptotic morphism'' part is clear: the bullet points in \cref{def:morphism} are all preserved under taking compositions (at least modulo reparametrizing \([1,\infty)\)).
  In order to check that the \(D_0\)-equivariance of asymptotic morphisms is preserved recall that (in the notation of \cref{prop:inject-equiv-map-induces}):
  \[
    \varphi_{1,\fc}^{**}\left(J_{D_0}^\CA\right) \subseteq J_{D_0,\fc}^\CB \;\; \text{ and } \;\; \psi_{1,\fc}^{**}\left(J_{D_0}^\CB\right) \subseteq J_{D_0,\fc}^\CCC,
  \]
  whence
  \[
    \left(\psi \circ \varphi\right)_{1,\fc}^{**}\left(J_{D_0}^\CA\right) \subseteq \psi_{1,\fc}^{**}\left(J_{D_0,\fc}^\CB\right) \subseteq J_{D_0,\fc}^{\CCC},
  \]
  as desired.
  The fact that the composition law is associative is clear, so the proof is omitted.
\end{proof}

The following details the so-called ``descent map'', and is a consequence of \cref{prop:inject-equiv-map-induces} and the usual Bott periodicity of \Eth-theory.
\begin{proposition} \label{prop:descent}
  There is a ``descent functor''
  \[
    \Egrp[S,D_0]{\CA}{\CB} \to \Egrp{\maxalg[D_0]{\CA}}{\maxalg[D_0]{\CB}}
  \]
  from the homotopy category of \((S, X)\)-\cstar{}algebras with \((S, D_0)\)-equivariant asymptotic morphisms to the homotopy category of \cstar{}algebras and asymptotic morphisms.
  It associates to the class of an \((S, D_0)\)-equivariant asymptotic morphism \(\varphi \colon \czeror \tensor \CA \tensor \CK_{S,D_0} \asymparrow \czeror \tensor \CB \tensor \CK_{S,D_0}\) the homotopy class of the induced asymptotic morphism \(\varphi \colon \czeror \tensor \maxalg[D_0]{\CA} \tensor \CK_{S,D_0} \asymparrow \czeror \tensor \maxalg[D_0]{\CB} \tensor \CK_{S,D_0}\) via \cref{prop:inject-equiv-map-induces}.
\end{proposition}
\begin{proof}
  Observe that \cref{prop:inject-equiv-map-induces} yields an asymptotic morphism \(\maxalg[D_0]{\czeror \tensor \CA \tensor \CK_{S,D_0}} \asymparrow \maxalg[D_0]{\czeror \tensor \CB \tensor \CK_{S,D_0}}\), so all we have to do is to construct a \Star{}isomorphism
  \[
    \Phi \colon \maxalg[D_0]{\czeror \tensor \CA \tensor \CK_{S,D_0}} \to \czeror \tensor \maxalg[D_0]{\CA} \tensor \CK_{S,D_0},
  \]
  and similarly for \(\CB\). This is precisely the contents of \cref{prop:trivial-extensions-maxalg}.
\end{proof}

\begin{remark}
  The reduced analogue of \cref{prop:descent} is \emph{false}, due to the complexity of inverse semigroups.
  This means that, in general, there is no ``reduced descent'' map
  \[
    \Egrp[S,D_0]{\CA}{\CB} \to \Egrp{\redalg[D_0]{\CA}}{\redalg[D_0]{\CB}}.
  \]
  Indeed, recall \cref{prop:inject-equiv-map-induces}: an asymptotic morphism \(\varphi \colon \CA \asymparrow \CB\) is only guaranteed to induce a map on \(\redalg[D_0]{\variable}\) if it is reduced itself, see \cref{def:all-equiv-morphism}.
  So the reduced descent map above has no reason to exist. 
  In fact, it does not, see \cite{Kwasniewski-Meyer:Pure_infiniteness}*{Example 4.7} for an easy example of a morphism \(\varphi \colon \CA \to \CB\) that does not induce any map \(\redalg{\CA} \to \redalg{\CB}\).
\end{remark}

Note that the discussion around equivariant extensions \(0 \to \CI \to \CA \to \CB \to 0\) in \cref{thm:maximal-extensions,thm:reduced-extensions} makes one wonder whether such extensions induce elements in \Eth-theory.
This is so.
\begin{theorem} \label{prop:ext-eth}
  Let \(\eta \colon 0 \to \CI \xrightarrow{\iota} \CA \xrightarrow{q} \CB \to 0\) be a \(D_0\)-equivariant short exact sequence of separable Fell bundles.\footnote{\, We say a Fell bundle over a \cstar{}algebra \(A\) is \emph{separable} if every fiber \(A_s\) is separable as a Banach space. In particular, \(A\) is separable as a \cstar{}algebra.}
  There is an associated element \([\eta]_h \in \Egrp[S,D_0]{\CB}{\CI}\) that, moreover, maps to the class of the extension
  \[
    \left[\maxalg[D_0]{\eta}\right]_h \colon 0 \to \maxalg[D_0]{\CI} \to \maxalg[D_0]{\CA} \to \maxalg[D_0]{\CB} \to 0
  \]
  of \cref{thm:maximal-extensions} via the descent functor of \cref{prop:descent}.
\end{theorem}

Before proving \cref{prop:ext-eth} recall one basic feature of non-equivariant \Eth-theory: any short exact sequence \(0 \to I \to A \to B \to 0\) induces an asymptotic morphism \(\czeror \tensor B \asymparrow I\). Indeed, one may choose any (set-theoretic only) section \(\zeta \colon B \to A\) and let
\begin{equation} \label{eq:central-invariant}
  \varphi \colon \czeror \tensor B \asymparrow I, \;\; \varphi_r\left(f \tensor b\right) \coloneqq f\left(u_r\right) \zeta\left(b\right),
\end{equation}
where \(\{u_r\}_{r \in [1, \infty)}\) is an approximate unit of \(I\).
This can be checked to be asymptotically equivalent to an asymptotic morphism \(\czeror \tensor B \asymparrow I\), and thus defines an element in \(\Egrp{B}{I}\).
In the proof of \cref{prop:ext-eth} we aim to do something virtually equal.
Note that in the classical group case this goes back to a fundamental theorem of Kasparov, see \cite{Kasparov-1988}*{Theorem 1.4}.
Thus, the following lemma is the key technical point we will need: it gives us the adequate notion of \emph{quasi-central approximate unit of \(\CI \subseteq \CA\)}.

\begin{lemma} \label{lemma:quasi-central-approx-unit}
  Suppose \(\CI \subseteq \CA\); the \cstar{}algebra \(A\) is separable and \(I \subseteq A\) is an ideal.
  Then there is a norm-continuous family of positive contractions \(\{u_{r}\}_{r \in [1,\infty)} \subseteq I\) such that:
  \begin{enumerate}
    \item \label{lemma:quasi-central-approx-unit:i} \(\norm{u_{r} i - i} \to 0\) when \(r \to \infty\) for all \(i \in I\); 
    \item \label{lemma:quasi-central-approx-unit:ii} \(\norm{u_{r} a_s - a_s u_{r}} \to 0\) when \(r \to \infty\) for all \(a_s \in A_s\); and
    \item \label{lemma:quasi-central-approx-unit:iii} \(\norm{f(u_r) a_s - a_s f(u_r)} \to 0\) for all \(f \in \czeror\).
  \end{enumerate}
\end{lemma}
\begin{proof}
  This is mostly an application of \cref{prop:inject-equiv-map-induces} and \cite{arveson-1977}*{Theorem 1}.
  The inclusion \(\CI \subseteq \CA\), by \cref{prop:inject-equiv-map-induces}, yields an inclusion \(\maxalg[D_0]{\CI} \subseteq \maxalg[D_0]{\CA}\) of an ideal in a \cstar{}algebra.
  Using \cite{arveson-1977}*{Theorem 1} we obtain an increasing sequence \(\{u_k\}_{k \in \NN} \subseteq \maxalg[D_0]{\CI}\) of positive contractions such that \(u_k i_s - i_s \to 0\) for all \(i_s \in I_s\); and \(\norm{u_k a_s - a_s u_k} \to 0\) for all \(a_s \in A_s\).
  It hence suffices to prove that \(u_k\) may be contained in the \cstar{}algebra \(I_{D_0}\), which is (by definition) the image of the canonical \Star{}homomorphism \(I \to \maxalg[D_0]{\CI}\).
  The key observation here is that the proof of \cite{arveson-1977}*{Theorem 1} actually shows that \(\{u_k\}_{k \in \NN}\) may be chosen to be contained in the convex hull of an approximate unit (without assuming quasi-centrality) of \(\maxalg[D_0]{\CI}\).
  As \(I_{D_0}\) is convex (it is a \cstar{}algebra), it follows from \cref{lemma:approx-unit} that \(\{u_k\}_{k \in \NN}\) will be completing the proof of the first two items.
  Item \cref{lemma:quasi-central-approx-unit:iii}, in turn, is a consequence of \cref{lemma:quasi-central-approx-unit:ii} and the Stone-Weierstrass theorem.
\end{proof}

\begin{proof}[Proof of \cref{prop:ext-eth}]
  Given \(0 \to \CI \to \CA \to \CB \to 0\) as in the statement, choose an approximate unit \(\{u_{r}\}_{r \in [1,\infty)}\) of \(I\) as in \cref{lemma:quasi-central-approx-unit}.
  As hinted at in the discussion around \eqref{eq:central-invariant} we aim to define a central invariant.
  For this, take any (set-theoretic only) sections \(\zeta_s \colon B_s \to A_s\) for all \(s \in S\); and define the tuple
  \begin{align*}
    \varphi \colon \czeror \odot \CB & \asymparrow \CI, \\
    \varphi_{s,r}\left(f \tensor b_s\right) & \coloneqq f\left(u_{r}\right) \zeta_{s}\left(b_s\right).
  \end{align*}
  We claim that \(\varphi = (\varphi_{s,r})_{s \in S, r \in [1,\infty)}\) defines an asymptotic morphism in the sense of \cref{def:morphism}.
  First note that this essentially suffices: by the universal property of \(\czeror \tensor \CB\) and \cref{prop:inject-equiv-map-induces} we will have that \(\varphi\) extends to a \Star{}homomorphism \(\czeror \tensor \maxalg[D_0]{\CB} \to \maxalg[D_0]{\CI}_\fc\).
  This, as is well known, is just an asymptotic morphism \(\czeror \tensor \maxalg[D_0]{\CB} \asymparrow \maxalg[D_0]{\CI}\), so that it is indeed enough to work within the algebraic tensor product only.
  
  Let \(f \tensor b_s \in \czeror \odot B_s\) be given.
  We start by noticing that \(\varphi_{s,r}(f \tensor b_s) \in I \cdot A_s \subseteq I_s\) is in the correct fiber.
  Moreover, for two pairs of basic tensors \(f_i \tensor b_{i} \in \czeror \tensor B_{s_i}\) we have that
  \[
    \varphi_{r,s_1s_2}\left(f_1 \tensor b_1 \cdot f_2 \tensor b_2\right) = f_1 \left(u_r\right) f_2\left(u_r\right) \zeta_{s_1s_2} \left(b_1b_2\right)
  \]
  whereas
  \[
    \varphi_{r,s_1}\left(f_1 \tensor b_1\right) \varphi_{r,s_2}\left(f_2 \tensor b_2\right) = f_1\left(u_r\right) \zeta_{s_1}\left(b_1\right) f_2\left(u_r\right) \zeta_{s_2}\left(b_2\right) \approx f_1\left(u_r\right) f_2\left(u_r\right) \zeta_{s_1}\left(b_1\right) \zeta_{s_2}\left(b_2\right),
  \]
  where the asymptotic equivalence in the middle is due to \cref{lemma:quasi-central-approx-unit}~\cref{lemma:quasi-central-approx-unit:iii}.
  Taking into account the assumption that \(\zeta_{s} \colon B_s \to A_s\) is a section of the quotient map \(A_s \to B_s\) for all \(s \in S\), we may notice that \(x \coloneqq \zeta_{s_1s_2}(b_1 b_2) - \zeta_{s_1}(b_1) \zeta_{s_2}(b_2) \in I_{s_1s_2}\).
  Whence, applying \cref{lemma:quasi-central-approx-unit}~\cref{lemma:quasi-central-approx-unit:i}:
  \[
    \varphi_{r,s_1s_2}\left(f_1 \tensor b_1 \cdot f_2 \tensor b_2\right) - \varphi_{r,s_1}\left(f_1 \tensor b_1\right) \varphi_{r,s_2}\left(f_2 \tensor b_2\right) \approx \left(f_1f_2\right)\left(u_r\right) x \to 0.
  \]
  The second bullet point of \cref{def:morphism} is proven similarly.
  For it, let \(f \tensor b_s \in \czeror \tensor B_s\), where \(s \leq t\), so that we may see \(f \tensor b_s \in \czeror \tensor B_s \subseteq \czeror \tensor B_t\) in the usual way.
  We have that \(\varphi_{r,s}(f \tensor b_s) = f(u_r) \zeta_s(b_s)\), and similarly for \(t\), whence
  \[
    \varphi_{r,s}\left(f \tensor b_s\right) - \varphi_{r,t}\left(f \tensor b_s\right) \approx f\left(u_r\right) \left(\zeta_s\left(b_s\right) - \zeta_t\left(b_s\right)\right).
  \]
  As before, whichever the value of \(\zeta_t(b_s) \in A_s\) is, it has to map to \(b_s\) under \(A_s \to B_s\), and thus \(\zeta_s(b_s) - \zeta_t(b_s) \in I_s\), whence the second bullet of \cref{def:morphism} again follows from \cref{lemma:quasi-central-approx-unit}~\cref{lemma:quasi-central-approx-unit:i}.
  
  Now that we have an induced asymptotic morphism \(\varphi_\eta \colon \czeror \tensor \CB \asymparrow \CI\) the rest of the proof is just applying \cref{prop:inject-equiv-map-induces} to get the extension and use that all the constructions are functorial.
\end{proof}

We end the subsection with a brief discussion of why we only ever discuss \Eth-theory, as opposed to \KKth-theory.

\begin{remark} \label{rem:why-not-kk}
  Were we to attempt to define an equivariant-\KKth-category then we would typically require extensions \(0 \to \CI \to \CA \to \CB \to 0\) to induce elements in some hypothetical \(\KKgrp[S,D_0]{\CB}{\CI}\) only if the quotient \(\CA \to \CB\) admits an equivariant \cpc split \(\sigma \colon \CB \to \CA\).
  This is done, for instance, for groups, see \cite{higson-kasparov-2000}*{Section 8}, and should be compared to \cref{prop:ext-eth}.
  This is asking too much.
  For instance, even if \(\CB\) arose from a proper, \'etale groupoid and the arguments in \cite{higson-kasparov-2000}*{Theorem 8.1}  were replicated, the splitting \(\sigma \colon \CB \to \CA\), in general, would only land in \(\CA^{**}\) (or, rather, in the Borel \cstar{}algebra \(\borelalg{A}\) described before \cref{thm:ess-alg-is-jd}, or even in the algebra constructed in \cite{brix-gonzalez-hume-li-2025}).
  Roughly speaking \(\CA\) does not admit sufficiently many discontinuities.
  (This was, in fact, one of the main problems of the approach in \cite{tu-baum-connes-2004}, see also the comments about \cite{renault-2013} in \cite{buss-martinez-ess-ame-2024}.)
\end{remark}

\subsection{The assembly map}
With the composition product and the previous discussion at hand we may almost construct the assembly map.
Before doing so, for the rest of the section if \(S \actson Y\) is an action we will assume that \(Y\) is a locally compact, \(S\)-compact, \(D_0\)-proper, \emph{Hausdorff} space.

Given \(v_1 \in \Egrp[S,D_0]{\cfunz{Y_1}}{\CB}\) and \(v_2 \in \Egrp[S,D_0]{\cfunz{Y_2}}{\CB}\) two elements in their own \(D_0\)-proper equivariant \Eth-theories, we say \(v_1 \sim_{D_0} v_2\) if there is some \(D_0\)-proper \(S\)-space \(Z\) and equivariant maps \(\varphi_1 \colon Y_1 \to Z\) and \(\varphi_2 \colon Y_2 \to Z\) such that \(\widetilde{\varphi_{1}} v_1 = \widetilde{\varphi_{2}} v_2\), where \(\widetilde{\varphi_i} \colon \cfunz{Z} \to \cfunz{Y}\) is the Gelfand dual of \(\varphi_i\).
Notice that \(\widetilde{\varphi_i}\) is automatically equivariant due to \(\varphi_i\) being equivariant.
Likewise, we see
\[
  \widetilde{\varphi_1} v_1 \in \Egrp[S,D_0]{\cfunz{Z}}{\cfunz{Y}} \cdot \Egrp[S,D_0]{\cfunz{Y}}{\CB} \subseteq \Egrp[S,D_0]{\cfunz{Z}}{\CB}
\]
as elements in \(\Egrp[S,D_0]{\cfunz{Z}}{\CB}\) via the composition product of \cref{prop:eth-composition}.
In particular it follows that \(\sim_{D_0}\) is a sensible notion.
It in fact even is an equivalence relation.
We are thankful to Ali Miller for fruitful conversations on the following proof.

\begin{lemma}
  With notation as above, \(\sim_{D_0}\) defines an equivalence relation.
\end{lemma}
\begin{proof}
  Reflexivity of \(\sim_{D_0}\) is clear, as we only consider \emph{\(D_0\)-proper} equivariant \Eth-theories.
  Likewise, symmetry of \(\sim_{D_0}\) is also clear.
  In order to prove transitivity, suppose that \(v_i \in \Egrp[S,D_0]{\cfunz{Y_i}}{\CB}\), where \(i = 1, 2, 3\), and \(v_1 \sim_{D_0} v_2 \sim_{D_0} v_3\).
  Let these relations be witnessed by \(D_0\)-proper spaces \(Z_1\) and \(Z_2\) and equivariant maps \(\varphi_1 \colon Y_1 \to Z_1, \varphi_2 \colon Y_2 \to Z_1\) and \(\psi_1 \colon Y_2 \to Z_2, \psi_2 \colon Y_3 \to Z_2\).
  Our aim is to construct an \(S\)-space \(W\) that is \(S\)-compact, \(D_0\)-proper and fits into the following commuting diagram:
  \begin{equation*}
    \begin{tikzcd}[scale=50em]
      Y_1 \arrow{r}{\varphi_1} & Z_1 \arrow{r}{\rho} & W \\
      & Y_2 \arrow{u}{\varphi_2} \arrow{r}{\psi_1} & Z_2 \arrow{u}{\sigma} \\
      & & Y_3 \arrow{u}{\psi_2},
    \end{tikzcd}
  \end{equation*}
  for some equivariant maps \(\rho\) and \(\sigma\).
  Should we manage to construct such a \(W\) then:
  \[
    \widetilde{\rho \varphi_1} v_1 = \widetilde{\rho} \cdot \widetilde{\varphi_1} v_1 = \widetilde{\rho} \cdot \widetilde{\varphi_2} v_2 = \widetilde{\rho \varphi_2} v_2 = \widetilde{\sigma \psi_1} v_2 = \widetilde{\sigma} \cdot \widetilde{\psi_1} v_2 = \widetilde{\sigma} \cdot \widetilde{\psi_2} v_3 = \widetilde{\sigma \psi_2} v_3, 
  \]
  where the commutation of the diagram is precisely used in order to guarantee that \(\rho \varphi_2 = \psi_1 \sigma\).
  Given \(z_1 \in Z_1\) and \(z_2 \in Z_2\) say that \(z_1 \sim_{Y_2} z_2\) if there is some \(y_2 \in Y_2\) such that \(\varphi_2(y_2) = z_1\) and \(\psi_1(y_2) = z_2\). 
  Write \(W \coloneqq Z_1 \sqcup Z_2/\sim_{Y_2}\), which we see as the disjoint union of \(Z_1\) and \(Z_2\) over \(Y_2\).
  One can see that \(\sim_{Y_2}\) defines an equivalence relation (or, rather, its trivial extension where \(z_i \sim_{Y_2} z_i\) for all \(z_i \in Z_i\)), and thus indeed the quotient \(W\) makes sense as a topological space.
  Notice, furthermore, that it is locally compact and Hausdorff, for \(\sim_{Y_2}\) is closed: if \(\varphi_2(y_\lambda) = z^{(1)}_\lambda \sim_{Y_2} z_\lambda^{(2)} = \psi_1(y_\lambda)\) each converge to \(z^{(1)}\) and \(z^{(2)}\) respectively, it follows from the properness of the maps \(\varphi_2\) and \(\psi_1\) that (passing to a subnet if necessary) \(y_\lambda \to y\), and from continuity: \(z^{(1)} = \varphi_2(y)\) and \(z^{(2)} = \psi_1(y)\), proving that \(\sim_{Y_2}\) is closed.
  The canonical maps \(\rho \colon Z_1 \subseteq Z_1 \sqcup Z_2 \to W\) and \(\sigma \colon Z_2 \subseteq Z_1 \sqcup Z_2 \to W\) then make the diagram above commute, finishing the proof.
\end{proof}

\begin{definition} \label{def:etop}
  Let \(\CB = (B_s)_{s \in S}\) be a Fell bundle over an \((S,X)\)-\cstar{}algebra \(B\).
  We define \(\Etop{S, D_0}{\CB}\) to be the abelian group \(E^{\rm top}_{0}(S, D_0; \CB) \oplus E^{\rm top}_{1}(S, D_0; \CB)\), where:
  \begin{enumerate}
    \item \(E^{\rm top}_{0}(S, D_0; \CB)\) is generated by \(\Egrp[S,D_0]{\cfunz{Y}}{\CB}\), where \(Y\) is a locally compact, \(S\)-compact, \(D_0\)-proper, Hausdorff \(S\)-space, divided by the equivalence relation \(\sim_{D_0}\).
    \item \(E^{\rm top}_{1}(S, D_0; \CB) \coloneqq E^{\rm top}_{0}(S, D_0; \czeror \tensor \CB)\).
  \end{enumerate}
\end{definition}

\begin{remark}
  We define \(\Etop{S,D_0}{\CB}\) in the way we do because it is not clear whether a Fell bundle admits a ``universal space for proper actions'', see \cite{valette-2002-book}.
  In fact this is not even clear for actions \(S \actson X\).
\end{remark}

We call \(E^{\rm top}_{0}(S, D_0; \CB)\) the ``grade \(0\)'' term of \(\Etop{S,D_0}{\CB}\), whereas \(E^{\rm top}_{1}(S, D_0; \CB)\) is the ``grade \(1\)''.
They correspond, under the assembly map, precisely to \(\Kgrp[0]{\variable}\) and \(\Kgrp[1]{\variable}\) respectively.
The following constructs the assembly map we are looking for.
For it, note that if \(S \actson Y\) is \(D_0\)-proper (and \(Y\) is Hausdorff) then there is a projection \(p_Y \in \maxalg[D_0]{S \ltimes Y}\) that is, moreover, unique up to homotopy, see \cref{prop:proj-proper-action}.
\begin{theorem} \label{thm:assembly}
  \(\Etop{S, D_0}{\CB}\) is an abelian group.
  Moreover, the assignment
  \[
    \mu_{D_0, \CB} \colon \Etop{S, D_0}{\CB} \to \Kgrp{\maxalg[D_0]{\CB}} \to \Kgrp{\redalg[D_0]{\CB}}
  \]
  that is, in grade \(0\), locally defined by
  \[
    \Egrp[S,D_0]{\cfunz{Y}}{\CB} \to \Egrp{\maxalg[D_0]{S \ltimes Y}}{\maxalg[D_0]{\CB}} \to \Kgrp[0]{\maxalg[D_0]{\CB}} \to \Kgrp[0]{\redalg[D_0]{\CB}},
  \]
  where:
  \begin{enumerate}
    \item the first map is the descent functor of \cref{prop:descent};
    \item the second map is left multiplication with \([p_Y]_0 \in \Kgrp[0]{\maxalg[D_0]{S \ltimes Y}} \cong \Egrp{\CC}{\maxalg[D_0]{S \ltimes Y}}\);
    \item and the third map is the one induced by the canonical quotient map \(\maxalg[D_0]{\variable} \to \redalg[D_0]{\variable}\);
  \end{enumerate}
  defines a group homomorphism.
  The map \(\mu_{D_0, \CB}\) in grade \(1\) is defined analogously.
\end{theorem}
\begin{proof}
  All we have to do is show that the map thus locally defined is well defined.
  Indeed, if it is then the desired composition is a group morphism by the virtue of the product in \Eth-theory being defining a homomorphism and \cref{prop:descent}.
  Suppose thus that \(v_i \in \Egrp[S,D_0]{\cfunz{Y_i}}{\CB}\), where \(i = 1, 2\), are such that \(v_1 \sim_{D_0} v_2\).
  Let \(\varphi_1 \colon Y_1 \to Z\) and \(\varphi_2 \colon Y_2 \to Z\) be \(D_0\)-equivariant maps to a common \(Z\).
  Let \(c \colon Z \to [0,1]\) be a cut-off function for \(Z\), see \cref{prop:cut-off-functions}.
  The Gelfand duals \(\widetilde{\varphi_1}\) and \(\widetilde{\varphi_2}\) are then equivariant, and by \cref{prop:inject-equiv-map-induces} they both induce maps on the respective \cstar{}algebras \(\maxalg[D_0]{\variable}\), which we denote by the same name by abuse of notation.
  By \cref{prop:hom-cut-off} it suffices to show that \(\widetilde{\varphi_i}(p_c)\) is a projection arising from a cut-off function of \(Y_i\), since all such projections are homotopic (and thus define the same maps on \Kth-theory).

  Dropping the subscript \(i = 1, 2\) (as the proofs are the same), it suffices to prove that the pull back \(c \varphi\) is a cut-off function for \(Y\).
  The only non-trivial feature to check is that the support of \(c \varphi\) has support contained in a compact set, which holds due to the same being true for \(c\) and \(\varphi\) being proper.
\end{proof}

\begin{definition}
  If \(\CB\) is a Fell bundle over an \((S,X)\)-algebra \(B\) and \(D_0 \subseteq X\) is invariant, we let
  \[
    \mu_{D_0, \CB} \colon \Etop{S,D_0}{\CB} \to \Kgrp{\redalg[D_0]{\CB}}
  \]
  be the map defined in \cref{thm:assembly}. We call it the \emph{Baum-Connes assembly map} for the tuple \((S, X, D_0, \CB)\).
\end{definition}

Note that it follows from \cref{thm:assembly} that the assembly map is always a group homomorphism.
Whether or not it is an isomorphism for a certain class of actions \(S \actson X\) will be studied elsewhere.
For now we observe two features.
First, it was proved in \cites{higson-lafforgue-skand-2002-counter-bc,martinez-szakacs-2025} that \(\mu_{D_0, \CB}\) cannot always be surjective.
Secondly, the following translates to equivariant \Eth-theory and assembly maps the fact that if \(D_0 \subseteq D_1\) then \(\redalg[D_0]{\CB} \to \redalg[D_1]{\CB}\) is a canonical quotient of \cstar{}algebras.
This is a naturality notion that is somehow intrinsic to inverse semigroup actions.

\begin{corollary} \label{cor:diagram-eth}
  Let \(\CB\) be a Fell bundle over an \((S,X)\)-algebra \(B\) and let \(D_0 \subseteq X\) be invariant.
  The following diagram commutes:
  \begin{equation*}
    \begin{tikzcd}[scale=50em]
      \Etop{S,\emptyset}{\CB} \arrow{r}{\mu_{\emptyset, \CB}} &[2.2em] \Kgrp{\redalg{\CB}} \arrow{d}{} \\
      \Etop{S,D_0}{\CB} \arrow{u}{} \arrow{r}{\mu_{D_0, \CB}} &[2.2em] \Kgrp{\redalg[D_0]{\CB}}.
    \end{tikzcd}
  \end{equation*}
\end{corollary}

\begin{remark}
  The contravariance of the arrow \(\Etop{S,D_0}{\CB} \to \Etop{S, \emptyset}{\CB}\) is due to the fact that \(D_0\)-equivariant asymptotic morphisms are, in particular, asymptotic morphisms, paired with the fact that if \(v_1 \sim_{\emptyset} v_2\) then \(v_1 \sim_{D_0} v_2\) for all \(v_1,v_2 \in \Egrp[S,\emptyset]{\cfunz{Y}}{\CB}\).
  Also recall that proper spaces are \(D_0\)-proper too.
\end{remark}

We end the section with some examples.
\begin{example}
  If \(S = \Gamma\) is a discrete group then the assembly map in \cref{thm:assembly} generalizes the usual one, see \cites{higson-lafforgue-skand-2002-counter-bc,higson-kasparov-2000,Guentner-Higson-Trout,valette-2002-book}, among many others.
  Likewise, if \(S \actson X\) is a closed action (so that the groupoid of germs is Hausdorff) then our \Eth-theoretic approach to the assembly map should be compared with the \KKth-theoretic one in \cite{tu-baum-connes-1999}, for instance.
\end{example}

\begin{example}
  On the other end of the spectrum, let \(S = \opensets{X}\) be as in \cref{ex:opensets-x}.
  As discussed before any Fell bundle \(\CB = (B_s)_{s \in S}\) over a nuclear \cstar{}algebra \(B\) will satisfy that \(\maxalg[D_0]{\CB} \cong \redalg[D_0]{\CB}\), and \(\maxalg{\CB} = B\).
  An \(S\)-space \(Y\) is just a space and a choice of open sets \(\CU_s \subseteq Y\), one for each \(s \in S = \opensets{X}\), that are neatly nested.
  Equivariant asymptotic morphisms \(\czeror \tensor \cfunz{Y} \tensor \CK_{S,D_0} \asymparrow \czeror \tensor \CB \tensor \CK_{S,D_0}\) are then simply usual asymptotic morphisms \(\varphi \colon \czeror \tensor \cfunz{Y} \tensor \CK \asymparrow \czeror \tensor B \tensor \CK\) with the extra condition that if \(f \in \cfunz{Y}\) is supported on \(\CU_s\) then \(\varphi(\xi \tensor f \tensor k) \in \czeror \tensor B_{s} \tensor \CK\).
  It is apparent that one may then suppose that \(X = Y\), and the group \(\Etop{S}{\CB}\) is just \(\Egrp[S]{\cfunz{X}}{\CB}\).
  This group should be compared with \(\KKgrp{X; \cfunz{X}}{B}\) in the sense of Gabe \cite{gabe-memoirs-2024}.
\end{example}

\section{Applications to groupoids, (weak) Cartan subalgebras, and representations of Fell bundles over groups} \label{sec:apps}

In this last section we provide a series of examples of \cstar{}algebras belonging to the class \(\maxalg[D_0]{\variable}\) and \(\redalg[D_0]{\variable}\) introduced previously.
Moreover, we discuss when the Fell bundles appearing in these examples are proper and/or amenable, and relate our equivariant \Eth-theory to previous versions thereof.
This is done in a purely expository manner, so we will keep the proofs brief.
Moreover, some of the results below have been discussed in the examples/remarks throughout the paper.

\subsection{Actions of groups}
We start by the simplest class, namely Fell bundles over discrete groups.
Discrete groups are, of course, canonical examples of inverse semigroups.
Moreover, a discrete group \(\Gamma\) is countable if and only if \(\Gamma\) is quasi-countable as an inverse semigroup, for the only idempotent \(e \in \Gamma\) is the unit element of the group.
Given an action \(\Gamma \actson B\) of a discrete group on a type I \cstar{}algebra \(B\), the procedure laid down in \eqref{eq:fell-bundle-from-action} forms a canonical Fell bundle \(\CB = (B_\gamma)_{\gamma \in \Gamma}\).
In fact, it is not hard to prove that
\[
  \Psi \colon \maxcrossedprod{B}{\Gamma} \to \redalg{\CB_\Gamma}, \;\; b u_\gamma \mapsto b_\gamma \in B_\gamma \cong B,
\]
where \(b_\gamma\) is the element \(b\) seen in the fiber \(\gamma\), is a \Star{}isomorphism.
It is also well known that if \(\maxcrossedprod{B}{\Gamma}\) is separable then \(\Gamma\) has to be countable and \(B\) has to be separable.
If the Fell bundle \(\CB_\Gamma\) has the approximation property then there is only one \cstar{}algebra, see \cite{buss-martinez-approx-prop-23}*{Section 6}.
In this setting, the equivariant \Eth-theory introduced in \cref{sec:equiv-eth} boils down to the usual equivariant \Eth-theory for actions of groups on \cstar{}algebras, see \cite{Guentner-Higson-Trout}.

\subsection{Non-commutative Cartan subalgebras}
A more interesting class of examples is that of type I non-commutative Cartan subalgebras.
These were introduced by Exel's \cite{exel-nc-cartan-subalgs-2011}, extending Renault's notion of Cartan subalgebra in \cite{Renault2008CartanSI} in order to admit possibly non-commutative subalgebras as well.
Noting that Exel's \cite{exel-nc-cartan-subalgs-2011} uses the nomenclature of \emph{generalized Cartan subalgebra}, which we have changed to \emph{non-commutative} to emphasize its true nature, the main theorem of \cite{exel-nc-cartan-subalgs-2011} reads as follows.

\begin{theorem}[see \cite{exel-nc-cartan-subalgs-2011}*{Theorem 14.7}] \label{thm-exel-nc-cartans}
  Let \(A\) be a separable \cstar{}algebra with a non-commutative Cartan subalgebra \(B\).
  Then there is a Fell bundle \(\CB = (B_s)_{s \in S}\) over \(B\), where \(S\) is some countable inverse semigroup, such that \(\redalg{\CB} \cong A\).
\end{theorem}

Note that it immediately follows from \cref{prop:setting-actions} that \(A\) is an \((S,\primidealspc{B})\)-algebra.
Moreover, \cref{prop:left-reg-rep} yields a canonical induced representation of \(A\) arising from a given representation of \(B\).
Likewise, \cref{thm:assembly} gives a way to compute the \Kth-theory of \(A\) via the assembly map.
In particular, note that there is not much to say about previous equivariant \Eth-theories in this case, for the Fell bundle case had not been studied previously.
If \(B\) is commutative, however, then we are in the case of the following subsection.

\subsection{\'Etale groupoids}
Fix some locally compact, \'etale groupoid \(G\) with Hausdorff unit space \(X \coloneqq G^{(0)}\) (for an introduction to these see \cite{buss-martinez-ess-ame-2024} and references therein).
Let \(S \coloneqq {\rm Bis}(G)\) be the inverse monoid of open bisections of \(G\).
It is well known that \(S\) canonically acts on \(X\) via multiplication of subsets of \(G\):
\[
  s \colon s^*s = \left\{g \in G : g^{-1}g \in s^*s\right\} \to ss^*, \;\; x \mapsto sxs^*,
\]
where \(AB \coloneqq \{ab : a \in A, b \in B\}\) for all \(A, B \subseteq G\), and \(A x \coloneqq A \{x\}\) and \(xB \coloneqq \{x\} B\).

\begin{remark}
  The notation \(sx\) now means a different thing in here than in the rest of the paper, for \(sx\) was always a unit elsewhere, where now it might be an arrow.
\end{remark}

It was first observed by Exel \cite{Exel:Inverse_combinatorial} that the groupoid of germs of the action \(S \actson X\) is isomorphic (as a groupoid) to \(G\), and the isomorphism is given by
\[
  S \ltimes X \mapsto G, \;\; \left[s, x\right] \mapsto sx,
\]
again understanding \(sx\) in the above sense, or equivalently as the necessarily unique arrow in \(s\) whose source is \(x\).
We may now prove that this isomorphism does not only happen at the level of the groupoids.
As usual, we denote by \(\cfunz{X}\) the Fell bundle given by \eqref{eq:fell-bundle-from-action}.

\begin{lemma} \label{prop:grpd-alg-fell-bundle}
  Let \(G\) be given, and let \(S \actson X\) be as above. Then the map
  \[
    \maxalg{\cfunz{X}} \to \maxalg{G}, \;\; f v_s \mapsto \left[G \supseteq s \ni sx \mapsto f\left(sxs^*\right) \in \CC\right],
  \]
  where \(f \in \cfunz{ss}^*\), is a \Star{}isomorphism.
  In fact, it also descends to an isomorphism \(\redalg{\cfunz{X}} \cong \redalg{G}\).
  Likewise, if \(D_0 \coloneqq D \subseteq X\) denotes the set of dangerous units in \(G\) then the above map also induces an isomorphisms \(\maxalg[D]{\cfunz{X}} \cong \essmaxalg{G}\) and \(\redalg[D]{\cfunz{X}} \cong \essalg{G}\).
\end{lemma}
\begin{proof}
  The proof of the isomorphism at the level of \(\maxalg{\variable}\) is well known, see \cite{armstrong-etal-2026}*{Theorem 8.1} for a quick reference.
  Note that in fact \cite{armstrong-etal-2026} also proves the isomorphism descends to \(\redalg{\variable}\) and \(\essalg{\variable}\).
  We also refer the reader to \cref{thm:ess-alg-is-jd} for an alternative (more general) proof, should they want it.
\end{proof}

The following relates properness in our sense and that of Tu~\cite{tu-baum-connes-2004}.
\begin{lemma} \label{lemma:proper-iff-proper}
  Let \(S \actson X\) be an action on a locally compact, Hausdorff space \(X\).
  The transformation groupoid \(S \ltimes X\) is proper in the sense of \cite{tu-baum-connes-2004}*{Definition 2.9} if and only if \(S \actson X\) is proper in the sense of \cref{def:proper-grpd}.
\end{lemma}
\begin{proof}
  Recall that \cite{tu-baum-connes-2004}*{Definition 2.9} defines an \'etale groupoid to be proper if \((\range,\source) \colon G \to G^{(0)} \times G^{(0)}\) is a proper map between topological spaces, \emph{i.e.}\ the preimage of any compact set is contained in a compact, where \(\range\) and \(\source\) are the range and source maps of \(G\) respectively.

  Assume first that \(S \actson X\) is proper. Any compact subset of \(X \times X\) is contained in some rectangle \(K_2 \times K_1\) for some compact sets \(K_1, K_2 \subseteq X\).
  Let \(\{t_1, \dots, t_\ell\} \subseteq S\) be a finite cover of \(\{s \in S : s K_1 \cap K_2 \neq 0\}\), which exists by properness.
  If \(g = [s,x] \in S \ltimes X\) is an arrow in the preimage of \(K_2 \times K_1\) then \(g = [s,x] = [t_i,x]\) for some \(i = \{1, \dots, \ell\}\).
  In particular it follows that the pre-image of \(K_2 \times K_1\) is contained in \(t_1 K_1 \cup \dots \cup t_\ell K_1\), which is indeed contained in a compact set.
  The opposite direction is similar.
\end{proof}

The notion of amenability introduced in \cref{def:approx-property} precisely yields the notions of amenability and essential amenability studied in \cite{buss-martinez-ess-ame-2024}.
\begin{lemma} \label{lemma:d-0-amen-grpd}
  Let \(D_0 \subseteq X\) be either the empty set of the set of dangerous units of \(G\), \emph{i.e.}\ \(D_0 = D\) as in \eqref{eq:d-s}.
  If \(\redalg[D_0]{G}\) is nuclear then \(S \actson X\) is \(D_0\)-amenable.
\end{lemma}
\begin{proof}
  This is exactly implications (i) implies (ii) in \cite{buss-martinez-ess-ame-2024}*{Theorems 5.11 and 5.12} respectively.
\end{proof}

As a consequence of the previous discussion, the class of algebras \(\maxalg[D_0]{\variable}\) and \(\redalg[D_0]{\variable}\) immediately includes all \'etale, in general non-Hausdorff, twisted groupoid \cstar{}algebras.
Indeed, we may even look at \cstar{}algebras with \emph{weak Cartan subalgebras}.
We will not define this notion here, but we will only mention the following, which was proven in \cite{exel-pitts-grpds-2022}*{Theorem 3.9.4}, being inspired by Renault and Kumjian's seminal papers \cites{renault-2013,kumjian-diagonals-1986}. 

\begin{theorem}[see \cite{exel-pitts-grpds-2022}*{Theorem 3.9.4}] \label{thm-weak-cartan-grpd-model}
  Let \(A\) be a separable \cstar{}algebra with a weak Cartan subalgebra \(D \subseteq A\).
  There is a locally compact, secound countable, twisted, \'etale, topologically free groupoid \((G, \Sigma)\) such that \(A \cong \essalg{G, \Sigma}\) and \(D \cong \cfunz{G^{(0)}}\).
\end{theorem}

\begin{example}
  Given an \'etale (non-Hausdorff) groupoid \(G\) and a twist \(\Sigma \to G\) over it, it is well known that one may construct a Fell line bundle \(\CB_\Sigma\) over, say, the inverse semigroup of open bisections of \(G\).
  In this scenario, one has that \(\maxalg{\CB_\Sigma} = \maxalg{G, \Sigma}\), \(\redalg{\CB_\Sigma} = \redalg{G, \Sigma}\) and even \(\essalg{\CB_\sigma} = \essalg{G, \Sigma}\).
  Thus, by the previous discussion \(\maxalg[D_0]{\variable}\) and \(\redalg[D_0]{\variable}\) does include all \cstar{}algebras with weak Cartan subalgebras.
\end{example}

The equivariant \Eth-theory of \cref{sec:equiv-eth} in this case reduces to equivariant groupoid \Eth-theory, and was studied (in the Hausdorff case, and possibly untwisted) for instance in \cites{popescu-2004,kwanieski-li-skalski-2022,burgstaller-2016}.

We have so far only studied examples where \(D_0 \subseteq X\) is ``natural'', that is, whenever \(D_0\) is either empty or equal to the set of dangerous units \(D\) as in \eqref{eq:d-s}.
Another class of examples is also of interest, however: that in which \(D_0\) is an orbit.
For the following, given a unit \(x \in G^{(0)}\) we let \(\lambda_x \colon \algalg{G} \to \CB(\ell^2(Gx))\) be the usual left regular representation at the unit \(x\).

\begin{proposition}
  Let \(G\) be an \'etale groupoid, and fix some unit \(x_0 \in G^{(0)} \eqqcolon X\). Fix some wide inverse semigroup \(S\) of open bisections of \(G\), and let \(Sx_0 \subseteq X\) be the orbit of \(x_0\).
  The \cstar{}algebra generated by \(\lambda_x(\algalg{G})\) is isomorphic to \(\redalg[X \setminus S x_0]{\cfunz{X}_S}\).
\end{proposition}
\begin{proof}
  The notation is rather ungrateful, so we denote \(\redalg[X \setminus S x_0]{\cfunz{X}_S}\) simply by \(\redcrossedprod[Sx_0]{\cfunz{X}}{S}\).
  \cref{prop:left-reg-rep} and the subsequent remark, applied to the Fell bundle \(\cfunz{X}_S\) in \eqref{eq:fell-bundle-from-action} and \(D_0 \coloneqq X \setminus Sx_0\), yields a \emph{faithful} representation \(\redcrossedprod[Sx_0]{\cfunz{X}}{S} \to \CB(\ell^2_{x_0}(S; X))\).
  Our first goal is to describe \(\ell^2_{x_0}(S; X)\) explicitly.
  For this, recall \cref{rem:more-left-reg}: \(\ell^2_{D_0}(S; \CB)\) can be thought as the Hausdorff completion of \(\oplus_{s \in S} \ell^\infty(ss^*)\), where \(\ell^\infty(ss^*)\) is the \cstar{}algebra of complex valued bounded functions \(ss^* \to \CC\).
  In particular it is a \(\ell^\infty(X \setminus D_0)\)-Hilbert bimodule.

  Given the usual diagonal (in general non-isometric) representation \(\rho \colon \cfunz{X} \to \CB(\ell^2(X \setminus D_0))\), consider the balanced tensor product \(\CH \coloneqq \ell^2_{D_0}(S; \CB) \tensor_{\rho^{**}} \ell^2(X \setminus D_0)\), that is, the Hausdorff completion of \(\ell^2_{D_0}(S; \CB) \odot \ell^2(X \setminus D_0)\) subject to the semi-inner product
  \[
    \innerprod{v_1 \tensor \xi_1}{v_2 \tensor \xi_2} \coloneqq \innerprod{\xi_1}{ \innerprod{v_1}{v_2}_{\ell^2_{D_0}\left(S; \CB\right)} \xi_2}_{\ell^2\left(X \setminus D_0\right)}.
  \]
  The same proof as in \cite{buss-martinez-approx-prop-23}*{Section 3} shows that the left regular representation
  \[
    \Lambda_{D_0} \tensor {\rm id} \colon \redcrossedprod[Sx_0]{\cfunz{X}}{S} \to \CL\left(\ell^2_{x_0}\left(S; X\right) \otimes_{\rho^{**}} \ell^2\left(X \setminus D_0\right)\right)
  \]
  is unitarily equivalent to the representation
  \begin{align*}
    \rho^\Lambda_s \colon B_s & \to \CB\left(\ell^2_{\rho,D_0}\left(S; \ell^2\left(X\right)\right)\right), \\
    \rho^\Lambda_s \left(b_s\right) \left(v_t \delta_t\right) & \coloneqq \rho\left(b_s\right) v_t \delta_{st},
  \end{align*}
  where on the right hand side we see \(b_s \in B_s \subseteq \ell^{\infty}(ss^*) \subseteq \ell^\infty(X)\) by extending by \(0\) outside of \(ss^*\).
  The Hilbert space \(\ell^2_{\rho,D_0}(S; \ell^2(X \setminus D_0))\) can be checked to be precisely \(\ell^2(S \ltimes (X \setminus D_0))\), that is, the \cstar{}algebra \(\redcrossedprod[Sx_0]{\cfunz{X}}{S}\) is precisely the \cstar{}algebra generated by the representation
  \[
    \oplus_{x \in Sx_0} \lambda_x \colon \algalg{S \ltimes X} \cong \algalg{G} \to \CB\left(\oplus_{x \in Sx_0} \ell^2(G_x)\right),
  \]
  which can, in turn, be checked to be \Star{}isomorphic to \(\lambda_x(\maxalg{G})\) as in the statement.
\end{proof}

\vspace{1.5mm}

We end the paper with a, by now, well known observation and an intriguing question.
For the first, note that the essential \cstar{}algebra \(\essalg{\CB}\) of a Fell bundle \(\CB = (B_s)_{s \in S}\) only makes sense if \(S\) is ``small enough'', \emph{i.e.}\ it is quasi-countable.
In that case the set \(D\) above realizing the essential algebra is meager and Borel (by countability and \cref{lemma:d-inv-countable}), so we ``are not killing too much''.
On the other hand if \(S\) is not quasi-countable then \(D\) may as well be all of \(\primidealspc{B}\), so \(\essalg{\CB}\) may trivialize.
In the same vein so may \(\redalg[D_0]{\CB}\) trivialize if \(D_0 \subseteq X\) is large.
In the setting of groupoids, \(S\) is quasi-countable when \(S \ltimes X\) is covered by countably many open bisections, a condition that has appeared already in previous work, see \cites{Kwasniewski-Meyer:Pure_infiniteness,KwasniewskiMeyer2021,buss-martinez-ess-ame-2024}, among others.

\vspace{1.5mm}

For the intriguing question we pose the following, to which we have no answer so far.

\begin{question}
  Is there any separable, nuclear \cstar{}algebra \(A\) that does \emph{not} have a model \(A \cong \maxalg[D_0]{\CB}\) or \(A \cong \redalg[D_0]{\CB}\) for some Fell bundle \(\CB = (B_s)_{s \in S}\) over a type I \cstar{}algebra \(B\)?
\end{question}

The only ``concrete'' algebras the author is aware of that could potentially fall outside of the class \(\maxalg[D_0]{\variable}\) and \(\redalg[D_0]{\variable}\) introduced in the paper is the class of quotients of \(\maxalg{\Gamma}\) for an amenable discrete group \(\Gamma\) (and variants thereof).
Nevertheless, as pointed out in \cite{eckhardt-2024}, see \cite{eckhardt-2024}*{Remark 2.1 and Theorem 2.6}, every quotient of \(\maxalg{\Gamma}\) for a virtually nilpotent group \(\Gamma\) automatically is of the form \(\maxalg[D_0]{\variable}\).
Hence one would need to discuss other larger classes of discrete groups.
A certainly interesting case would, for instance, be the class of simple quotients of \(\maxalg{{\rm F}(\ZZ \actson X)}\), where \({\rm F}(\ZZ \actson X)\) is the topological full group of a free minimal action of \(\ZZ\) on the Cantor space.
Another interesting, perhaps more approachable, case would be the Baumslag-Solitar group \({\rm BS}(1,2) = \{a, b : bab^{-1} = a^2\}\), which is amenable and of exponential growth (and hence not under the study of \cite{eckhardt-2024}).

\bibliography{bibuctbc}

\end{document}